\documentclass[12pt]{article}%
\usepackage{tikz}
\usepackage{enumitem}
\usepackage{graphicx}
\usepackage{float}
\usepackage{subcaption}
\usepackage{amsmath}
\usepackage{amsfonts}
\usepackage{amssymb}%
\providecommand{\U}[1]{\protect\rule{.1in}{.1in}}
\usetikzlibrary{calc}
\usetikzlibrary{intersections}
\newtheorem{theorem}{Theorem}[section]

\newtheorem{definition}[theorem]{Definition}
\newtheorem{example}[theorem]{Example}

\newtheorem{lemma}[theorem]{Lemma}

\newtheorem{problem}[theorem]{Problem}
\newtheorem{proposition}[theorem]{Proposition}
\newtheorem{remark}[theorem]{Remark}

\newenvironment{proof}[1][Proof]{\textbf{#1.} }{\ \rule{0.5em}{0.5em}}
\begin{document}

\author{Lawrence Reeves\\School of Mathematics and Statistics\\University of Melbourne\\Victoria 3010, Australia.\\email:lreeves@unimelb.edu.au
\and Peter Scott\\Mathematics Department\\University of Michigan\\Ann Arbor, Michigan 48109, USA.\\email:pscott@umich.edu
\and Gadde A.~Swarup\\718 High Street Road\\Glen Waverley\\Victoria 3150, Australia.\\email: anandaswarupg@gmail.com }
\title{Comparing decompositions of Poincar\'{e} duality pairs}
\maketitle

\begin{abstract}
Analogues of JSJ decompositions were developed for Poincar\'{e} duality pairs
in \cite{SS04}. These decompositions depend only on the group. Our focus will
be on describing the edge splittings of these decompositions more precisely.
We use our results to compare these decompositions with two other closely
related decompositions.

\end{abstract}
\date{}

\section{Introduction}

In this paper, we consider algebraic analogues of previous work in the
topology of $3$--manifolds related to the JSJ\ decomposition introduced by
Jaco and Shalen \cite{JacoShalen} and Johannson \cite{Johannson}. In
\cite{JacoShalen} and \cite{Johannson}, the authors considered a compact
orientable Haken $3$--manifold $M$ with incompressible boundary, and
constructed the characteristic submanifold $V(M)$ as a maximal Seifert pair
embedded in $M$. The frontier of $V(M)$ is a family of disjoint essential
annuli and tori in $M$, which decompose $M$ into pieces either in $V(M)$ or
its complement. In \cite{NS}, \cite{Sc1}, \cite{Sc2}, \cite{SS02} and
\cite{SS05}, the emphasis turned to annuli and tori rather than Seifert pairs.
In \cite{NS}, the authors gave a new approach to constructing this
decomposition of $M$ in which embedded essential annuli and tori were the main
subject of interest. They defined an embedded essential annulus or torus in
$M$ to be \textit{canonical} if it can be isotoped to be disjoint from any
other embedded essential annulus or torus in $M$. This led to a finer
decomposition of $M$, which they called the Waldhausen decomposition
(W-decomposition), from which the JSJ decomposition can be obtained in a
natural way. In \cite{SS02}, the interest was in possibly singular essential
annuli and tori in $M$. The authors defined an embedded essential annulus or
torus in $M$ to be \textit{topologically canonical} if it has intersection
number zero with any (possibly singular) essential annulus or torus in
$(M,\partial M)$. Most canonical annuli and tori in a $3$--manifold are also
topologically canonical, and the exceptions can be precisely described. The
existence of these exceptions explains why the W-decomposition of $M$ is in
general finer than the JSJ decomposition. In \cite{SS02}, the authors also
defined an algebraic analogue in which a splitting of $\pi_{1}(M)$ given by an
essential embedded annulus or torus in $M$ is \textit{algebraically canonical}
if it has intersection number zero with any almost invariant subset of
$\pi_{1}(M)$ which is over $\mathbb{Z}$ or $\mathbb{Z}\times\mathbb{Z}$. (See
\cite{SS01} for a discussion of the idea of intersection numbers of almost
invariant sets.) They showed that topologically canonical splittings are not
quite the same as algebraically canonical ones, and gave some examples to
demonstrate this. If $M$ has empty boundary, there is no difference.

In \cite{SS05}, (which is a revised version of \cite{SS04}), an analogue of
the JSJ decomposition of $3$--manifolds was developed for orientable $PD(n+2)$
pairs, with $n\geq1$. The decomposition for a $PD(n+2)$ pair $(G,\partial G)$
is simply the decomposition $\Gamma_{n,n+1}(G)$ of \cite{SS03}, which is
defined for many almost finitely presented groups $G$. (If $\partial G$ is
empty, so that $G$ is a $PD(n+2)$ group, then $\Gamma_{n,n+1}(G)$ is equal to
the decomposition $\Gamma_{n+1}(G)$.) See \cite{SS03errata} for corrections to
\cite{SS03}. Thus this decomposition depends only on the group $G$ and not on
$\partial G$. In the case when $\partial G$ is empty, Kropholler
\cite{Kropholler} had earlier obtained this decomposition. In \cite{SS03}, the
decomposition $\Gamma_{n,n+1}(G)$ was constructed as the regular neighbourhood
of the family consisting of all almost invariant (a.i.) subsets of $G$ over
$VPCn$ subgroups together with all a.i. subsets of $G$ over $VPC(n+1)$
subgroups which do not cross any a.i. subset of $G$ over a $VPCn$ subgroup. (A
group is $VPCn$ if it is virtually polycyclic of Hirsch length $n$.) This
regular neighbourhood is reduced, meaning that two adjacent vertices cannot
both be isolated, except when the graph is a loop with just these two
vertices. (A vertex $w$ of a $G$--tree $T$ is called \emph{isolated} if it has
valence $2$, and these two edges have the same stabilizer as $w$. The image of
$w$ in $G\backslash T$ is also called \emph{isolated}.) In general, the edge
groups of $\Gamma_{n,n+1}$ may not even be finitely generated, but it is shown
in \cite{SS05} that, in the case of $PD(n+2)$ pairs, the edge groups are all
either $VPCn$ or $VPC(n+1)$.

It was shown in \cite{SS05} that if $G$ is the fundamental group of a compact
orientable Haken $3$--manifold $M$ with incompressible boundary, then
$\Gamma_{1,2}(G)$ differs from the JSJ decomposition of $M$ only in some small
Seifert pieces which have no crossing annuli or tori. One can easily move from
$\Gamma_{1,2}(G)$ to its completion $\Gamma_{1,2}^{c}(G)$, and this completion
corresponds to the JSJ decomposition of $M$. The difference between
$\Gamma_{1,2}(G)$ and its completion $\Gamma_{1,2}^{c}(G)$ is related to
special properties of small Seifert fibre spaces. If the boundary of $M$ is
empty, then similar comments apply to $\Gamma_{2}(G)$ and its completion
$\Gamma_{2}^{c}(G)$. In the general case of a $PD(n+2)$ pair $(G,\partial G)$
we denote the completion of $\Gamma_{n,n+1}(G)$ by $\Gamma_{n,n+1}^{c}(G)$ or
just $\Gamma_{n,n+1}^{c}$ when the group $G$ is clear from the context. The
completion $\Gamma_{n,n+1}^{c}(G)$ is essentially obtained from $\Gamma
_{n,n+1}(G)$ by re-labelling some $V_{1}$--vertices as $V_{0}$--vertices and
then adding isolated $V_{1}$--vertices as needed to keep the graph bipartite.
In particular, the edge splittings of $\Gamma_{n,n+1}$\ are the same as the
edge splittings of $\Gamma_{n,n+1}^{c}$. Again the completed decomposition
depends only on the group $G$ and not on $\partial G$.

Our focus in this paper will be on the edge splittings of these decompositions
of a group $G$. This is closely related to the approach in \cite{SS02} in the
case of $3$--manifolds. Our main result, Theorem \ref{mainresult}, is similar
to that in \cite{SS02} but is in the setting of the decomposition
$\Gamma_{n,n+1}(G)$ of a $PD(n+2)$ pair $(G,\partial G)$. Our result is more
detailed than that in \cite{SS02}, and gives a precise description of the
special cases which arise. These results are new even in the setting of
$3$--manifolds, and they yield a substantial refinement of the results in
\cite{SS02}.

There are other natural approaches to finding JSJ\ decompositions of a
$PD(n+2)$ pair $(G,\partial G)$. The analogue of \cite{NS} would be to
consider a maximal family of splittings of $G$ by annuli and tori which cross
no other such splitting. (In this paper, we use the word "cross" to mean "has
non-zero intersection number with".) Another approach would be simply to
consider the regular neighbourhood of the family of all almost invariant
subsets of $G$ which are over $VPCn$ or $VPC(n+1)$ subgroups. For general
groups, neither of these decompositions need exist. However, in this paper, we
use the results of \cite{SS05} and of Theorem \ref{mainresult} to show that
both decompositions exist in the setting of Poincar\'{e} duality pairs, and we
compare these three different decompositions. The differences between them
leads to a detailed study of various small $2$--dimensional orbifolds and
fibrations over them by $VPCn$ groups. We think that this clarifies how these
various natural decomposition come about. We also discuss the special case of
$PD3$ pairs where the descriptions are somewhat simpler. This seems to make an
analogue of Johannson's Deformation Theorem possible for $PD3$ pairs, and also
seems relevant to some questions raised by Wall in sections 6 and 10 of
\cite{Wall2}.

In section \ref{section:prelim}, we describe the notions of annuli, tori, and
their associated almost invariant sets as in \cite{SS05}. We will also recall
some constructions and results from \cite{SS05} which we will use. In section
\ref{section:examples}, we describe examples generalizing those given in
\cite{SS02}. We also completely characterize all such examples. In section
\ref{section:mainresult}, we prove our main result, Theorem \ref{mainresult}.
In section \ref{section:comparisons}, we compare three different versions of
JSJ decompositions of $PD$--pairs. In section \ref{section:dimension3}, we
discuss how our results become substantially simpler when applied to the case
of $PD3$ pairs, and we compare our results with those in \cite{NS}. In section
\ref{section:relatedquestions}, we discuss some related questions.

We will also use earlier definitions and results from \cite{BE01},
\cite{BE02}, \cite{Brown} and \cite{Kropholler}. There are two survey articles
from around 2000, \cite{Davis} and \cite{Wall2}, which contain a number of
problems related to $PD$ groups and pairs.

\section{Preliminaries\label{section:prelim}}

We will consider orientable $PD(n+2)$ pairs $(G,\partial G)$ and first
describe, following \cite{SS05}, annuli and tori in $(G,\partial G)$ and their
associated almost invariant sets.

Let $H$ be a $VPC(n+1)$ subgroup of $G$. Note that as $G$ is a $PD\ $group, it
is torsion free. Hence $H$ is also torsion free, and so is a $PD(n+1)$ group.
The double $DG$ of $G$ is an orientable $PD(n+2)$ group, and so the pair
$(DG,H)$ has two ends if $H$ is orientable, and only one end otherwise. In the
first case, $DG$ contains two complementary nontrivial $H$--almost invariant
subsets $X$ and $X^{\ast}$, and any nontrivial $H$--almost invariant subset of
$DG$ is equivalent to one of these. Let $Y$ denote the intersection $X\cap G$.
Thus $Y$ and its complement $Y^{\ast}$ in $G$ are $H$--almost invariant
subsets of $G$. Further they are nontrivial unless $H$ is peripheral in
$(G,\partial G)$, i.e. $H$ is conjugate into a group in $\partial G$. We say
that $Y$ is \textit{dual to} $H$. If $H$ is an orientable $VPC(n+1)$ subgroup
of $G$, we call it a torus in $G$. Note that the $H$--almost invariant set $Y$
dual to $H$ is automatically adapted to $\partial G$. Conversely, suppose that
$H$ is a $VPC(n+1)$ subgroup of $G$, and $Y$ is a nontrivial $H$--almost
invariant subset of $G$ which is adapted to $\partial G$. Then $Y$ extends to
a nontrivial $H$--almost invariant subset of $DG$. It follows that $H$ must be
orientable and hence a torus in $G$.

The case of annuli requires more work. An annulus in a $PD(n+2)$ pair is a
certain type of orientable $PD(n+1)$ pair. We need to consider two types of
annulus. One type is $\Lambda_{H}=(H,\{H,H\})$, where $H$ is an orientable
$PDn$ group which is also $VPCn$. We call this an untwisted annulus. The other
type is $\Lambda_{H}=(H,H_{0})$, where $H$ is a non-orientable $PDn$ group
which is $VPCn$, and $H_{0}$ is the orientation subgroup of $H$. We call this
a twisted annulus. Corresponding to these, we have $K(\pi,1)$ spaces which we
denote by $(A,\partial A)$. Similarly, we denote the $K(\pi,1)$ pair
corresponding to $(G,\partial G$) by $(M,\partial M)$. Note that when $n=1$,
the only $PD1$ group is $\mathbb{Z}$, and this is orientable. Thus twisted
annuli do not appear in the theory of $3$--manifolds. For simplicity, we
assume that $G$ is finitely presented so that we can identify certain
cohomology groups of $G$ with the cohomology groups with compact supports of
various covers of $M$. In the general case, when $G$ is almost finitely
presented, we have to take the appropriate `finitely supported' cohomology
groups of covers of $M$.

An annulus in $(G,\partial G)$ is an injective homomorphism of group pairs
$\Theta:\Lambda_{H}\rightarrow(G,\partial G)$. This means that $\Theta$ maps
$H$ to $G$ and also maps each group in $\partial\Lambda_{H}$ to a conjugate of
some group in $\partial G$. Such $\Theta$ induces a continuous map
$\theta:(A,\partial A)\rightarrow(M,\partial M)$. Note that in the untwisted
case, such a map $\theta$ is determined up to homotopy by choosing a copy of
$H$ in two conjugates of groups in $\partial G$, such that the two copies of
$H$ are conjugate in $G$. And in the twisted case, $\theta$ is determined up
to homotopy by choosing a copy of $H$ in $G$ and a conjugate of some group in
$\partial G$ such that the intersection of $H$ with this conjugate contains
$H_{0}$. Thus an annulus can be thought of purely algebraically. We call the
annulus `essential' if $\theta$ cannot be homotoped relative to $\partial A$
into $\partial M$. It is clear that the essentiality of an annulus is also a
purely algebraic property. An untwisted annulus is essential if and only if
the images of the two boundary groups are not conjugate in a group in
$\partial G$. And a twisted annulus is essential if and only if $H_{0}$ lies
in a boundary group $K$ in $\partial G$, and $H\cap K=H_{0}$.

We next show how to associate an almost invariant set to an essential annulus.
Consider the lift $\theta_{H}:(A,\partial A)\rightarrow(M_{H},\partial M_{H})$
to the cover $M_{H}$ of $M$ with fundamental group $H$. Let $S$ be the
component of $\partial M_{H}$ containing $\theta_{H}(\partial_{0}A)$, where
$\partial_{0}A$ is one specified component of $\partial A$ in the untwisted
case and is $\partial A$ in the twisted case.

In the untwisted case, since $\theta_{H}(A)$ cannot be homotoped rel $\partial
A$ into $\partial M$, the other component $\partial_{1}A$ must be mapped by
$\theta_{H}$ into some other component $T$ of $\partial M_{H}$. Thus both $S$
and $T$ have fundamental groups isomorphic to $H$ and the images of the
fundamental cycle of $H$ generate $H_{n}(S)$ and $H_{n}(T)$, both with
$\mathbb{Z}$ and $\mathbb{Z}_{2}$ coefficients. Let $[A]$ denote the
fundamental cycle of $A$ in $H_{n+1}(A,\partial A)$. Then in the boundary map
\[
H_{n+1}(M_{H},\partial M_{H})\rightarrow H_{n}(\partial M_{H})\simeq
H_{n}(S)\oplus H_{n}(T)\oplus\cdots
\]
we see that the projection of the image of $[A]$ to each of the first two
direct summands of $H_{n}(\partial M_{H})$ is a generator.

In the twisted case, $\pi_{1}(S)$ must be isomorphic to $H_{0}$ since $\theta$
cannot be homotoped into $\partial M$ relative to $\partial A$. Thus the
projection of the image of $[A]$ to the summand $H_{n}(S)$ is a generator. So,
with $\mathbb{Z}$ or $\mathbb{Z}_{2}$ coefficients, we have $(\theta
_{H})_{\ast}([A])$ is non-zero in $H_{n+1}(M_{H},\partial M_{H})$. Also the
image with $\mathbb{Z}_{2}$--coefficients is the specialisation of the image
with $\mathbb{Z}$--coefficients. Denote these images by $\alpha$ and
$\bar{\alpha}$, and denote the duals of these images in $H_{c}^{1}%
(M_{H};\mathbb{Z})$ and $H_{c}^{1}(M_{H};\mathbb{Z}_{2})$ by $\beta$ and
$\bar{\beta}$. Again $\bar{\beta}$ is the specialisation of $\beta$.

Next, we want to relate $\bar{\beta}$ to the ends of the pair $(G,H)$. We have
that $H$ is of infinite index in $G$, so that $H_{f}^{0}(M_{H};\mathbb{Z}%
_{2})=0$ (finite cohomology is used as in section 7.4 of \cite{Brown}). Thus
$H_{f}^{1}(M_{H};\mathbb{Z}_{2})$ fits into the exact sequence:
\[
0\rightarrow\mathbb{Z}_{2}\rightarrow H_{e}^{0}(M_{H};\mathbb{Z}%
_{2})\rightarrow H_{f}^{1}(M_{H};\mathbb{Z}_{2})\overset{\overline
{r}}{\rightarrow}H^{1}(M_{H};\mathbb{Z}_{2}).
\]
Here $H_{e}^{0}$ is the $0$--the cohomology of the space of ends and the last
map is the restriction map from finite cohomology to ordinary cohomology.
Group theoretically, the above sequence identifies with the following. Let
$P[H\backslash G]$ denote the power set of right cosets $Hg$ of $H$ in $G$,
and let $E[H\backslash G]$ denote $P[H\backslash G]/\mathbb{Z}_{2}[H\backslash
G]$. The above sequence can be identified with
\[
0\rightarrow\mathbb{Z}_{2}\rightarrow H^{0}(G;E[H\backslash G])\overset{\delta
}{\rightarrow}H^{1}(G;\mathbb{Z}_{2}[H\backslash G])\overset{\overline
{r}}{\rightarrow}H^{1}(H;\mathbb{Z}_{2}).
\]
Thus $\bar{\beta}$ gives an element of $H^{1}(G;\mathbb{Z}_{2}[H\backslash
G])$ which we continue to denote by $\bar{\beta}$. In order to show that this
element gives a nontrivial almost invariant set, we need to know that it is
non-zero in $H^{1}(G;\mathbb{Z}_{2}[H\backslash G])$, and is in the kernel of
$\bar{r}$. We already know that $\bar{\beta}$ is non-zero since we started
with an essential annulus. Thus it remains to show that $\bar{r}(\bar{\beta
})=0$. Consider the following diagram:
\[%
\begin{array}
[c]{ccc}%
H^{1}(G;\mathbb{Z}[H\backslash G]) & \overset{r}{\rightarrow} & H^{1}%
(H;\mathbb{Z})\\
\downarrow\rho &  & \downarrow\rho\\
H^{1}(G;\mathbb{Z}_{2}[H\backslash G]) & \overset{\overline{r}}{\rightarrow} &
H^{1}(H;\mathbb{Z}_{2})
\end{array}
\]

In Theorem 2 of \cite{Swarup}, Swarup showed that $r$ is the zero map. Since
$\bar{\beta}=\rho(\beta)$ it follows that $\bar{r}(\bar{\beta})=0$, although
in general $\bar{r}$ is not the zero map. Thus we see that $\bar{\beta}%
=\delta(e)$ for some element $e$ of $H^{0}(G;E[H\backslash G])$. Since the
kernel of $\delta$ is just $\mathbb{Z}_{2}$, the element $e$ defines a
nontrivial $H$--almost invariant subset $Y$ of $G$ up to equivalence and
complementation. This completes our association of an almost invariant set $Y$
with an essential annulus $\theta$. We say that $Y$ is dual to $\theta$. It
turns out that given a nontrivial almost invariant subset $X$ of $G$ which is
over a $VPCn$\ group $H$, there is a subgroup $H^{\prime}$ of finite index in
$H$ such that $X$ is a finite sum of almost invariant sets over $H^{\prime}$
each dual to an annulus.

For future reference, we give some more terminology. As usual $(G,\partial G)$
is an orientable $PD(n+2)$ pair. If the almost invariant subset of $G$ dual to
an essential annulus or torus is associated to a splitting $\sigma$, we will
say that $\sigma$ is \textit{dual to} the same essential annulus or torus. If
$\Gamma$ is a graph of groups structure for $G$, nd $v$ is a vertex of
$\Gamma$, then an essential annulus or torus in $(G,\partial G)$ is enclosed
by $v$ if the dual almost invariant subset of $G$ is enclosed by $v$. Finally
if $\theta$ and $\phi$ are each an essential annulus or torus in $(G,\partial
G)$, we will say that $\theta$ and $\phi$ cross if the dual almost invariant
subsets of $G$ cross.

In \cite{SS03}, the authors considered an almost finitely presented group $G$
and an integer $n$ such that $G$ has no nontrivial almost invariant subsets
over $VPCk$ subgroups for $k<n$. Then it was shown that the family
$\mathcal{F}_{n,n+1}$ of all equivalence classes of almost invariant subsets
of $G$ over $VPCn$\ groups, and all $n$-canonical almost invariant subsets
over $VPC(n+1)$ groups has an algebraic regular neighbourhood, denoted
$\Gamma_{n,n+1}(G)$. In this setting a $H$--almost invariant subset of $G$ is
$n$--canonical if it does not cross any almost invariant subset over a
$VPCn$\ subgroup. If $(G,\partial G)$ is a $PD(n+2)$ pair, it was shown by
Kropholler and Roller (Lemma 4.3 of \cite{Kropholler}) that $G$ has no
nontrivial almost invariant subsets over $VPCk$ subgroups for $k<n$, so that
the decomposition $\Gamma_{n,n+1}(G)$ exists. In \cite{SS05}, the authors
showed that almost invariant subsets of $G$ over $VPC(n+1)$ subgroups which do
not cross any almost invariant subset over a $VPCn$ subgroup are automatically
adapted to $\partial G$. Further, if we enlarge the family $\mathcal{F}%
_{n,n+1}$ to include all almost invariant sets over $VPC(n+1)$ subgroups which
are adapted to $\partial G$, the new family $\mathcal{G}_{n,n+1}$ has the same
regular neighbourhood $\Gamma_{n,n+1}(G)$.

If $M$ is a compact orientable Haken $3$--manifold with incompressible (i.e.
$\pi_{1}$--injective) boundary, the characteristic submanifold $V(M)$ of $M$
is a compact submanifold whose frontier consists of incompressible annuli and
tori in $M$. This decomposition of $M$ into pieces is called the JSJ
decomposition. The components of $V(M)$ are Seifert fibre spaces or
$I$--bundles. Cutting $M\ $along the frontier of $V(M)$ yields a graph of
groups structure $\Gamma(M)$ for $G=\pi_{1}(M)$ whose edge groups are
isomorphic to $\mathbb{Z}$ or to $\mathbb{Z}\times\mathbb{Z}$. This graph is
bipartite as each component of the frontier of $V(M)$ lies in the boundary of
a component of $V(M)$ and a component of the complement. In \cite{SS05}, the
authors showed that if $(G,\partial G)$ is a $PD(n+2)$ pair, then
$\Gamma_{n,n+1}(G)$ has many properties in common with $\Gamma(M)$, with
$V_{0}$--vertices of $\Gamma_{n,n+1}(G)$ corresponding to the components of
$V(M)$. For the complete details the reader is referred to \cite{SS05}, but we
will need to recall some of the definitions for use in this paper.

In \cite{SS05}, an important part is played by groups which are $VPCn$%
--by--Fuchsian. Such a group has a $VPCn$ normal subgroup whose quotient is
the orbifold fundamental group of a compact $2$--orbifold. Further the
quotient is assumed not to be virtually cyclic. A $V_{0}$--vertex $v$ of
$\Gamma_{n,n+1}$ such that $G(v)$ is $VPCn$--by--Fuchsian corresponds to a
component $W$ of $V(M)$ which is a Seifert fibre space. Topologically there
are different cases depending on how $W$ meets $\partial M$, and extra
conditions are imposed on the edges of $\Gamma_{n,n+1}(G)$ which are incident
to $v$ to reflect this. This is what is meant by saying that $v$ is \textit{of
Seifert type}. If $v$ is a $V_{0}$--vertex of $\Gamma_{n,n+1}$ such that
$G(v)$ is $VPCn$--by--$\pi_{1}^{orb}(X)$, where $X$ is a $2$--orbifold with
virtually cyclic fundamental group, we say that $v$ is \textit{of solid torus
type} if $\pi_{1}^{orb}(X)$ is finite, and \textit{of torus type} otherwise.
The terminology reflects the type of the corresponding components of $V(M)$.
Again conditions need to be imposed on the edges of $\Gamma_{n,n+1}(G)$ which
are incident to $v$.

There are some other important special cases. In \cite{SS05}, the authors
defined a $V_{1}$--vertex of $\Gamma_{n,n+1}(G)$ to be of \textit{special
Seifert type }if it has only one incident edge $e$ which is dual to an
essential torus, and $G(e)$ is of index $2$ in $G(v)$. Also a $V_{1}$--vertex
of $\Gamma_{n,n+1}(G)$ is of \textit{special solid torus type}, if $v$ is of
solid torus type and does not enclose any crossing annuli. In Lemma 8.5 of
\cite{SS05}, the authors gave a complete list of possible such vertices. The
authors also considered the completion $\Gamma_{n,n+1}^{c}$\ of $\Gamma
_{n,n+1}$. This is obtained from $\Gamma_{n,n+1}$ by re-labelling as $V_{0}%
$--vertices those $V_{1}$--vertices of special Seifert type or of special
solid torus type, then adding isolated $V_{1}$--vertices to keep the graph
bipartite. If the result is not reduced, we reduce it by collapsing edges.

Now the main theorem of \cite{SS05} can be stated.

\begin{theorem}
[Theorem 3.14 of \cite{SS05}]\label{thm21} Let $(G,\partial G)$ be an
orientable $PD(n+2)$ pair such that $G$ is not $VPC$. Let $\mathcal{F}%
_{n,n+1}$ denote the family of equivalence classes of all nontrivial almost
invariant subsets of $G$ which are over a $VPCn$ subgroup, together with the
equivalence classes of all $n$--canonical almost invariant subsets of $G$
which are over a $VPC(n+1)$ subgroup. Finally let $\Gamma_{n,n+1}$ denote the
reduced algebraic regular neighbourhood of $\mathcal{F}_{n,n+1}$ in $G$, and
let $\Gamma_{n,n+1}^{c}$ denote the completion of $\Gamma_{n,n+1}$. Thus
$\Gamma_{n,n+1}$ and $\Gamma_{n,n+1}^{c}$ are bipartite graphs of groups
structures for $G$, with vertices of $V_{0}$--type and of $V_{1}$--type.

Then $\Gamma_{n,n+1}$ and $\Gamma_{n,n+1}^{c}$ have the following properties:

\begin{enumerate}
\item Each $V_{0}$--vertex $v$ of $\Gamma_{n,n+1}$ satisfies one of the
following conditions:

\begin{enumerate}
\item $v$ is isolated, and $G(v)$ is $VPC$ of length $n$ or $n+1$, and the
edge splittings associated to the two edges incident to $v$ are dual to
essential annuli or tori in $G$.

\item $v$ is of $VPC(n-1)$--by--Fuchsian type, and is of $I$--bundle type.

\item $v$ is of $VPCn$--by--Fuchsian type, and is of interior Seifert type.

\item $v$ is of commensuriser type. Further $v$ is of Seifert type, or of
torus type, or of solid torus type.
\end{enumerate}

\item The $V_{0}$--vertices of $\Gamma_{n,n+1}^{c}$ obtained by the completion
process are of special Seifert type or of special solid torus type.

\item Each edge splitting of $\Gamma_{n,n+1}$ and of $\Gamma_{n,n+1}^{c}$ is
dual to an essential annulus or torus in $G$.

\item Any nontrivial almost invariant subset of $G$ over a $VPC(n+1)$ group
and adapted to $\partial G$ is enclosed by some $V_{0}$--vertex of
$\Gamma_{n,n+1}$, and also by some $V_{0}$--vertex of $\Gamma_{n,n+1}^{c}$.

\item If $H$ is a $VPC(n+1)$ subgroup of $G$ which is not conjugate into
$\partial G$, then $H$ is conjugate into a $V_{0}$--vertex group of
$\Gamma_{n,n+1}^{c}$.
\end{enumerate}
\end{theorem}

\begin{remark}
Recall that a vertex $w$ of a $G$--tree $T$ is called \emph{isolated} if it
has valence $2$, and these two edges have the same stabilizer as $w$. The
image of $w$ in $G\backslash T$ is also called \emph{isolated}.
\end{remark}

Notice that any vertex $v$ of $\Gamma_{n,n+1}$\ or $\Gamma_{n,n+1}^{c}$\ has
two types of "boundary" subgroups. The first type comes from the edge groups
of the decomposition and the family of all these subgroups will be denoted by
$\partial_{1}v$. The second type comes from the decomposition of $\partial G$
by edges of the decompositions and this family will be denoted by
$\partial_{0}v$. The first type gives us $PD(n+1)$\ pairs in $(G,\partial G)$,
namely annuli or tori, and the second type gives us $PD(n+1)$\ pairs which are
contained in $\partial G$. In the three-dimensional topological case,
$\partial_{0}v$ and $\partial_{1}v$ correspond to surfaces which can be
amalgamated to yield the boundary of the $3$--manifold $M(v)$ which
corresponds to $v$. But this boundary may be compressible, and so need not
yield a $PD3$--pair. In the general case we get a triple $(G(v);\partial
_{0}v,\partial_{1}v)$ which corresponds to a Poincar\'{e} triad (\cite{Wall1})
but this theory in the case of groups has not been worked out.

We should also discuss the reason for excluding $VPC$ groups from
consideration in Theorem \ref{thm21}. For simplicity we will consider the case
when $\partial G$ is empty, so that $G$ is a $PD(n+2)$ group. Thus
$\mathcal{F}_{n,n+1}$ consists of the equivalence classes of all almost
invariant subsets of $G$ which are over a $VPC(n+1)$ subgroup, so that
$\Gamma_{n,n+1}(G)=\Gamma_{n+1}(G)$. As $G$ is $VPC$ and $PD(n+2)$, it must be
$VPC(n+2)$. As $\partial G$ is empty, cases 1b) and 1d) of Theorem \ref{thm21}
cannot arise. Also as $G\ $is $VPC$, the condition of being of $VPCn$%
--by--Fuchsian type in case 1c) can never occur. It should be replaced by the
condition of being $VPCn$--by--$VPC2$, to have a statement with some chance of
holding. By the definition of $VPC$, any $VPC(n+2)$ group $G$ contains some
$VPC(n+1)$ subgroup, and hence must contain a torus $T$. If $G$ admits a
second torus $T^{\prime}$ which crosses $T$, then $\Gamma_{n+1}(G)$ consists
of a single $V_{0}$--vertex. But this vertex need not satisfy the modified
condition 1c) in the statement of Theorem \ref{thm21}. For example, in section
7 of \cite{Hillman}, the author gives two examples of torsion free $VPC4$
groups which are orientable $PD4$ groups, and do not contain any normal $VPC2$
subgroup. As these examples are finite extensions of $\mathbb{Z}^{4}$, they
contain many subgroups isomorphic to $\mathbb{Z}^{3}$, and hence many tori, so
that $\Gamma_{3}(G)$ consists of a single $V_{0}$--vertex, which cannot
satisfy condition 1a) or the modified condition 1c) in the statement of
Theorem \ref{thm21}. Note that torsion free $VPC3$ groups which are orientable
$PD3$ groups, do satisfy the modified version of Theorem \ref{thm21}. For any
such group is the fundamental group of a closed orientable $3$--manifold $M$
which admits a geometric structure modeled on $E^{3}$, $Nil$ or $Solv$. In the
first two cases, $M\ $is a Seifert fibre space, and $\Gamma_{2}(G)$ consists
of a single $V_{0}$--vertex of $VPC1$--by--$VPC2$ type. In the third case,
either $\Gamma_{2}(G)$ consists of a single isolated $V_{0}$--vertex and a
single isolated $V_{1}$--vertex joined by two edges, so that $\Gamma_{2}(G)$
is a loop, or $\Gamma_{2}(G)$ consists of a single isolated $V_{0}$--vertex,
joined to two $V_{1}$--vertices of special Seifert type.

We recall Definition 5.1 and Proposition 5.3 from \cite{SS05}.

\begin{definition}
An orientable $PD(n+2)$ pair $(G,\partial G)$ is \emph{atoroidal} if any
orientable $VPC(n+1)$ subgroup of $G$ is conjugate into one of the groups in
$\partial G$.
\end{definition}

\begin{proposition}
\label{prop23} Let $(G,\partial G)$ be an orientable atoroidal $PD(n+2)$ pair,
where $n\geq1$. Let $A$ and $B$ be $VPC(n+1)$ groups in $\partial G$, possibly
$A=B$. Let $S$ and $T$ be $VPCn$ subgroups of $A$ and $B$ respectively, and
let $g$ be an element of $G$ such that $gSg^{-1}=T$. Then one of the following holds:

\begin{enumerate}
\item $A$ and $B$ are the same element of $\partial G$, and $g\in A$.

\item $A$ and $B$ are distinct elements of $\partial G$, are the only groups
in $\partial G$, and $A=G=B$. Thus $(G,\partial G)$ is the trivial pair
$(G,\{G,G\})$.

\item $A$ and $B$ are the same element of $\partial G$. Further $A$ is the
only group in $\partial G$, and has index $2$ in $G$.
\end{enumerate}
\end{proposition}

In \cite{SS05}, the above proposition was applied to the $V_{1}$--vertices of
the torus decomposition of a $PD(n+2)$ pair $(G,\partial G)$. More precisely,
let $V$ be a $V_{1}$--vertex of the torus decomposition $T_{n+1}(G,\partial
G)$, and let $K$ denote the associated group $G(V)$. Let $\partial_{1}K$
denote the family of subgroups of $K$ associated to the edges of
$T_{n+1}(G,\partial G)$ incident to $V$, let $\partial_{0}K$ denote the family
of subgroups of $K$ which lie in $\partial G$, and let $\partial K$ denote the
union $\partial_{1}K\cup\partial_{0}K$ of these two families. Then
$(K,\partial K)$ is an orientable atoroidal $PD(n+2)$ pair. Each group in
$\partial_{1}K$ is $VPC(n+1)$, so the above proposition can be applied to any
annulus in $(K,\partial K)$ with ends in $\partial_{1}K$.

We now generalize this idea to apply to $V_{1}$--vertices of $\Gamma
_{n,n+1}^{c}(G)$.

\begin{proposition}
\label{rem24}Let $(G,\partial G)$ be an orientable $PD(n+2)$ pair, where
$n\geq1$. Let $K$ be the group associated to a $V_{1}$--vertex $V$ of
$\Gamma_{n,n+1}^{c}(G)$, and let $\partial_{1}K$ denote the family of
subgroups of $K$ associated to the edges of $\Gamma_{n,n+1}(G)$ incident to
$V$. Let $A$ and $B$ be groups in $\partial_{1}K$, possibly $A=B$. Let $S$ and
$T$ be $VPCn$ subgroups of $A$ and $B$ respectively, and let $g$ be an element
of $K$ such that $gSg^{-1}=T$. Then one of the following holds:

\begin{enumerate}
\item $A$ and $B$ are the same element of $\partial_{1}K$, and $g\in A$.

\item $A$ and $B$ are distinct elements of $\partial_{1}K$, are the only
groups in $\partial_{1}K$, and $A=K=B$. Thus $V$ is an isolated $V_{1}%
$--vertex $V$ of $\Gamma_{n,n+1}^{c}(G)$.
\end{enumerate}
\end{proposition}

\begin{remark}
This result fails if we consider the uncompleted decomposition $\Gamma
_{n,n+1}(G)$. For example, if $V$ is a $V_{1}$--vertex of $\Gamma_{n,n+1}(G)$
of special Seifert type, or of solid torus type such that $V\ $has valence
$1$, and the edge group has index $2$ or $3$ in $G(V)$, then $V$ does not
satisfy either of the conclusions 1)-2).
\end{remark}

\begin{proof}
Let $\partial_{0}K$ denote the family of subgroups of $K$ coming from the
decomposition of $\partial G$ induced by the edge splittings of $\Gamma
_{n,n+1}^{c}(G)$. For later use, we note that any essential annulus in
$(G,\partial G)$ is enclosed by a $V_{0}$--vertex of $\Gamma_{n,n+1}^{c}(G)$.
Thus if an essential annulus $\Lambda$ in $(G,\partial G)$ is enclosed by the
$V_{1}$--vertex $V$ with associated group $K$, it cannot be essential in $K$.
This is because there is an edge $e$ of $\Gamma_{n,n+1}^{c}(G)$ incident to
$V$ such that the associated edge splitting is dual to an annulus
$\Lambda^{\prime}$ covered by $\Lambda$. Note that the group associated to $e$
lies in $\partial_{1}K$.

Now let $\partial K$ denote the union of the two families $\partial_{0}K$ and
$\partial_{1}K$. (Recall that groups in $\partial_{1}K$ are $PD(n+1)$\ pairs
in $(G,\partial G)$, and groups in $\partial_{0}K$ are $PD(n+1)$\ pairs which
are contained in $\partial G$.) The pair $(K,\partial K)$ is again atoroidal,
in the sense that any orientable $VPC(n+1)$ subgroup of $K$ is conjugate into
one of the groups in $\partial K$, but $(K,\partial K)$ need not be a
$PD$--pair. This is because the groups in $\partial_{0}K$ and $\partial_{1}K$
need not be $PD$--groups. Now we let $DK$ denote the double of $K\ $along the
family $\Sigma$ of groups in $\partial_{0}K$ which are not tori, and let
$\partial DK$ denote the family consisting of the induced double of
$\partial_{1}K$ together with the double of the family of torus groups in
$\partial_{0}K$. Note that as $\partial_{1}K$ consists of essential annuli and
tori in $(G,\partial G)$, each group in the induced double of $\partial_{1}K$
is $VPC(n+1)$. For the double of an essential annulus, this is proved in
section 2 of \cite{SS05}. Lemma 8.7 of \cite{SS05} tells us that $(DK,\partial
DK)$ is an orientable atoroidal $PD(n+2)$ pair.

Now we proceed as follows. The group $A$ in $\partial_{1}K$ yields the group
$A^{\prime}$ in $\partial DK$, where $A^{\prime}$ equals $A$ if $A$ is a
torus, and equals the double $DA$ of $A$ if $A$ is an annulus. Similarly the
group $B$ in $\partial_{1}K$ yields the group $B^{\prime}$ in $\partial DK$.
The $VPCn$ subgroups $S$ and $T$ of $A$ and $B$ are subgroups of $A^{\prime}$
and $B^{\prime}$ respectively, and the element $g$ of $K$ such that
$gSg^{-1}=T$ lies in $DK$. Now we apply Proposition \ref{prop23}, to obtain
one of the three cases listed there. Case 1) of Proposition \ref{prop23}
implies that case 1) of Proposition \ref{rem24} holds, and case 2) of
Proposition \ref{prop23} implies that case 2) of Proposition \ref{rem24}
holds. Finally case 3) of Proposition \ref{prop23} implies that either $V$ is
of special Seifert type or of special solid torus type. Neither case can occur
as such vertices cannot be $V_{1}$--vertices of $\Gamma_{n,n+1}^{c}(G)$. This
completes the proof of Proposition \ref{rem24}.
\end{proof}

\section{Examples of almost invariant sets\label{section:examples}}

The discussion in \cite{SS05} was mostly about almost invariant sets which are
adapted to the boundary. However, in \cite{SS02}, Scott gave examples of
almost invariant sets over orientable $VPC2$ subgroups of a $PD3$ pair which
are not adapted to the boundary. This gave rise to the concept of special
canonical torus which was used to show that the JSJ-decomposition of
orientable $3$--manifolds is algebraic, meaning that it depends only on the
fundamental group of the manifold, not the boundary. As discussed earlier, we
will say that an embedded essential annulus or torus in a $3$--manifold $M$
with incompressible boundary is \textit{topologically canonical} if it has
intersection number zero with any (possibly singular) essential annulus or
torus in $(M,\partial M)$. We will say that a splitting of $\pi_{1}(M)$ given
by an essential annulus or torus is \textit{algebraically canonical} if it has
intersection number zero with any almost invariant subset of $\pi_{1}(M)$
which is over $\mathbb{Z}$ or $\mathbb{Z}\times\mathbb{Z}$. See \cite{SS01}
for a discussion of the idea of intersection numbers. Now we recall Scott's example.

\begin{example}
[Scott's example]\label{scottexample}

This is Example 2.13 of \cite{SS02}. Let $F$ be an orientable surface with at
least two boundary components and let $C$ denote one of the boundary
components. Thus $\pi_{1}(F)$ is free, and $\pi_{1}(C)$ is a free factor of
$\pi_{1}(F)$. If the rank of $\pi_{1}(F)$ is at least $3$, then it is easy to
see that there is a nontrivial splitting of $\pi_{1}(F)$ as an amalgamated
free product over $\pi_{1}(C)$. Similar considerations apply to express
$\pi_{1}(F)$ as an HNN extension if it has rank $2$.

We now take two copies $F_{1}$, $F_{2}$ of $F$ and consider the two
$3$--manifolds $M_{i}=F_{i}\times S^{1}$, each with a boundary component
$T_{i}$ corresponding to $C_{i}\times S^{1}$. Form a $3$--manifold $M$ by
gluing the $M_{i}$'s along $T_{i}$ so that the fibrations do not match. The
resulting torus $T$ is a topologically canonical torus in the JSJ splitting of
$M$. If each $\pi_{1}(F_{i})$ has rank at least $3$, we have $\pi_{1}%
(M_{i})=A_{i}\underset{H_{i}}{\ast}B_{i}$, $i=1,2$, where $H_{i}=\pi_{1}%
(T_{i})$. If $G$ denotes $\pi_{1}(M)$, and $H$ denotes the subgroup
$H_{1}=H_{2}$, and $A=A_{1}\underset{H}{\ast}A_{2}$, $B=B_{1}\underset{H}{\ast
}B_{2}$, we have a splitting $G=A\underset{H}{\ast}B$ of $G$ that crosses the
splitting associated to $T$. Thus although $T$ is topologically canonical, it
is not algebraically canonical. Notice that embedded essential annuli in
$M_{1}$ and $M_{2}$, disjoint from $T$, yield splittings of $G$ over the
fibres of $M_{1}$ and $M_{2}$, so that $G$ also has splittings over
incommensurable cyclic subgroups of $H$.
\end{example}

We construct some more examples. Consider $M_{i}=T_{i}\times I\cup N_{i}$,
where $N_{i}$ is an orientable $3$--manifold attached to $T_{i}\times\{1\}$
along at least two disjoint annuli in $\partial N_{i}$. Now form $M$ by
identifying the $M_{i}$'s along $T_{i}\times\{0\}$, and let $T$ denote the
torus $T_{1}\times\{0\}=T_{2}\times\{0\}$. Assume that the annuli in
$T_{1}\times\{1\}$ and $T_{2}\times\{1\}$ used to construct $M_{1}$ and
$M_{2}$ carry incommensurable subgroups of $\pi_{1}(T)=H$. Thus $G=\pi_{1}(M)$
splits over incommensurable cyclic subgroups of $H$. Again, $T$ is a
topologically canonical torus in the JSJ decomposition of $M$. Now we form
$H$--almost invariant subsets of $G=\pi_{1}(M)$ as follows. Consider the cover
$M_{H}$ of $M$ with $\pi_{1}(M_{H})=H$, so that each $T_{i}\times I$ lifts to
$M_{H}$, and the pre-image $\widetilde{N_{i}}$ of each $N_{i}$ is
disconnected. Let $C_{i}$, $i=1,2$, be one of the annuli in this lift of
$T_{i}\times\{1\}$ used to construct $M_{i}$, and let $N_{i}^{\prime}$ be the
component of $\widetilde{N_{i}}$ attached to $C_{i}$. For $i=1,2$, let $X_{i}$
be the set of vertices of $M_{H}$ in $T_{i}\times\lbrack0,1]$ together with
the vertices in $N_{i}^{\prime}$. This gives us two sets $X_{1}$, $X_{2}$ on
different sides of $T$. To $X_{1}$ we add the vertices in $\widetilde{N_{2}%
}-N_{2}^{\prime}$ and to $X_{2}$ we add the vertices in $\widetilde{N_{1}%
}-N_{1}^{\prime}$ to obtain two $H$--almost invariant sets $Y_{1}$, $Y_{2}$.
Clearly $Y_{2}=Y_{1}^{\ast}$ and crosses the almost invariant set $X$
determined by $T$, namely $X$ consists of all vertices of $M_{H}$ on one side
of $T$.

We can mix the above two examples to get one manifold of each type on each
side of $T$. These types fall under the heading of $V_{0}$--vertices of
commensuriser type of Theorem \ref{thm21}. The first examples are special
cases of the so called peripheral Seifert type in \cite{SS05}, that is, those
Seifert type pieces in the decompositions $\Gamma_{n,n+1}$ and $\Gamma
_{n,n+1}^{c}$\ of a $PD(n+2)$ pair $(G,\partial G)$ which "intersect the
boundary". The second are called toral type 2), in which one component of
$\partial(T\times I)$ is an edge of the decomposition and the other boundary
component intersects $\partial G$ in a number of parallel annuli. Note that in
\cite{SS05}, the definition of Seifert type requires the base orbifold to have
fundamental group which is not virtually cyclic. The terminology torus type is
used if this base group is virtually cyclic. These examples suggest the
following definition.

\begin{definition}
\label{defnofspecialcanonicaltorus}Let $(G,\partial G)$ be an orientable
$PD(n+2)$ pair, such that $G$ is not $VPC$. A splitting of $G$ over a
$VPC(n+1)$ subgroup $H$ is called a \textbf{\emph{special canonical torus}} if
it has intersection number zero with any essential annulus in $(G,\partial
G)$, and $G$ splits over incommensurable $VPCn$ subgroups of $H$.
\end{definition}

\begin{remark}
In Proposition \ref{characterizespecialcanonical}, we will prove that a
special canonical torus in $G$ is an edge splitting of $\Gamma_{n,n+1}(G)$ or
$\Gamma_{n,n+1}^{c}(G)$, so that this concept depends only on $G$, and not on
$\partial G$. Recall that the edge splittings of $\Gamma_{n,n+1}(G)$ are the
same as those of $\Gamma_{n,n+1}^{c}(G)$. As these edge splittings are each
dual to an essential annulus or torus in $(G,\partial G)$, it follows that a
special canonical torus in $G$ is dual to an essential torus in $(G,\partial
G)$, which partly justifies the terminology. Note that if $\partial G$ is
empty, so that $G$ is a $PD(n+2)$ group, then $G\ $cannot split over a $VPCn$
subgroup. Thus special canonical tori do not exist in this case.
\end{remark}

Essentially the same definition was used in \cite{SS02}. However, the above
examples are not the only possibilities. Again let $M_{1}$ denote the
$3$--manifold $F_{1}\times S^{1}$, where $F_{1}$ is an orientable surface with
at least two boundary components. Also let $M_{2}$ denote the orientable
$3$--manifold which is a twisted $I$--bundle over the Klein bottle. Form a
$3$--manifold $M$ by gluing the boundary torus $T$ of $M_{2}$ to one of the
boundary tori of $M_{1}$. In this case there are two distinct Seifert
fibrations on $M_{2}$, reflecting the fact that the Klein bottle itself has
two distinct Seifert fibrations. But so long as the gluing is chosen not to
match the fibration of $M_{1}$ with either fibration of $M_{2}$, the torus $T$
will again be a topologically canonical torus in the JSJ splitting of $M$. It
will also not be algebraically canonical but will be a special canonical
torus. The easiest way to see these facts is to note that $M$ is double
covered by the union of two copies of $M_{1}$, glued along a boundary torus so
that their fibrations do not match. Thus $M\ $is double covered by one of
Scott's examples.

Next we need the following technical result.

\begin{lemma}
\label{paralleledges}Let $(G,\partial G)$ be an orientable $PD(n+2)$ pair such
that $G$ is not $VPC$, let $J$ and $J^{\prime}$ be $VPC(n+1)$ subgroups of
$G$, and let $f$ and $f^{\prime}$ be edges of the universal covering $G$--tree
$T$ of $\Gamma_{n,n+1}^{c}(G)$, with stabilizers $J\ $and $J^{\prime}$
respectively. If $J$ and $J^{\prime}$ are commensurable, then $J=J^{\prime}$
and one of the following cases holds:

\begin{enumerate}
\item $f=f^{\prime}$.

\item $f\cap f^{\prime}$ is an isolated vertex.

\item $f\cap f^{\prime}$ is a $V_{0}$--vertex $w$ of valence $2$ whose
stabiliser contains an element interchanging $f$ and $f^{\prime}$, and so
contains $J$ with index $2$.

\item There are consecutive adjacent edges $f,b,b^{\prime},f^{\prime}$ of $T$
such that $b\cap f$ and $b^{\prime}\cap f^{\prime}$ are isolated $V_{1}%
$--vertices, and $b\cap b^{\prime}$ is a $V_{0}$--vertex $w$ of valence $2$
whose stabiliser contains an element interchanging $b$ and $b^{\prime}$, and
so contains $J$ with index $2$.
\end{enumerate}
\end{lemma}

\begin{proof}
As above, we will start with a $K(\pi,1)$ pair $(M,\partial M)$ and a
decomposition of $(M,\partial M)$ mimicking the decomposition $\Gamma
_{n,n+1}^{c}$. If $f=f^{\prime}$, we have case 1) of the lemma, so for the
rest of the proof we will assume that $f\neq f^{\prime}$.

The edges $f$ and $f^{\prime}$ determine splittings of $G$ over $J\ $and
$J^{\prime}$, so that these are tori in $(G,\partial G)$. Let $L$ denote the
intersection $J\cap J^{\prime}$, so that $L\ $is also $VPC(n+1)$, and let
$\Sigma$ denote a torus with fundamental group $L$. Consider the $PD(n+2)$
pair $(K,\partial K)$ obtained by cutting $(G,\partial G)$ along these two
splittings, and let $(N,\partial N)$ be obtained from $M$ in the corresponding
way. (It is possible that $f$ and $f^{\prime}$ yield a single splitting.) The
path in $T$ between $f$ and $f^{\prime}$ determines (up to homotopy) a map
$F:(\Sigma\times I,\Sigma\times\partial I)\rightarrow(N,\partial N)$. We let
$N_{0}$ denote the component of $N$ which contains the image of $F$, and
consider the induced map $(\Sigma\times I,\Sigma\times\partial I)\rightarrow
(N_{0},\partial N_{0})$ which we continue to denote by $F$. The degree of $F$
on each component of $\Sigma\times\partial I$ is non-zero, and if
$F(\Sigma\times\partial I)$ is contained in a single component of $\partial
N_{0}$, these degrees add. Thus the degree of $F$ is non-zero. It follows that
$(N_{0},\partial N_{0})$ has a finite cover to which $F$ lifts by a map which
is an isomorphism of fundamental groups. In particular, it follows that the
boundary of this cover consists of two tori with fundamental groups equal to
$L$. Hence either $\partial N_{0}$ consists of two tori with the same
fundamental group as $N_{0}$, or $\partial N_{0}$ consists of a single torus
with fundamental group $J$, and $J$ has index $2$ in $K=\pi_{1}(N_{0})$. In
either case, it follows that $J=J^{\prime}$. In the first case, the path
$\lambda$ joining $f$ and $f^{\prime}$ in $T$ has stabilizer $J$, and each
vertex on that path must be isolated. As $\Gamma_{n,n+1}^{c}(G)$ is reduced,
we must have case 2) of the lemma. In the second case, the stabilizer of
$\lambda$ contains $J$ with index $2$, and contains a reflection. Thus there
is a vertex $w$ of $\lambda$ of valence $2$ whose stabiliser contains an
element interchanging the two incident edges, and so contains $J$ with index
$2$. Further, as $\Gamma_{n,n+1}^{c}(G)$ is reduced, either $w$ equals $f\cap
f^{\prime}$ or there are consecutive adjacent edges $f,b,b^{\prime},f^{\prime
}$ of $T$ such that $b\cap f$ and $b^{\prime}\cap f^{\prime}$ are isolated
vertices, and $w$ equals $b\cap b^{\prime}$. In either case, this implies that
the image of $w$ in $\Gamma_{n,n+1}^{c}(G)$ is of special Seifert type, and so
is a $V_{0}$--vertex. Thus we must have cases 3) or 4) of the lemma.
\end{proof}

Now we can give an alternative description of special canonical tori in terms
of $\Gamma_{n,n+1}^{c}(G)$.

\begin{proposition}
\label{characterizespecialcanonical}For an orientable $PD(n+2)$ pair
$(G,\partial G)$ such that $G$ is not $VPC$, a splitting $\alpha$ of $G$ over
a $VPC(n+1)$ subgroup $H$ is a special canonical torus if and only if the
following conditions hold:

\begin{enumerate}
\item $\alpha$ is an edge splitting of $\Gamma_{n,n+1}^{c}(G)$.

\item The $V_{1}$--vertex $w$ of $\alpha$ is isolated.

\item Each of the $V_{0}$--vertices adjacent to $w$ is of peripheral Seifert
type, of toral type 2), or of special Seifert type.

\item At most one $V_{0}$--vertex can be of special Seifert type.

\item If the two edges incident to $w$ form a loop, there is only one adjacent
$V_{0}$--vertex. In this case, that vertex must be of peripheral Seifert type.
\end{enumerate}

(The concepts of peripheral Seifert type, and toral type 2) are discussed
immediately preceding Definition \ref{defnofspecialcanonicaltorus},
and\textquotedblleft special Seifert type" was defined before Theorem
\ref{thm21}. The reader is referred to \cite{SS05} for full details.)
\end{proposition}

\begin{proof}
First suppose that $\alpha$ is a splitting of $G$ over a $VPC(n+1)$ subgroup
$H$ which satisfies conditions 1)-5) of the Proposition. As $\alpha$ is an
edge splitting of $\Gamma_{n,n+1}^{c}(G)$, it has intersection number zero
with any essential annulus in $(G,\partial G)$. It remains to show that $G$
splits over incommensurable $VPCn$ subgroups of $H$. By condition 2), the
$V_{1}$--vertex $w$ of $\alpha$ is isolated. Suppose there is no $V_{0}%
$--vertex adjacent to $w$ of special Seifert type. Then each of the adjacent
$V_{0}$--vertices meets $\partial G$ in annuli (and/or tori in the case of
peripheral Seifert type) and choosing an essential embedded annulus in the
$V_{0}$--vertex and with boundary in $\partial G$ determines a splitting of
$G$ over the $VPCn$ subgroup of $H$ carried by the fibres. These subgroups of
$H$ must be incommensurable, as otherwise there would be an annulus in
$(G,\partial G)$ which crosses $\alpha$, contradicting the fact that any edge
splitting of $\Gamma_{n,n+1}^{c}$ crosses no annulus in $(G,\partial G)$. Thus
$\alpha$ is a special canonical torus in this case. Next suppose there is a
$V_{0}$--vertex adjacent to $w$ of special Seifert type. Denote its associated
group by $\overline{H}$, and recall that $H$ is a subgroup of index $2$ in
$\overline{H}$. As before, choosing an essential embedded annulus in the other
$V_{0}$--vertex with boundary in $\partial G$ determines a splitting of $G$
over the $VPCn$ subgroup $L$ of $H$ carried by the fibres. Now let $g$ denote
an element of $\overline{H}-H$. We also have a splitting of $G$ over $L^{g}$.
As $g$ normalises $H$, this is also a subgroup of $H$. Finally $L\ $and
$L^{g}$ must be incommensurable subgroups of $H$, for otherwise there would be
an annulus in $(G,\partial G)$ which crosses $\alpha$, contradicting the fact
that $\alpha$ is an edge splitting of $\Gamma_{n,n+1}^{c}$. Thus again
$\alpha$ is a special canonical torus.

Now suppose that $\alpha$ is a special canonical torus. By Lemma 6.2 of
\cite{SS05}, the fact that $\alpha$ has intersection number zero with any
essential annulus in $(G,\partial G)$ implies that $\alpha$ is adapted to
$\partial G$. Thus the splitting $\alpha$ is dual to a torus in $(G,\partial
G)$, and so must be enclosed by some $V_{0}$--vertex of $\Gamma_{n,n+1}%
^{c}(G)$. As $G$ splits over a $VPCn$ subgroup $L\subset H$, the pair
$(G,\partial G)$ admits an essential annulus with group $L$, and any such
annulus must be enclosed by a $V_{0}$--vertex $v$ of $\Gamma_{n,n+1}^{c}(G)$,
of commensuriser type, so that $G(v)$ contains $H$. Thus $\alpha$ is enclosed
by $v$. As $\alpha$ has intersection number zero with any essential annulus in
$(G,\partial G)$, it follows that $\alpha$ must be the splitting of $G$
associated to an edge of $\Gamma_{n,n+1}^{c}(G)$ incident to $v$. This proves
that $\alpha$ satisfies condition 1) of the proposition. Further $v$ must be
of peripheral Seifert type or of toral type 2). Now we use the hypothesis that
$G$ splits over two incommensurable $VPCn$ subgroups $L\ $and $L^{\prime}$ of
$H$. So the pair $(G,\partial G)$ admits an essential annulus with group
$L^{\prime}$, which is enclosed by a $V_{0}$--vertex $v^{\prime}$ of
$\Gamma_{n,n+1}(G)$ which also has an incident edge with splitting $\alpha$.
Further $v^{\prime}$ must be of peripheral Seifert type or of toral type 2).

If $v$ and $v^{\prime}$ are distinct, the incident edges with splitting
$\alpha$ must also be distinct. As neither of $v$ and $v^{\prime}$ is isolated
or of special Seifert type, Lemma \ref{paralleledges} implies that $v$ and
$v^{\prime}$ must be separated by an isolated $V_{1}$--vertex, so that
$\alpha$ satisfies conditions 2)-5) of the Proposition, as required.

If $v$ and $v^{\prime}$ coincide, there must be $g\in G$ such that $L^{\prime
}=L^{g}$. There are now two cases, depending on whether or not there are two
distinct edges incident to $v$ with associated splitting $\alpha$. Again we
will apply Lemma \ref{paralleledges} and use the fact that $v$ is not isolated
nor of special Seifert type. If there are two distinct such edges, Lemma
\ref{paralleledges} implies that they meet in an isolated $V_{1}$--vertex. In
addition, $v$ cannot be of torus type 2), as such a $V_{0}$--vertex can have
at most one incident edge dual to a torus. It follows that $\alpha$ satisfies
conditions 2)-5) of the Proposition. If there is only one such edge, Lemma
\ref{paralleledges} implies that there is a $V_{0}$--vertex $v^{\prime\prime}$
of special Seifert type which is adjacent to $v$ and separated from $v$ by an
isolated $V_{1}$--vertex $w$. Again this implies that $\alpha$ satisfies
conditions 2)-5) of the Proposition, as required.
\end{proof}

We can now apply Proposition \ref{characterizespecialcanonical} to obtain the
following result.

\begin{proposition}
\label{specialcanonicaltoruscfrossesa.i.setoverH}Let $(G,\partial G)$ be an
orientable $PD(n+2)$ pair such that $G$ is not $VPC$. Let $\alpha$ be a
special canonical torus in $(G,\partial G)$ with group $H$. Then $\alpha$
crosses some almost invariant subset of $G$ over a subgroup of finite index in
$H$.
\end{proposition}

\begin{proof}
From the definition of a special canonical torus, $G\ $splits over
incommensurable subgroups $L$ and $L^{\prime}$ of $H$. Let $X$ be the
$H$--almost invariant subset of $G$ determined (up to equivalence and
complementation) by $\alpha$. We will apply Conditions 1)-5) of Proposition
\ref{characterizespecialcanonical}. Thus $\alpha$ is an edge splitting of
$\Gamma_{n,n+1}^{c}(G)$, and the $V_{1}$--vertex $w$ of $\alpha$ is isolated.

If there are two distinct $V_{0}$--vertices $v$ and $v^{\prime}$ adjacent to
$w$, neither of special Seifert type, we can assume that the splitting of
$G\ $over $L$ is enclosed by $v$, and that the splitting of $G\ $over
$L^{\prime}$ is enclosed by $v^{\prime}$. Thus there is an edge of
$\Gamma_{n,n+1}^{c}(G)$ incident to $v$ with associated splitting dual to an
annulus with group $L$, and there is an edge of $\Gamma_{n,n+1}^{c}(G)$
incident to $v^{\prime}$ with associated splitting dual to an annulus with
group $L^{\prime}$. Let $Y$ denote the $L$--almost invariant subset of $G$
determined by the edge splitting of $G\ $over $L$, and let $Y^{\prime}$ denote
the $L^{\prime}$--almost invariant subset of $G$ determined by the edge
splitting of $G\ $over $L^{\prime}$. By replacing each of $X$, $Y$ and
$Y^{\prime}$ by its complement if needed, we can arrange that $Y\subset X$ and
$Y^{\prime}\subset X^{\ast}$. Then $H(Y\cup Y^{\prime})$ is a $H$--almost
invariant subset of $G$. This subset crosses $X$, and hence crosses $\alpha$,
unless we are in one of the exceptional cases where $HY=X$, or $HY^{\prime
}=X^{\ast}$. Now for any $h\in H$, the set $hY$ is equal to or disjoint from
$Y$. Thus if $P$ is a proper subgroup of finite index in $H$ which contains
$L$, then $PY$ and $X-PY$ are both $H$--infinite. Similarly if $P^{\prime}$ is
a proper subgroup of finite index in $H$ which contains $L^{\prime}$, then
$P^{\prime}Y^{\prime}$ and $X^{\ast}-P^{\prime}Y^{\prime}$ are both
$H$--infinite. Thus, if $Q$ denotes $P\cap P^{\prime}$, then $Q$ has finite
index in $H$, and $Q(Y\cup Y^{\prime})$ is a $Q$--almost invariant subset of
$G$, which crosses $X$, and hence crosses $\alpha$, as required.

If there are two distinct $V_{0}$--vertices $z$ and $z^{\prime}$ adjacent to
$w$, and if $z^{\prime}$ is of special Seifert type, there is a homomorphism
$G(z^{\prime})\rightarrow\mathbb{Z}_{2}$, with kernel $H$. This extends to a
homomorphism $G\rightarrow\mathbb{Z}_{2}$, which is trivial on all vertex
groups other than $G(z^{\prime})$. Let $K\ $denote the kernel of this
homomorphism, so that $K$ is of index $2$ in $G$. This naturally has the
structure of a $PD(n+2)$ pair, and there is a natural map $\Gamma_{n,n+1}%
^{c}(K)\rightarrow$ $\Gamma_{n,n+1}^{c}(G)$. The pre-image of $z^{\prime}$ is
an isolated $V_{1}$--vertex of $\Gamma_{n,n+1}^{c}(K)$. The adjacent $V_{0}%
$--vertices consist of two copies of $z$. Hence the preceding paragraph yields
a subgroup $Q$ of finite index in $H$, and a $Q$--almost invariant subset of
$K$ which crosses $\alpha$. It follows that there is also a $Q$--almost
invariant subset of $G$ which crosses $\alpha$, as required.

If the two edges incident to $w$ form a loop, so there is only one adjacent
$V_{0}$--vertex $z$, Condition 5) tells us that $z$ must be of peripheral
Seifert type. This loop determines a natural map from $G\ $to $\mathbb{Z}$,
which is trivial on all vertex groups, and hence determines a natural map from
$G\ $to $\mathbb{Z}_{2}$, which is trivial on all vertex groups. The kernel is
a subgroup $K$ of $G$ of index $2$ which is still naturally a $PD(n+2)$ pair,
and again there is an edge splitting over $H$ of $\Gamma_{n,n+1}^{c}(K)$ whose
$V_{1}$--vertex is isolated. The adjacent $V_{0}$--vertices now consist of two
copies of $z$. As before, this yields a subgroup $Q$ of finite index in $H$,
and a $Q$--almost invariant subset of $G$ which crosses $\alpha$, as required.
\end{proof}

We note that in Example \ref{scottexample}, the special canonical torus with
group $H$ crosses a splitting over the same group $H$. Thus it seems
reasonable to ask the following.

\begin{problem}
Let $(G,\partial G)$ be an orientable $PD(n+2)$ pair such that $G$ is not
$VPC$. Let $\alpha$ be a special canonical torus in $(G,\partial G)$ with
group $H$. When is it true that $\alpha$ crosses a splitting of $G$ over $H$?
\end{problem}

Proposition \ref{specialcanonicaltoruscfrossesa.i.setoverH} does nothing to
answer this question. But the argument does show that in most cases, a special
canonical torus with group $H$ crosses an almost invariant set over the same
group $H$. For the special case in the argument when $HY=X$ can only occur if
$v$ is of torus type 2) and also has only one incident edge with associated
splitting dual to an annulus. A similar statement for $v^{\prime}$ holds if
$HY^{\prime}=X^{\ast}$. We believe that the above problem has a positive
answer except possibly in these exceptional cases. In the exceptional cases,
it seems possible that a special canonical torus with group $H$ may not cross
any $H$--almost invariant subset of $G$.

Here is a construction which shows that in many cases, a special canonical
torus $\alpha$ in $(G,\partial G)$ with group $H$ crosses a splitting of $G$
over $H$. Suppose that there are two distinct $V_{0}$--vertices $v$ and
$v^{\prime}$ adjacent to $w$, and that each of $v$ and $v^{\prime}$ has at
least two incident edges in addition to the edge joining it to $w$. Let $e$ be
an edge of $\Gamma_{n,n+1}^{c}(G)$ incident to $v$ with associated splitting
dual to an annulus with group $L$, and let $e^{\prime}$ be an edge of
$\Gamma_{n,n+1}^{c}(G)$ incident to $v^{\prime}$ with associated splitting
dual to an annulus with group $L^{\prime}$.\ Now construct a new graph of
groups structure for $G$ by sliding the end of $e$ at $v$ along the two edges
joining $v$ to $v^{\prime}$, and also sliding the end of $e^{\prime}$ at
$v^{\prime}$ along the two edges joining $v^{\prime}$ to $v$. Let $\beta$
denote the splitting of $G$ determined by each of the two edges joining $v$
and $v^{\prime}$ to $w$. Clearly $\beta$ is a splitting of $G$ over $H$, and
by considering the universal covering $G$--trees of $\Gamma_{n,n+1}^{c}(G)$
and of the new graph of groups, it is easy to see that it must cross $\alpha$.

\section{Proof of the main result\label{section:mainresult}}

In this section, we prove the main theorem below and then a result about
enclosings of all almost invariant sets over $VPC(n+1)$ groups. First we need
to introduce yet another version of the term "canonical", which generalizes
the term "algebraically canonical" discussed in the introduction. Let
$\mathcal{E}_{n,n+1}(G)$ denote the collection of all a.i. subsets of $G$
which are over a $VPCn$ or $VPC(n+1)$ subgroup. We will say that an element of
$\mathcal{E}_{n,n+1}$ is \textit{canonical} if it has intersection number zero
with every element of $\mathcal{E}_{n,n+1}$.

\begin{theorem}
[Main result]\label{mainresult}Let $(G,\partial G)$ be an orientable $PD(n+2)$
pair such that $G$ is not $VPC$. The edge splittings of $\Gamma_{n,n+1}(G)$
and of $\Gamma_{n,n+1}^{c}(G)$ are either canonical or are special canonical tori.
\end{theorem}

\begin{remark}
Proposition \ref{specialcanonicaltoruscfrossesa.i.setoverH} shows that these
two conditions are mutually exclusive. Note that if $\partial G$ is empty, the
decomposition $\Gamma_{n+1}(G)$ is an algebraic regular neighbourhood of all
almost invariant subsets of $G$ over a $VPC(n+1)$ subgroup, so that all the
edge splittings of $\Gamma_{n+1}(G)$ and of $\Gamma_{n+1}^{c}(G)$ are
canonical, by definition.
\end{remark}

We start by setting up some notation. The main step of the start of the
argument is discussed in Section 6 of \cite{SS05} in a different context.
There the authors used it to show that $n$--canonical almost invariant sets
over $VPC(n+1)$ groups are automatically adapted to the boundary. Here we use
it differently.

Let $(G,\partial G)$ be an orientable $PD(n+2)$ pair, and let $T$ denote the
universal covering $G$--tree of the graph of groups $\Gamma_{n,n+1}^{c}(G)$.
Let $(M,\partial M)$ be a $K(\pi,1)$ pair with a decomposition mimicking the
decomposition $\Gamma_{n,n+1}^{c}$. This induces a decomposition of the
universal cover $(\widetilde{M},\partial\widetilde{M})$ of $({M,\partial M)}$,
and we have an equivariant map $\widetilde{M}\rightarrow T$ preserving the
decompositions. If $v$ is a vertex of $\Gamma_{n,n+1}^{c}$ or of $T$, the
corresponding subspaces of $M$ or of $\widetilde{M}$ will be denoted by
{$M_{v}$ or }$\widetilde{M}_{v}$ respectively, and similarly for edges.

\begin{lemma}
\label{ThereisXnotadaptedtodGcrossinge}Let $(G,\partial G)$ be an orientable
$PD(n+2)$ pair such that $G$ is not $VPC$, and let $e$ be an edge of $T$ such
that the associated splitting of $G$ is not canonical. Then there is an almost
invariant set $X$ over a $VPC(n+1)$ subgroup $H$ of $G$ which is not adapted
to $\partial G$ and crosses $e$.
\end{lemma}

\begin{proof}
Recall that $\Gamma_{n,n+1}^{c}(G)$ is the completion of the reduced algebraic
regular neighbourhood $\Gamma_{n,n+1}(G)$ of $\mathcal{F}_{n,n+1}$ in $G$,
where $\mathcal{F}_{n,n+1}$ denotes the family of equivalence classes of all
nontrivial almost invariant subsets of $G$ which are over a $VPCn$ subgroup,
together with the equivalence classes of all $n$--canonical almost invariant
subsets of $G$ which are over a $VPC(n+1)$ subgroup. In \cite{SS05}, the
authors showed that $n$--canonical almost invariant subsets of $G$ over
$VPC(n+1)$ subgroups are automatically adapted to $\partial G$. Further, if we
enlarge the family $\mathcal{F}_{n,n+1}$ to include all almost invariant sets
over $VPC(n+1)$ subgroups which are adapted to $\partial G$, the new family
$\mathcal{G}_{n,n+1}$ has the same regular neighbourhood $\Gamma_{n,n+1}(G)$.
In particular, no set in $\mathcal{G}_{n,n+1}$ can cross any edge of
$\Gamma_{n,n+1}^{c}(G)$. Thus our assumption on $e$ implies that there must be
an almost invariant set $X$ over a $VPC(n+1)$ subgroup $H$ of $G$ which is not
adapted to $\partial G$ and crosses $e$, as required.
\end{proof}

\begin{lemma}
\label{LisVPCn}Let $(G,\partial G)$ be an orientable $PD(n+2)$ pair such that
$G$ is not $VPC$, and let $X$ be an almost invariant set over a $VPC(n+1)$
subgroup $H$ of $G$ which is not adapted to $\partial G$. Then the following
statements hold:

\begin{enumerate}
\item There is a group $S$ with a conjugate in $\partial G$ such that $L=H\cap
S$ is $VPCn$, and $X\cap S$ and $X^{\ast}\cap S$ are both $H$--infinite.

\item There is a $V_{0}$--vertex $v$ of $T$ of commensuriser type which
encloses all $L$--almost invariant subsets of $G$. Further, $\widetilde{M}%
_{v}\cap\partial\widetilde{M}$ is non-empty, ${v}$ is of Seifert type or of
torus type, and {$G(v)=Comm_{G}(L)=N_{G}(L)$ contains }$H$.
\end{enumerate}
\end{lemma}

\begin{proof}
1) Since $X$ is not adapted to $\partial G$, there is a group $S$ in $\partial
G$, and $g\in G$ such that $X\cap gS$ and $X^{\ast}\cap gS$ are both
$H$--infinite. By replacing $S$ by a conjugate if needed, we can arrange that
$X\cap S$ and $X^{\ast}\cap S$ are both $H$--infinite. Consider the component
$\Sigma$ of $\partial\widetilde{M}$ with stabilizer $S$, and identify $X$ and
$X^{\ast}$ with subsets of the $0$--skeleton $\widetilde{M}_{0}$ of
$\widetilde{M}$. We have that the intersections of $X$ and $X^{\ast}$ with the
$0$--skeleton $\Sigma_{0}$ of $\Sigma$ are $H$--infinite. Hence they are
$L$--infinite, where $L=H\cap S$. Thus $e(S,L)\geq2$. As $S$ is $PD(n+1)$ and
$H$ is $VPC$, it follows that $L$ is $VPCn$.

2) By replacing $H$ by a subgroup of finite index if necessary, we can assume
that $L$ is normal in $H$ with $L\backslash H$ infinite cyclic. Let
$P_{L}:\widetilde{M}\rightarrow M_{L}$, and $P_{H}:\widetilde{M}\rightarrow
M_{H}$ denote the covering projections, and let $\Sigma_{L}$ and $\Sigma_{H}$
denote the images of $\Sigma$ in $\partial M_{L}$ and $\partial M_{H}$
respectively. As $L\backslash H$ acts on $M_{L}$, there are infinitely many
translates of $\Sigma_{L}$ each with fundamental group $L$, and thus
infinitely many essential annuli in $M_{L}$. In \cite{SS05}, it was shown
that, if the number of essential annuli in $M_{L}$ is at least $4$, there is a
$V_{0}$--vertex ${v}$ of $T$ of commensuriser type which encloses all
$L$--almost invariant subsets of $G$. Hence the stabilizer, $G(v)$, of $v$
contains $H$, and $\widetilde{M}_{v}$ intersects $\partial\widetilde{M}$. It
is possible that $\widetilde{M}_{v}\cap\partial\widetilde{M}$ contains
$\Sigma$. This happens if $S$ is $VPC(n+1)$. In any case, $v$ is a $V_{0}%
$--vertex with $H\subset G(v)$, and $\widetilde{M}_{v}\cap\partial
\widetilde{M}$ is non-empty, and $L$ stabilizes $\widetilde{M}_{v}\cap
\partial\widetilde{M}$. Thus $v$ is either of Seifert type or of toral type
(see Definition 3.12 of \cite{SS05}). If $v$ is of Seifert type, Lemma 5.10 of
\cite{SS05} tells us that {$G(v)=Comm_{G}(L)=N_{G}(L)$}. If $v$ is of toral
type, we use the fact that $G(v)$ is $VPC(n+1)$ and splits over $L$. Now Lemma
1.10 of \cite{SS05} implies that $L$ is normal in $G(v)$ with quotient
$\mathbb{Z}$ or $\mathbb{Z}_{2}\ast\mathbb{Z}_{2}$. It follows that in this
case also {$G(v)=Comm_{G}(L)=N_{G}(L)$.}
\end{proof}

Note that if $X$ crosses an edge $e$ of $T$, Lemma \ref{LisVPCn} does not tell
us that $e$ is incident to the vertex $v$ of $T$ obtained in part 2) of the
above lemma. However the next lemma assures us that $X$ must cross some edge
incident to $v$.

\begin{lemma}
\label{Xcrossestorus}Let $(G,\partial G)$ be an orientable $PD(n+2)$ pair such
that $G$ is not $VPC$, let $T$ denote the universal covering $G$--tree of the
graph of groups $\Gamma_{n,n+1}^{c}(G)$, and let $e$ be an edge of $T$. Let
$X$ be an almost invariant set over a $VPC(n+1)$ subgroup $H$ of $G$ which is
not adapted to $\partial G$ and crosses $e$.

Using the notation of Lemma \ref{LisVPCn}, if $f$ is the first edge on the
path in $T$ from $v$ to $e$, then $G(f)$ is $VPC(n+1)$, so the associated
splitting of $G$ is dual to a torus, and $X$ crosses this torus.
\end{lemma}

\begin{remark}
\label{torustypeimpliestype2)}If $v$ is of torus type, the fact that an edge
incident to $v$ determines a splitting of $G$ dual to a torus means that $v$
is of torus type 2).
\end{remark}

\begin{proof}
Let $Z$ and $Z^{\ast}$ denote the almost invariant sets associated to $e$
chosen so that $Z$ contains $G(v)$, and let $Y$ and $Y^{\ast}$ denote the
almost invariant sets associated to $f$ chosen so that $Y\subset Z$ and
$Y^{\ast}\supset Z^{\ast}$. As $X$ {crosses} $Z$, we know that $X\cap Z^{\ast
}$ and $X^{\ast}\cap Z^{\ast}$ are both $H$-infinite. As $Y^{\ast}\supset
Z^{\ast}$, it follows that $X\cap Y^{\ast}$ and $X^{\ast}\cap Y^{\ast}$ are
also both $H$-{infinite}.

Now suppose that $G(f)$ is $VPCn$. As $v$ is of Seifert type or of toral type,
the edge group $G(f)$ must be commensurable with $L$. Thus $\delta Y$ is
$L$--finite. As $H\subset G(v)$, and $\delta X$ is $H$--finite, it follows
that $\delta X$ lies in a bounded neighbourhood of $\widetilde{M}_{v}$. As
$G(f)$ is commensurable with $L$, this implies that $\delta X\cap Y^{\ast}$ is
$L$--finite. As $\delta Y$ is also $L$--finite, it follows that $X\cap
Y^{\ast}$ and $X^{\ast}\cap Y^{\ast}$ each have $L$--finite coboundary. For
$\delta(X\cap Y^{\ast})=(X\cap\delta Y^{\ast})\cup(\delta X\cap Y^{\ast})$. We
conclude that each of $X\cap Y^{\ast}$ and $X^{\ast}\cap Y^{\ast}$ is a
nontrivial $L$--almost invariant subset of $G$ contained in $Y^{\ast}$. In
particular, they cannot be enclosed by {$v$, which is a contradiction. }This
contradiction shows that $G(f)$ must be $VPC(n+1)$, so that the associated
splitting of $G$ is dual to a torus, as required. It remains to show that $X$
crosses this torus.

As $G(f)$ determines a torus, it follows that the component $\Sigma$ of
$\partial\widetilde{M}$ with stabilizer $S$ cannot meet this torus, and so
must lie on the same side of $\widetilde{M}_{f}$ as does $\widetilde{M}_{v}$.
In particular, $S\subset Y$. As $X\cap S$ and $X^{\ast}\cap S$ are both
$H$--infinite, it follows that $X\cap Y$ and $X^{\ast}\cap Y$ are both
$H$--infinite. Hence all four corners of the pair $(X,Y)$ are $H$--infinite,
so that $X$ crosses $Y$, as required.
\end{proof}

Next we will show that $H$ and $G(f)$ must be commensurable subgroups of $G$.

We split $G$ along the torus $G(f)$ to obtain a new $PD(n+2)$ pair
$(G^{\prime},\partial G^{\prime})$ (by Theorem 8.1 of \cite{BE01}) with
$(G^{\prime},\partial G^{\prime})$ containing $G(w)$, where $v$ and $w$ are
the vertices of $f$. Correspondingly, $M$ is split along $M_{f}$ to obtain a
new space $N$ containing $M_{w}$. Thus, $\widetilde{M}$ is split along
$\widetilde{M}_{f}$ and its translates to obtain a new space $\widetilde{N}$
containing $\widetilde{M}_{w}$. In particular, the boundary of $\widetilde{N}$
consists of boundary components of $\widetilde{M}$ together with translates of
$\widetilde{M}_{f}$.

\begin{lemma}
\label{thereisannulusA}Using the notation of Lemma \ref{Xcrossestorus},
suppose that $H$ and $G(f)$ are not commensurable, and let $L^{\prime}$ denote
$H\cap G(f)$. Then $L^{\prime}$ is $VPCn$ and contains $L$ with finite index,
and there is an essential annulus in $N$, carrying $L^{\prime}$, which lifts
to an annulus $A$ in $P_{L^{\prime}}(\tilde{N})$ from $P_{L^{\prime}%
}(\widetilde{M}_{f})$ to a component of $P_{L^{\prime}}(\partial\tilde{N})$.
\end{lemma}

\begin{proof}
Note that part 2) of Lemma \ref{LisVPCn} tells us that {$G(v)=Comm_{G}%
(L)=N_{G}(L)$ contains }$H$. As $v$ is of Seifert type or of torus type, and
$G(f)$ is a boundary torus of $G(v)$, it follows that $G(f)$ also contains
$L$. In\ particular, the intersection $H\cap G(f)$ contains $L$. As $H$ and
$G(f)$ are not commensurable, it follows that $H\cap G(f)=L^{\prime}$ is
$VPCn$ and contains $L$ with finite index.

As in the proof of Lemma \ref{Xcrossestorus}, we consider the intersections
$X\cap Y^{\ast}$ and $X^{\ast}\cap Y^{\ast}$. Both sets are invariant under
$H\cap G(f)=L^{\prime}$. Again we know that $\delta X\cap Y^{\ast}$ must be
$L$--finite. Suppose that $X\cap\delta Y$ is $L$--finite. Then $X\cap Y^{\ast
}$ has $L$--finite coboundary and so $X\cap Y^{\ast}$ is a nontrivial
$L$--almost invariant set which is not enclosed by $v$, which is again a
contradiction. Thus $X\cap\delta Y$ must be $L$--infinite, and similarly
$X^{\ast}\cap\delta Y$ must be $L$--infinite. Note that the intersections of
$X\cap Y^{\ast}$ and $X^{\ast}\cap Y^{\ast}$ with the $0$--skeleton of
$\widetilde{N}$ have coboundaries (in $\widetilde{N}$) equal to $\delta X\cap
Y^{\ast}$ and $\delta X^{\ast}\cap Y^{\ast}$ respectively, each of which is
$L$--finite. As $X\cap Y^{\ast}$ and $X^{\ast}\cap Y^{\ast}$ contain
$X\cap\delta Y^{\ast}$ and $X^{\ast}\cap\delta Y^{\ast}$ respectively which
are both $L^{\prime}$--infinite, it follows that each determines a nontrivial
$L^{\prime}$--almost invariant subset of $G^{\prime}$. Hence there are
essential annuli in $N$, carrying $L^{\prime}$, one of which lifts to an
annulus $A$ in $P_{L^{\prime}}(\tilde{N})$ from $P_{L^{\prime}}(\widetilde{M}%
_{f})$ to a component of $P_{L^{\prime}}(\partial\tilde{N})$, as required.
\end{proof}

\begin{lemma}
\label{HandG(f)arecommensurable}Using the notation of Lemma
\ref{Xcrossestorus}, the subgroups $H$ and $G(f)$ of $G$ are commensurable.
\end{lemma}

\begin{proof}
Suppose that $H$ and $G(f)$ are not commensurable, and let $L^{\prime}=H\cap
G(f)$. By Lemma \ref{thereisannulusA}, there is an essential annulus $A$ in
$P_{L^{\prime}}(\tilde{N})$ from $P_{L^{\prime}}(\widetilde{M}_{f})$ to a
component $\Sigma_{L^{\prime}}$ of $P_{L^{\prime}}(\partial\tilde{N})$.

As $A$ is essential, $P_{L^{\prime}}(\widetilde{M}_{f})$ and $\Sigma
_{L^{\prime}}$ must be distinct components of $P_{L^{\prime}}(\partial
\widetilde{N})$. Recall that $\widetilde{N}$ contains $\widetilde{M}_{w}$,
where $w$ is the $V_{1}$-vertex of the edge $f$. It follows that $A$ has a
sub-annulus $A^{\prime}$ which lies in $P_{L^{\prime}}(\widetilde{M}_{w})$,
and joins distinct boundary components. Thus the vertex $w$ has an incident
edge $g$, distinct from $f$, such that $G(g)$ contains $L^{\prime}$. As $G(g)$
is an edge group of $\Gamma_{n,n+1}^{c}(G)$, it must be $VPCn$ or $VPC(n+1)$.
In either case, we apply Proposition \ref{rem24}. Note that $T$ is the
universal covering $G$--tree of the graph of groups $\Gamma_{n,n+1}^{c}(G),$
not of $\Gamma_{n,n+1}(G)$, so this proposition is applicable. As $f$ and $g$
are distinct, case 1) of the conclusion is not possible. It follows that $w$
is an isolated vertex of $\Gamma_{n,n+1}^{c}(G)$. In particular, $G(g)=G(f)$
is $VPC(n+1)$. Let $v^{\prime}$ denote the $V_{0}$-vertex at the other end of
the edge $g$. The part $A_{1}$ of $A$ in $M_{v^{\prime}}$ is an essential
annulus carrying $L^{\prime}$ in the pair $(G(v^{\prime}),\partial
G(v^{\prime}))$. We need to recall from Theorem \ref{thm21} the possible types
of $V_{0}$--vertex of $\Gamma_{n,n+1}^{c}(G)$.

If $v^{\prime}$ is isolated, this would contradict the fact that
$\Gamma_{n,n+1}^{c}(G)$ is reduced.

If $v^{\prime}$ is of $VPC(n-1)$--by--Fuchsian type, and is of $I$--bundle
type, then each edge splitting for edges incident to $v^{\prime}$ would be
dual to an annulus. As the edge splitting dual to $g$ is a torus, this case
cannot occur.

If $v^{\prime}$ is of $VPCn$--by--Fuchsian type, and is of interior Seifert
type, then the essential annulus $A_{1}$ in $(G(v^{\prime}),\partial
G(v^{\prime}))$ projects to an annulus in the base $2$--orbifold. As
$v^{\prime}$ is of $VPCn$--by--Fuchsian type, the fundamental group of this
base orbifold is not virtually cyclic. It follows that it does not admit an
essential annulus. We conclude that this projected annulus is inessential,
which implies that $L^{\prime}$ is commensurable with the $VPCn$ fibre group
of $G(v^{\prime})$. But this means that $G(v^{\prime})$ commensurises
$L^{\prime}$, and hence commensurises $L$, which is again a contradiction as
$G(v)=Comm_{G}(L)$.

Finally, if $v^{\prime}$ is of commensuriser type, the existence of the
essential annulus $A_{1}$ in $(G(v^{\prime}),\partial G(v^{\prime}))$ implies
that $G(v^{\prime})$ commensurises $L^{\prime}$, and hence commensurises $L$,
which is again a contradiction as $G(v)=Comm_{G}(L)$.

We have shown that all cases lead to a contradiction so that $H$ and $G(f)$
must be commensurable, as required.
\end{proof}

Combining the preceding lemmas and Remark \ref{torustypeimpliestype2)}, we
have proved the following.

\begin{lemma}
\label{summary}Let $(G,\partial G)$ be an orientable $PD(n+2)$ pair such that
$G\ $is not $VPC$, let $T$ denote the universal covering $G$--tree of the
graph of groups $\Gamma_{n,n+1}^{c}(G)$, and let $e$ be an edge of $T$ such
that the associated splitting of $G$ is not canonical. Then the following
statements hold:

\begin{enumerate}
\item There is an almost invariant set $X$ over a $VPC(n+1)$ subgroup $H$ of
$G$ which is not adapted to $\partial G$ and crosses $e$.

\item There is a group $S$ with a conjugate in $\partial G$ such that $L=H\cap
S$ is $VPCn$, and $X\cap S$ and $X^{\ast}\cap S$ are both $H$--infinite.

\item There is a $V_{0}$--vertex $v$ of $T$ of commensuriser type which
encloses all $L$--almost invariant subsets of $G$. Further, $\widetilde{M}%
_{v}\cap\partial\widetilde{M}$ is non-empty, ${v}$ is of Seifert type or of
toral type 2), and {$G(v)=Comm_{G}(L)=N_{G}(L)$ contains }$H$.

\item If $f$ is the first edge on the path from $v$ to $e$, then $G(f)$ is
$VPC(n+1)$ and commensurable with $H$, and $X$ crosses the torus splitting
given by $f$.
\end{enumerate}
\end{lemma}

Now we can complete the proof of Theorem \ref{mainresult}, that the edge
splittings of $\Gamma_{n,n+1}(G)$ and of $\Gamma_{n,n+1}^{c}(G)$ are either
canonical or are special canonical tori.

\begin{proof}
Let $e$ be an edge of $T$ such that the associated splitting of $G$ is not
canonical, and apply Lemma \ref{summary}. Let $Y$ and $Y^{\ast}$ denote the
almost invariant sets associated to $f$, chosen so that $G(v)\subset Y$.

As $H$ and $G(f)$ are commensurable, and $X$ crosses $Y$, the intersections
$X\cap Y^{\ast}$ and $X^{\ast}\cap Y^{\ast}$ are nontrivial almost invariant
sets over the $VPC(n+1)$ group $H^{\prime}=H\cap G(f)$.

As $H^{\prime}$ is a torus in $(G,\partial G)$, up to equivalence, we have
only two $H^{\prime}$--almost invariant sets which are adapted to $\partial
G$, namely $Y$ and $Y^{\ast}$. Thus neither of $X\cap Y^{\ast}$ and $X^{\ast
}\cap Y^{\ast}$ is adapted to $\partial G$. Let $Z$ denote $X\cap Y^{\ast}$.
Then Lemma \ref{LisVPCn} tells us that there is a group $S^{\prime}$ in
$\partial G$ such that $Z\cap S^{\prime}$ and $Z^{\ast}\cap S^{\prime}$ are
both $H^{\prime}$--infinite, and $H^{\prime}\cap S^{\prime}$ is a $VPCn$ group
$K$. Further there is a $V_{0}$--vertex $v^{\prime}$ of $T$, which encloses
all $K$--almost invariant subsets of $G$ so that $G(v^{\prime})$ contains
$H^{\prime}$, and $\widetilde{M}_{v^{\prime}}$ intersects $\partial
\widetilde{M}$.

If $Z$ crosses some edge $e^{\prime}$ of $T$, Lemma \ref{Xcrossestorus} tells
us that if $f^{\prime}$ is the first edge on the path from $v^{\prime}$ to
$e^{\prime}$, then $G(f^{\prime})$ is $VPC(n+1)$ and commensurable with
$H^{\prime}$, and $Z$ crosses the torus splitting given by $f^{\prime}$. In
particular, $f$ and $f^{\prime}$ have commensurable stabilizers. Thus we can
apply Lemma \ref{paralleledges} to deduce that $G(f)=G(f^{\prime})$. Cases 1)
or 3) of that lemma would imply that $v=v^{\prime}$, so that $Z$ is enclosed
by $v$. But this is impossible as $Z\subset Y^{\ast}$, and $G(v)\subset Y$.
Thus we must have cases 2) or 4) of Lemma \ref{paralleledges}. Further, as $v$
and $v^{\prime}$ are not isolated, in case 2), $f\cap f^{\prime}$ must be a
$V_{1}$--vertex.

If $Z$ crosses no edge of $T$, then $Z$ is enclosed by some $V_{0}$--vertex
$v^{\prime}$, which again cannot be $v$. Thus $H^{\prime}$ is a subgroup of
$G(v)\ $and of $G(v^{\prime})$, and hence of the edge $f^{\prime}$ incident to
$v^{\prime}$ and on the path in $T$ from $v$ to $v^{\prime}$. Again we must
have cases 2) or 4) of Lemma \ref{paralleledges}, and in case 2), $f\cap
f^{\prime}$ must be a $V_{1}$--vertex.

Now Proposition \ref{characterizespecialcanonical} implies that in all cases,
the splitting determined by $f$ is a special canonical torus, and so is the
splitting determined by $f^{\prime}$ (and the splittings determined by $b$ and
$b^{\prime}$ in case 4) of Lemma \ref{paralleledges}).

Finally, we will show that the original edge $e$ that was crossed by $X$ must
equal one of $f$ or $f^{\prime}$ (or $b$ or $b^{\prime}$ in case 4) of Lemma
\ref{paralleledges}). In all cases, it follows that the splitting determined
by $e$ is a special canonical torus.

Recall that $Z=X\cap Y^{\ast}$, and that $Z\cap S^{\prime}$ and $Z^{\ast}\cap
S^{\prime}$ are both $H^{\prime}$--infinite. We claim that $X\cap S^{\prime}$
and $X^{\ast}\cap S^{\prime}$ are also both $H^{\prime}$--infinite. As $X\cap
S^{\prime}$ contains $Z\cap S^{\prime}$, the first part of the claim is clear.
As $G(f)$ determines a torus, it follows that the component of $\partial
\widetilde{M}$ with stabilizer $S^{\prime}$ cannot meet this torus, and so
must lie on the same side of $\widetilde{M}_{f}$ as does $\widetilde{M}%
_{v^{\prime}}$. In particular, $S^{\prime}\subset Y^{\ast}$, so that $Y\cap
S^{\prime}$ is $H^{\prime}$--finite. Now $Z^{\ast}=X^{\ast}\cup(X\cap Y)$, so
it follows that $X^{\ast}\cap S^{\prime}$ is also $H^{\prime}$--infinite,
completing the proof of the claim.

Next we show that $f$ is the only edge incident to $v$ which is crossed by
$X$. For suppose that $X\ $crosses an edge $f^{\prime\prime}$ incident to $v$.
Lemma \ref{summary} shows that $G(f^{\prime\prime})$ is $VPC(n+1)$ and
commensurable with $H$. Now we apply Lemma \ref{paralleledges} to the pair
$(f,f^{\prime\prime})$. Thus $G(f)=G(f^{\prime\prime})$, and we must have case
1), 2), 3) or 4) of that lemma. As $f$ and $f^{\prime\prime}$ have the same
$V_{0}$--vertex which is not isolated nor of special Seifert type, this is
impossible unless $f=f^{\prime\prime}$.

Next suppose that $X\ $crosses some edge $e^{\prime\prime}$ of $T$. Note that
$X$ is $H^{\prime}$--almost invariant, that $K=H^{\prime}\cap S^{\prime}$ is
$VPCn$, and $X\cap S$ and $X^{\ast}\cap S$ are both $H^{\prime}$--infinite.
Further $v^{\prime}$ is a $V_{0}$--vertex of $T$ of commensuriser type which
encloses all $K$--almost invariant subsets of $G$. Now we apply Lemma
\ref{Xcrossestorus} with $v^{\prime}$ and $e^{\prime\prime}$ in place of $v$
and $e$. This shows that if $f^{\prime\prime}$ is the first edge on the path
from $v^{\prime}$ to $e^{\prime\prime}$, then $G(f^{\prime\prime})$ is
$VPC(n+1)$ and commensurable with $H^{\prime}$, and $X$ crosses the torus
splitting given by $f^{\prime\prime}$. Now we can argue as in the preceding
paragraph to show that $f^{\prime}$ is the only edge incident to $v^{\prime}$
which is crossed by $X$.

As $X$ crosses the edge $e$, we conclude that $f$ is the first edge of the
path joining $v$ to $e$, and that $f^{\prime}$ is the first edge of the path
joining $v^{\prime}$ to $e$. It follows that $e$ must lie between $v$ and
$v^{\prime}$, so that the edge $e$ of $T$ must be equal to $f$ or $f^{\prime}$
(or $b$ or $b^{\prime}$), as required. It follows that each edge splitting of
$\Gamma_{n,n+1}^{c}(G)$ is either canonical or is a special canonical torus,
thus completing the proof of Theorem \ref{mainresult}.
\end{proof}

The following result is an easy consequence of the above arguments. Recall
that a $H$--almost invariant subset $X$ of a group $G$ is enclosed by a vertex
$v$ of a $G$--tree $T$, if for every edge $e$ incident to $v$, we have either
$X\leq Z_{e}$ or $X^{\ast}\leq Z_{e}$, where $Z_{e}$ and $Z_{e}^{\ast}$ are
the almost invariant subsets of $G$ associated to $e$ chosen so that $v$ lies
in $Z_{e}$. There is a natural extension of this idea as follows. If
$T^{\prime}$ is a subtree of $T$, we will say that $X$ is enclosed by
$T^{\prime}$, if for every edge $e$ incident to $T^{\prime}$, but not
contained in $T^{\prime}$, we have either $X\leq Z_{e}$ or $X^{\ast}\leq
Z_{e}$, where $Z_{e}$ and $Z_{e}^{\ast}$ are the almost invariant subsets of
$G$ associated to $e$ chosen so that $T^{\prime}$ is contained in $Z_{e}$.

\begin{proposition}
\label{anya.i.setiscontainedinoneortwovertices}Let $(G,\partial G)$ be an
orientable $PD(n+2)$ pair such that $G$ is not $VPC$. Then any almost
invariant subset of $G$ over a $VPC(n+1)$ subgroup $H$ is either enclosed by a
$V_{0}$--vertex of $T$, or is enclosed by an interval whose endpoints are the
$V_{0}$--vertices on opposite sides of a special canonical torus edge. Further
$H$ is commensurable with that splitting torus group.
\end{proposition}

\begin{proof}
Let $X$ be an almost invariant set over a $VPC(n+1)$ subgroup $H$ of $G$. If
$X$ is adapted to $\partial G$, then $X$ is enclosed by some $V_{0}$--vertex.
If $X$ is not adapted to $\partial G$, there are two cases depending on
whether or not it crosses some edge of $T$. If it crosses no edge of $T$, then
$X$ is enclosed by some vertex, and hence by a $V_{0}$--vertex. If $X$ crosses
some edge of $T$, we apply all the results above. This yields two $V_{0}%
$--vertices $v$ and $v^{\prime}$ on opposite sides of a special canonical
torus edge, and $X$ crosses no edges of $T$ apart from the edges $f$ and
$f^{\prime}$ ( and $b$ and $b^{\prime}$ in case 4) of Lemma
\ref{paralleledges}) between $v$ and $v^{\prime}$. Further $H$ is
commensurable with $G(f)$. Thus $X$ is enclosed by the interval whose
endpoints are $v$ and $v^{\prime}$, the edges of this interval all have
associated the same special canonical torus, and $H$ is commensurable with
that splitting torus group, as required.
\end{proof}

\section{Comparisons\label{section:comparisons}}

As promised in the introduction, we can now give the comparisons of the JSJ
decomposition of a $PD(n+2)$ pair $(G,\partial G)$ with two other naturally
defined decompositions. The easiest to handle is the algebraic regular
neighbourhood of the set $\mathcal{E}_{n,n+1}(G)$ (the collection of all a.i.
subsets of $G$ which are over a $VPCn$ or $VPC(n+1)$ subgroup). Note that it
is not at all obvious that this family of almost invariant subsets of $G$ has
a regular neighbourhood. Such a regular neighbourhood does not exist for
general groups, as discussed in \cite{SS03}. However Proposition
\ref{anya.i.setiscontainedinoneortwovertices} implies that the decomposition
of $G$ obtained by collapsing to a point each interval whose endpoints are the
$V_{0}$--vertices on opposite sides of a special canonical torus edge is the
regular neighbourhood of $\mathcal{E}_{n,n+1}(G)$. We have shown the following result.

\begin{theorem}
Let $(G,\partial G)$ be an orientable $PD(n+2)$ pair such that $G$ is not
$VPC$, and let $\mathcal{E}_{n,n+1}(G)$ denote the collection of all a.i.
subsets of $G$ which are over a $VPCn$ or $VPC(n+1)$ subgroup. Then the
regular neighbourhood of $\mathcal{E}_{n,n+1}(G)$ in $G$ exists and is
obtained from $\Gamma_{n,n+1}(G)$ by collapsing to a point each interval whose
endpoints are the $V_{0}$--vertices on opposite sides of a special canonical
torus edge.
\end{theorem}

\begin{remark}
If $\partial G$ is empty, the decomposition $\Gamma_{n+1}(G)$ is equal to the
regular neighbourhood of $\mathcal{E}_{n,n+1}(G)$ in $G$, as $\mathcal{E}%
_{n,n+1}(G)$ is equal to the collection of all a.i. subsets of $G$ which are
over a $VPC(n+1)$ subgroup.
\end{remark}

Next we turn to the analogue of the topological decomposition of
$3$--manifolds obtained in \cite{NS}. One considers the family $\mathcal{S}%
_{n,n+1}(G)$, which consists of all a.i. subsets of $G$ which are dual to
splittings of $G$ over annuli or tori in $(G,\partial G)$. We claim that this
family also has a regular neighbourhood, and that the edge splittings are
those in $\mathcal{S}_{n,n+1}(G)$ which cross no element of $\mathcal{S}%
_{n,n+1}(G)$. As in \cite{NS}, we call this the Waldhausen decomposition or
W--decomposition. Again it is not at all obvious that this family of almost
invariant subsets of $G$ has a regular neighbourhood. Such a regular
neighbourhood does not exist for general groups, as discussed in \cite{SS03}.

We will say that an element of the family $\mathcal{S}_{n,n+1}(G)$ which
crosses no element of $\mathcal{S}_{n,n+1}(G)$ is \textit{isolated} in
$\mathcal{S}_{n,n+1}(G)$. We start by describing the isolated elements of
$\mathcal{S}_{n,n+1}(G)$, and showing that there are only finitely many such
elements. Trivially the edge splittings of $\Gamma_{n,n+1}(G)$ are all
isolated elements of $\mathcal{S}_{n,n+1}(G)$. We have the following result.

\begin{lemma}
\label{alphaisenclosedbycommensurisertypevertex}Let $(G,\partial G)$ be an
orientable $PD(n+2)$ pair such that $G$ is not $VPC$, and let $\alpha$ be an
isolated element of $\mathcal{S}_{n,n+1}(G)$, which is not an edge splitting
of $\Gamma_{n,n+1}(G)$. Then $\alpha$ is enclosed by a $V_{0}$--vertex of
$\Gamma_{n,n+1}(G)$ of commensuriser type.
\end{lemma}

\begin{proof}
As $\alpha$ is a splitting of $G$ dual to an annulus or torus, it must be
enclosed by some $V_{0}$--vertex $v$ of $\Gamma_{n,n+1}(G)$. Theorem
\ref{thm21} tells us that a $V_{0}$--vertex of $\Gamma_{n,n+1}(G)$ must be
isolated, of $I$--bundle type, of interior Seifert type, or of commensuriser
type. Thus it suffices to show that the first three cases cannot occur.

If $v$ is isolated, any splitting enclosed by $v$ is equal to an edge
splitting of $\Gamma_{n,n+1}(G)$. Hence this case cannot occur.

If $v$ is of $I$--bundle type, and so of $VPC(n-1)$--by--Fuchsian type, then
$\alpha$ must be a splitting dual to an annulus. Let $K$ denote the $VPCn$
group carried by the splitting annulus, and let $X_{v}$ denote the base
$2$--orbifold of $v$. Then the image of $K$ in the fundamental group of
$X_{v}$ is $VPC(\geq1)$. As a Fuchsian group cannot have a $VPC2$ subgroup, it
follows that the image of $K$ must be $VPC1$. As $\alpha$ is not an edge
splitting of $\Gamma_{n,n+1}(G)$, and crosses no such splitting, it determines
a splitting of $G(v)$ which is adapted to $\partial_{1}v$. This then yields a
splitting over a $VPC1$ subgroup of the fundamental group of $X_{v}$, which is
adapted to the boundary $\partial X_{v}$. As discussed in section 5.1.2 of
\cite{GL}, this splitting is dual to a "simple closed curve" $\gamma$ in
$X_{v}$, meaning that $\gamma$ is a connected, closed $1$--dimensional
suborbifold of $X_{v}$. Thus either $\gamma$ is a circle or it is the quotient
of a circle by a reflection involution. The assumption that $\alpha$ is not an
edge splitting of $\Gamma_{n,n+1}(G)$, means that $\gamma$ cannot be homotopic
to a boundary curve of $X_{v}$. Hence there is another "simple closed curve"
$\delta$ in $X_{v}$ whose intersection number with $\gamma$ is non-zero
(Corollary 5.10 of \cite{GL}). This determines a splitting of $G$ dual to an
annulus whose intersection number with $\alpha$ is non-zero, contradicting our
assumption that $\alpha$ crosses no element of $\mathcal{S}_{n,n+1}(G)$. We
conclude that this case cannot occur.

If $v$ is of interior Seifert type, and so of $VPCn$--by--Fuchsian type, then
$\alpha$ must be a splitting dual to a torus. As in the preceding paragraph,
this torus yields an essential "simple closed curve" in the base $2$--orbifold
of $v$. As in that paragraph, this implies that there is a splitting of $G$
dual to a torus whose intersection number with $\alpha$ is non-zero,
contradicting our assumption that $\alpha$ crosses no element of
$\mathcal{S}_{n,n+1}(G)$. We conclude that this case also cannot occur, which
completes the proof of the lemma.
\end{proof}

Theorem \ref{thm21} tells us that if $v$ is a $V_{0}$--vertex of
$\Gamma_{n,n+1}(G)$ of commensuriser type, then $v$ is of peripheral Seifert
type, or of torus type, or of solid torus type. In the penultimate section of
\cite{SS05}, the authors discussed the structure of such vertices in great
detail. They showed that in each of these cases, $G(v)$ is $VPCn$%
--by--$\Gamma$, where $\Gamma$ is the fundamental group of a compact
$2$--dimensional orbifold $X_{v}$. The group $\Gamma$ is finite if $v$ is of
solid torus type, is virtually infinite cyclic if $v$ is of torus type, and is
not virtually cyclic if $v$ is of Seifert type. Recall that any vertex $v$ of
$\Gamma_{n,n+1}$\ has two types of "boundary" subgroups. The first type comes
from the edge groups of the decomposition and the family of all these
subgroups will be denoted by $\partial_{1}v$. The second type comes from the
decomposition of $\partial G$ by edges of the decompositions and this family
will be denoted by $\partial_{0}v$. The first type gives us $PD(n+1)$\ pairs
in $(G,\partial G)$, namely annuli or tori, and the second type gives us
$PD(n+1)$\ pairs which are contained in $\partial G$. If $v$ is of
commensuriser type, these families of subgroups determine a division of the
boundary $\partial X_{v}$ of $X_{v}$ into suborbifolds $\partial_{0}X_{v}$ and
$\partial_{1}X_{v}$, where $\partial_{0}X_{v}$ equals the closure of $\partial
X_{v}-\partial_{1}X_{v}$. Note that $\partial_{0}X_{v}$ must be non-empty. It
is possible that $\partial_{1}X_{v}$ may be empty, but this happens if and
only if $\Gamma_{n,n+1}(G)$ consists of the single vertex $v$, so that
$G=G(v)$. Next the authors of \cite{SS05} show that one can double $G(v)$
along $\partial_{0}v$ which is "the intersection of $G(v)$ with $\partial G$".
The new object $DG(v)$ is the fundamental group of a $V_{0}$--vertex of
$\Gamma_{n+1}(DG)$. It is $VPCn$--by--$D\Gamma$, where $D\Gamma$ is the
fundamental group of $DX_{v}$, the double of $X_{v}$ along $\partial_{0}X_{v}%
$. As we are assuming that $G\ $is not $VPC$, it follows that $DG$ is also not
$VPC$, so that $\pi_{1}^{orb}(DX_{v})$ cannot contain a $VPC2$ subgroup.

In the proof of Lemma \ref{alphaisenclosedbycommensurisertypevertex}, we used
the close connection between splittings of $G\ $over annuli or tori enclosed
by a $V_{0}$--vertex $v$ and "simple closed curves" in the base $2$--orbifold
$X_{v}$. Now we will extend these ideas. We will need the idea of a "simple
arc" $\lambda$ in the pair $(X_{v},\partial_{0}X_{v})$. This means that
$\lambda$ is a connected $1$--dimensional suborbifold of $X_{v}$ with
non-empty boundary, so that either $\lambda$ is an interval or it is the
quotient of an interval by a reflection involution. Further $\lambda$ has
boundary contained in $\partial_{0}X_{v}$.

We will say that a "simple closed curve" $\gamma$ in the $2$--orbifold $X_{v}$
is \textit{essential in} $(X_{v},\partial_{0}X_{v})$ if $\pi_{1}^{orb}%
(\gamma)$ injects into $\pi_{1}^{orb}(X_{v})$ and $\gamma$ is not homotopic
into $\partial_{0}X_{v}$ or into $\partial_{1}X_{v}$. Also we will say that a
"simple arc" $\lambda$ in $(X_{v},\partial_{0}X_{v})$ is \textit{essential in}
$(X_{v},\partial_{0}X_{v})$ if $\lambda$ cannot be homotoped into
$\partial_{0}X_{v}$ nor into $\partial_{1}X_{v}$ while keeping $\partial
\lambda$ in $\partial_{0}X_{v}$.

\begin{lemma}
\label{bijectionbetweencurvesandtori}Let $(G,\partial G)$ be an orientable
$PD(n+2)$ pair such that $G$ is not $VPC$, and let $\alpha$ be an isolated
element of $\mathcal{S}_{n,n+1}(G)$, which is not an edge splitting of
$\Gamma_{n,n+1}(G)$, and is enclosed by a $V_{0}$--vertex $v$ of
$\Gamma_{n,n+1}(G)$ of commensuriser type, and with base $2$--orbifold $X_{v}%
$. Then the following hold:

\begin{enumerate}
\item If $\alpha$ is a splitting of $G$ dual to a torus, it determines a
"simple closed curve" $C$ in $X_{v}$ which is essential in $(X_{v}%
,\partial_{0}X_{v})$, and this yields a bijection between splittings of
$G\ $dual to a torus which are enclosed by $v$ and not an edge torus of $v$,
and "simple closed curves" in $X_{v}$ which are essential in $(X_{v}%
,\partial_{0}X_{v})$.

\item If $\alpha$ is a splitting of $G$ dual to an annulus, it determines a
"simple arc" $\lambda$ in $(X_{v},\partial_{0}X_{v})$ which is essential in
$(X_{v},\partial_{0}X_{v})$, and this yields a bijection between splittings of
$G\ $dual to an annulus which are enclosed by $v$ and not an edge annulus of
$v$, and "simple arcs" in $(X_{v},\partial_{0}X_{v})$ which are essential in
$(X_{v},\partial_{0}X_{v})$.
\end{enumerate}
\end{lemma}

\begin{proof}
1) Note that although $\alpha$ is not an edge splitting of $\Gamma_{n,n+1}%
(G)$, it need not be the case that it determines a splitting of $G(v)$. This
is clear if $v$ is of torus type, but the same difficulty can arise if $v$ is
of Seifert type. We resolve this problem by working with the double $DG(v)$
which is the fundamental group of a $V_{0}$--vertex $V$ of $\Gamma_{n+1}(DG)$.
Now $\alpha$ gives a splitting of $DG$ dual to a torus which is enclosed by
$V$, and is not an edge splitting of $\Gamma_{n+1}(DG)$. As in the proof of
Lemma \ref{alphaisenclosedbycommensurisertypevertex}, this splitting is dual
to a "simple closed curve" $\gamma$ in $DX_{v}$ which cannot be homotopic to a
boundary curve of $DX_{v}$. As $\alpha$ is enclosed by $v$, it follows that
$\gamma$ can be chosen to lie in $X_{v}$. Further $\gamma$ must be essential
in $(X_{v},\partial_{0}X_{v})$, as $DX_{v}$ is the double of $X_{v}$ along
$\partial_{0}X_{v}$. Now reversing the arguments yields the required bijection.

2) Let $A$ denote the annulus in $(G,\partial G)$ which induces the splitting
$\alpha$ of $G$, let $DA$ denote the torus in $DG$ obtained by doubling $A$
along $\partial A$, and let $D\alpha$ denote the splitting of $DG$ induced by
$DA$. As $\alpha$ is not an edge splitting of $\Gamma_{n,n+1}(G)$, it follows
that $D\alpha$ is not an edge splitting of $\Gamma_{n+1}(DG)$. Now as in part
1), $Da$ determines a splitting of $DG(v)$ dual to the torus $DA$ which yields
a "simple closed curve" $\gamma$ in $DX_{v}$ which cannot be homotopic to a
boundary curve of $DX_{v}$. This splitting of $DG(v)$ is invariant under the
involution interchanging the two copies of $G(v)$, so the curve $\gamma$ can
be chosen to be invariant under the involution of $DX_{v}$ interchanging the
two copies of $X_{v}$. Thus $\gamma$ is the double of a "simple arc" $\lambda$
in the base orbifold $X_{v}$. Further $\lambda$ has boundary contained in
$\partial_{0}X_{v}$, reflecting the fact that the boundary of the annulus lies
in $\partial G$. As $\gamma$ is not homotopic to a boundary curve of $DX_{v}$,
it follows that $\lambda$ is essential in $(X_{v},\partial_{0}X_{v})$. One can
easily reverse this argument to obtain the required bijection.
\end{proof}

In the preceding lemma, as $\alpha$ crosses no element of $\mathcal{S}%
_{n,n+1}(G)$, it follows that $\gamma\ $and $\lambda$ cross no simple closed
curve in $X_{v}$ which is essential in $(X_{v},\partial_{0}X_{v})$, and cross
no essential simple arc in $(X_{v},\partial_{0}X_{v})$. We will say that such
$\gamma$ and $\lambda$ are \textit{isolated}.

Note that in the above discussions, any compact $2$--orbifold with non-empty
boundary can occur as $X_{v}$, which is not the case in the setting of
$3$--manifolds. See Lemma \ref{nomirrors} and the discussion at the end of
section \ref{section:dimension3}. Note also that in general it is possible
that an arc or closed curve in $X_{v}$ can be essential in $(X_{v}%
,\partial_{0}X_{v})$, but inessential in $(X_{v},\partial X_{v})$, meaning
that it can be homotoped into $\partial X_{v}$ while keeping $\partial\lambda$
in $\partial X_{v}$.

\begin{lemma}
\label{essentiality}Let $X$ be a compact $2$--orbifold, and let $\partial
_{1}X$ denote a possibly empty compact suborbifold of $\partial X$. Let
$\partial_{0}X$ denote the closure of $\partial X-\partial_{1}X$.

\begin{enumerate}
\item If the Euler characteristic $\chi(X)\leq0$, and $\lambda$ is an isolated
essential "simple arc" in $(X,\partial_{0}X)$, then $\lambda$ is essential in
$(X,\partial X)$.

\item If $\chi(X)\leq0$, and $C$ is an isolated "simple closed curve" in $X$
which is essential in $(X,\partial_{0}X)$, then $C$ is essential in
$(X,\partial X)$.

\item If $\chi(X)>0$, there cannot be any "simple closed curve" in $X$ which
is essential in $(X,\partial_{0}X)$. Also there cannot be any isolated "simple
arc" in $X$ which is essential in $(X,\partial_{0}X)$.
\end{enumerate}
\end{lemma}

\begin{remark}
Recall that here a "simple closed curve" is either a circle or the orbifold
quotient of the circle by a reflection, and a similar comment applies to the
phrase "simple arc". In each case, the quotient by a reflection is a
non-orientable $1$--orbifold
\end{remark}

\begin{proof}
1) Suppose that $\lambda$ is orientable, so that $\partial\lambda$ consists of
two points, and is not essential in $(X,\partial X)$. Thus $\lambda$ is
parallel to an arc $\mu$ contained in some component $C$ of $\partial X$. Of
course, the ends of $\mu$ lie in $\partial_{0}X$. As $\lambda$ is essential in
$(X,\partial_{0}X)$, the arc $\mu$ cannot be contained in $\partial_{0}X$.
There must be at least two components of $\partial_{1}X$ in the interior of
$\mu$, so there must be at least one component $D$ of $\partial_{0}X$ in the
interior of $\mu$. There can be no mirrors in $X$, as otherwise we could join
$D$ to a mirror to obtain an essential "simple arc" in $(X,\partial_{0}X)$
which crosses $\lambda$. As $X$ has no mirrors, it follows that $C$ is a
circle. But now there is an arc $\lambda^{\prime}$ with both ends in $D$ which
is parallel to an arc $\mu^{\prime}$ in $C$, such that $\mu\cup\mu^{\prime}%
=C$, and $\lambda^{\prime}$ is also essential in $(X,\partial_{0}X)$ and
crosses $\lambda$. This contradicts the hypothesis that $\lambda$ is isolated,
which proves the required result in the case when $\lambda$ is orientable.

Next suppose that $\lambda$ is not orientable, so that $\partial\lambda$
consists of one point, and is not essential in $(X,\partial X)$. Thus
$\lambda$ is homotopic to an isomorphic $1$--orbifold $\mu$ contained in some
component $C$ of $\partial X$. Note that the reflector point of $\mu$ must be
a reflector point of $C$, which must be an intersection point of $C$ with a
mirror component $m$ of $X$. Of course, $\partial\mu$ lies in $\partial_{0}X$.
As $\lambda$ is essential in $(X,\partial_{0}X)$, the orbifold $\mu$ cannot be
contained in $\partial_{0}X$, nor in $\partial_{1}X$. It follows that there
must be a component $D$ of $\partial_{0}X$ in $\mu$ other than the component
which contains $\partial\mu$. Note that $D$ may contain the reflector point of
$\mu$. If $X\ $has a mirror other than $m$, we could join $D$ to such a mirror
to obtain an essential "simple arc" in $(X,\partial_{0}X)$ which crosses
$\lambda$. It follows that $m$ is the only mirror in $X$. In particular,
$C\ $and $m$ must together form a boundary component of the surface underlying
$X$. But now there is an arc $\lambda^{\prime}$ with both ends in $D$ which is
parallel to an arc $\mu^{\prime}$ of $C\cup m$, such that $\mu\cup\mu^{\prime
}=C\cup m$, and $\lambda^{\prime}$ is also essential in $(X,\partial_{0}X)$
and crosses $\lambda$. This again contradicts the hypothesis that $\lambda$ is
isolated, which proves the required result in the case when $\lambda$ is not orientable.

2) Suppose that $C$ is not essential in $(X,\partial X)$, so that $C$ is
homotopic to a boundary component $S$ of $X$. As $C$ is essential in
$(X,\partial_{0}X)$, it follows that $S$ is not contained in $\partial_{0}X$
or in $\partial_{1}X$. If $X$ has negative Euler characteristic, there is a
simple arc $\mu$ in $X$ with both ends in $S$ which is essential in
$(X,\partial X)$, and so crosses $C$. By choosing the ends of $\mu$ to lie in
$\partial_{0}X$, we obtain a contradiction. If $X$ has zero Euler
characteristic, and is orientable, it must be an annulus or $D^{2}(2,2)$, the
$2$--disk with two interior cone points each labeled $2$, as $X$ has non-empty
boundary. Note that $D^{2}(2,2)$ is double covered by the annulus. Thus in
general, if $X$ has zero Euler characteristic, it is covered by the annulus.
In particular $\partial X$ has $1$ or $2$ components. In the first case, there
is again an essential simple arc $\lambda$ in $(X,\partial X)$ with boundary
in $\partial_{0}X$ which must cross $C$. See Figures \ref{fig1}e), f), h), i)
and j). In the second case, the two boundary components are homotopic, so that
$C$ is homotopic to each. Thus neither is contained in $\partial_{0}X$ or in
$\partial_{1}X$, and there is a simple arc in $X\ $with ends in $\partial
_{0}X$ which joins these two boundary components, and so must be essential and
cross $C$. See Figures \ref{fig1}a), b), c), d) and g). These contradictions
complete the proof that $C$ must be essential in $(X,\partial X)$, as required.

3) As $\chi(X)>0$, the orbifold fundamental group of $X$ must be finite, so
that $X$ cannot contain any "simple closed curve" which is essential in
$(X,\partial_{0}X)$.

If $\chi(X)>0$, and $\partial X$ is non-empty, the universal orbifold cover of
$X$ must be the $2$--disc, so that $X$ is either a cone or the quotient of a
cone by a reflection. Let $D^{2}(p)$ denote the $2$--orbifold with underlying
surface the $2$--disk and with a single interior cone point of order $p\geq1$,
and let $Y_{p}$ denote the quotient of $D^{2}(p)$ by a reflection. Note that
$D^{2}(1)$ is simply the $2$--disk. The underlying surface of $Y_{p}$ is a
disk $D$, and the boundary $\partial Y_{p}$ consists of a single interval in
$\partial D$ with reflector ends. If $p=1$, the rest of $\partial D$ is a
single mirror, and if $p\geq2$, the rest of $\partial D$ is divided into two
mirrors separated by a boundary cone point, labeled $p$. Let $\left\vert
\partial_{0}X\right\vert $ denote the number of components of $\partial_{0}X$.
In all cases, there is a number $k$ depending on $X$ such that if $\left\vert
\partial_{0}X\right\vert <k$, then there are no essential "simple arcs" in
$(X,\partial_{0}X)$, and if $\left\vert \partial_{0}X\right\vert \geq k$, then
$X$ contains such simple arcs, but no such arc can be isolated. If $X$ is the
disk $D^{2}(1)$, then $k=4$. If $X$ is a cone $D^{2}(p)$, $p\geq2$, then
$k=3$. If $X$ is $Y_{1}$, then $k=3$, and if $X$ is $Y_{p}$, $p\geq2$, then
$k=2$. Note that if $X$ is $D^{2}(p)$, the existence of such $k$ is clear. For
if $\rho$ denotes an orientation preserving homeomorphism of $(X,\partial
_{0}X)$ which sends each component of $\partial_{0}X$ and $\partial_{1}X$ to
the next such component, and if there is an essential simple arc $\lambda$ in
$(X,\partial_{0}X)$, then $\rho(\lambda)$ must cross $\lambda$.
\end{proof}

Now we can prove the following result.

\begin{lemma}
Let $(G,\partial G)$ be an orientable $PD(n+2)$ pair such that $G$ is not
$VPC$. A splitting of $G$ dual to an annulus or torus of $(G,\partial G)$
which is isolated in $\mathcal{S}_{n,n+1}(G)$ is either an edge splitting of
$\Gamma_{n,n+1}(G)$ or is dual to an annulus enclosed by a $V_{0}$--vertex of
commensuriser type, which is not of solid torus type. Further the family of
all splittings of $G$ dual to an annulus or torus which are isolated in
$\mathcal{S}_{n,n+1}(G)$ is finite.
\end{lemma}

\begin{proof}
If $\alpha$ is a splitting dual to an annulus or torus and is not an edge
splitting of $\Gamma_{n,n+1}(G)$, Lemma
\ref{alphaisenclosedbycommensurisertypevertex} tells us that $\alpha$ is
enclosed by a $V_{0}$--vertex $v$ of commensuriser type.

If $\alpha$ is a splitting dual to a torus, it determines an isolated simple
closed curve $C$ in the base orbifold $X_{v}$ of $v$, and $C$ is essential in
$(X_{v},\partial_{0}X_{v})$. Thus part 2) of Lemma \ref{essentiality} implies
that $C$ is essential in $(X_{v},\partial X_{v})$. Now Corollary 5.10 of
\cite{GL} implies there is some simple closed curve in $X\ $which crosses $C$,
which is a contradiction. It follows that a splitting of $G$ dual to a torus
of $(G,\partial G)$ which crosses no element of $\mathcal{S}_{n,n+1}(G)$ must
be an edge splitting of $\Gamma_{n,n+1}(G)$.

If $\alpha$ is a splitting dual to an annulus, it determines an isolated
essential simple arc in $(X_{v},\partial_{0}X_{v})$, which must be essential
in $(X_{v},\partial X_{v})$ by part 1) of Lemma \ref{essentiality}. As the
number of disjoint non-parallel such arcs in $X_{v}$ is finite, and as
$\Gamma_{n,n+1}(G)$ has only finitely many vertices, the result follows.
Finally part 3) of Lemma \ref{essentiality} implies that $v$ cannot be of
solid torus type.
\end{proof}

Now we can proceed to give a complete description of the exceptional
splittings of $G$. These are splittings of $G$ dual to an annulus of
$(G,\partial G)$ which are isolated in $\mathcal{S}_{n,n+1}(G)$ and are not
edge splittings of $\Gamma_{n,n+1}(G)$. Thus each exceptional splitting is
enclosed by some $V_{0}$--vertex $v$ of commensuriser type, which is not of
solid torus type. Let $X_{v}$ be the base orbifold of $v$. Then our annulus
determines an isolated essential simple arc $\lambda$ in $(X_{v},\partial
_{0}X_{v})$, and so $\lambda$ is essential in $(X_{v},\partial X_{v})$, by
Lemma \ref{essentiality}.

In order to give a complete list of cases, the following lemma will be very useful.

\begin{lemma}
\label{d1XintersectCisempty}Let $X$ be a compact $2$--orbifold, with Euler
characteristic $\chi(X)\leq0$, and let $\partial_{1}X$ denote a possibly empty
compact suborbifold of $\partial X$. Let $\partial_{0}X$ denote the closure of
$\partial X-\partial_{1}X$. Let $\lambda$ be an isolated essential simple arc
in $(X,\partial_{0}X)$, such that a point of $\partial\lambda$ lies in a
component $C$ of $\partial X$. Then the following hold:

\begin{enumerate}
\item If $\chi(X)<0$, then $C\subset\partial_{0}X$.

\item If $\chi(X)=0$, and $\partial X$ is connected, then $C\subset
\partial_{0}X$.
\end{enumerate}
\end{lemma}

\begin{remark}
\label{d1Xisempty}It follows that in all cases, if $X$ admits such an arc
$\lambda$, and if $\partial X$ is connected, then $\partial_{1}X$ must be
empty. If $\chi(X)<0$, the same conclusion holds if $\partial X$ has two
components which are joined by $\lambda$.

There are two orbifolds where the hypotheses of the lemma hold and $\chi
(X)=0$, and $\partial X$ is not connected. In each of these cases, $C$ need
not be contained in $\partial_{0}X$. See Figures \ref{fig1}a), \ref{fig1}b)
and \ref{fig1}c).
\end{remark}

\begin{proof}
By Lemma \ref{essentiality}, $\lambda$ is essential in $(X,\partial X)$.

1) As $\chi(X)<0$, it is not possible to have two component of $\partial
X\ $which are homotopic. Thus if $C$ is not contained in $\partial_{0}X$, a
push off $C^{\prime}$ is essential in $(X,\partial_{0}X)$ and crosses
$\lambda$. This contradiction shows that $C$ must be contained in
$\partial_{0}X$, as required.

2) As $\partial X$ is connected, we can use the same argument as in part 1) to
show that $C$ must be contained in $\partial_{0}X$, as required.
\end{proof}

Now we will proceed to list all cases of $(X,\partial_{0}X,\partial_{1}X)$,
where $X$ is a compact $2$--orbifold with $\chi(X)\leq0$ and non-empty
boundary, $\partial_{1}X$ is a possibly empty compact suborbifold of $\partial
X$, and $\partial_{0}X$ is the closure of $\partial X-\partial_{1}X$, and
there is an isolated essential simple arc $\lambda$ in $(X,\partial_{0}X)$. In
all cases, when such an arc exists, it is unique up to isotopy. As $\lambda$
has at least one boundary point, which must lie in $\partial_{0}X$, it follows
that $\partial_{0}X$ is non-empty. To find an isolated essential simple arc
$\lambda$ we have to find all essential simple arcs and omit those that cross
others. In what follows we have not shown all these arcs, only the isolated
ones. (We show some examples with all possible arcs in Figure \ref{fignew5}.)

Recall from Lemma \ref{bijectionbetweencurvesandtori} that any isolated
essential arc in the base orbifold $X_{v}$ of a $V_{0}$--vertex $v$ of
$\Gamma_{n,n+1}(G)$ of commensuriser type gives rise to an exceptional annulus
in $\mathcal{S}_{n,n+1}(G)$, under the assumption that $G$ is not $VPC$. In
particular, this excludes the situation where $\pi_{1}^{orb}(X_{v})$ is $VPC1$
and $\partial_{1}X_{v}$ is empty. Thus in Figure \ref{fig1}, the isolated arcs
shown in \ref{fig1}a), \ref{fig1}b) and \ref{fig1}c) are the only cases which
are relevant to finding exceptional annuli.

However if $G$ is $VPC(n+1)$ and is $VPCn$--by--$\pi_{1}^{orb}(X)$ where $X$
is a $2$--orbifold such that $\pi_{1}^{orb}(X)$ is $VPC1$ and $\partial_{1}X$
is empty, then an isolated arc in $X_{v}$ still determines an essential
annulus in $(G,\partial G)$, but that annulus need not be isolated. For
example, consider the isolated arcs shown in Figures \ref{fig1}e) and
\ref{fig1}f). The orientable $3$--dimensional Seifert fibre spaces over the
orbifolds in these two figures are each homeomorphic to the twisted
$I$--bundle over the Klein bottle with orientable total space, and the annuli
determined by the isolated arcs cross each other. For a discussion of this
example, see page 15 in \cite{SS03}. Higher dimensional examples can be
obtained from this example by taking the product with circles.

In drawing the orbifold $X$, the pictured boundary consists of the orbifold
boundary $\partial X$ and mirrors. The mirrors are drawn in thick lines and
$\partial X$ in thin lines. We then proceed to the division of $\partial X$
into $\partial_{0}X$ and $\partial_{1}X$. In the following pictures
$\partial_{0}X$ is still drawn in thin lines, $\partial_{1}X$ in dashed lines,
and the isolated arc $\lambda$ in dotted lines. Figure \ref{fig1} shows all
examples with $\chi(X)=0$. Each of the orbifolds in Figure \ref{fig1} is
covered by the annulus, and so has orbifold fundamental group which is $VPC1$.

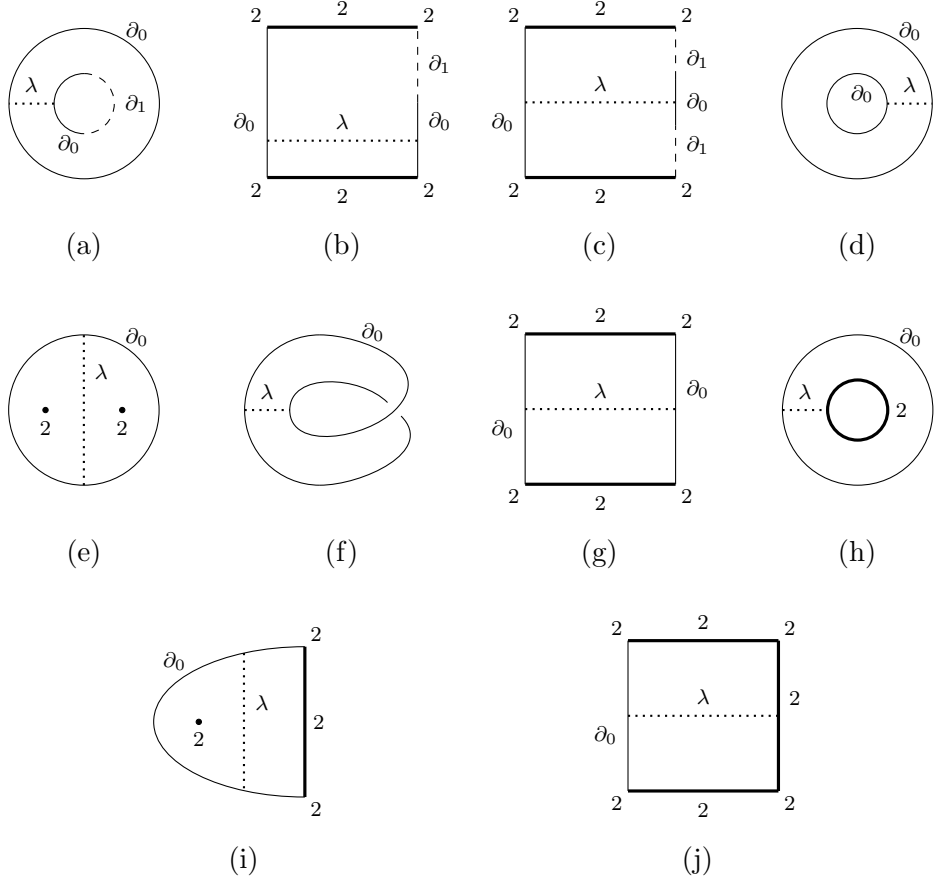
\begin{figure}[ph]
\centering
\begin{subfigure}[b]{0.23\textwidth}
\centering
\begin{tikzpicture}
\newcommand{\outrad}{1}
\newcommand{\inrad}{0.4}
\draw (0,\inrad) arc[start angle=90, end angle=270, radius=\inrad];
\draw [dashed] (0,\inrad) arc[start angle=90, end angle=-90, radius=\inrad];
\node [right ] at (\inrad,0) {\scriptsize $\partial_1$};
\draw (0,0) circle (\outrad);
\draw [thick, dotted] (-\inrad,0) -- node [above] {\scriptsize$\lambda$} (-\outrad,0);
\node [above] at (45:\outrad) {\scriptsize$\partial_0$};
\node [below=-2pt ] at (-120:\inrad) {\scriptsize$\partial_0$};
\node [below] at (0,-\outrad) {\phantom{\scriptsize 2}}; 
\end{tikzpicture}
\caption{}
\end{subfigure}
\hfil
\begin{subfigure}[b]{0.23\textwidth}
\centering
\begin{tikzpicture}
\newcommand{\w}{1cm}
\newcommand{\h}{1cm}
\draw [very thick] (-\w,\h) -- node [above] {\scriptsize 2} (\w,\h);
\draw [very thick] (-\w,-\h) -- node [below] {\scriptsize 2} (\w,-\h);
\draw [thick, dotted] (-\w, -0.5\h) -- node [above] {\scriptsize$\lambda$} (\w,-0.5\h);
\draw (-\w,-\h) -- (-\w,\h);
\draw (\w,-\h) -- (\w,0);
\draw [dashed] (\w,0) -- (\w,\h);
\node [below left] at (-\w,0) {\scriptsize$\partial_0$};
\node [above right] at (\w,-0.5\h) {\scriptsize$\partial_0$};
\node [ right] at (\w,0.5\h) {\scriptsize$\partial_1$};
\node [above right=-2pt] at (\w,\h) {\scriptsize 2};
\node [below right=-2pt] at (\w,-\h) {\scriptsize 2};
\node [above left=-2pt] at (-\w,\h) {\scriptsize 2};
\node [below left=-2pt] at (-\w,-\h) {\scriptsize 2};
\end{tikzpicture}
\caption{}
\end{subfigure}
\hfil
\begin{subfigure}[b]{0.23\textwidth}
\centering
\begin{tikzpicture}
\newcommand{\w}{1cm}
\newcommand{\h}{1cm}
\draw [very thick] (-\w,\h) -- node [above] {\scriptsize 2} (\w,\h);
\draw [very thick] (-\w,-\h) -- node [below] {\scriptsize 2} (\w,-\h);
\draw [thick, dotted] (-\w, 0) -- node [above] {\scriptsize$\lambda$} (\w,0);
\draw (-\w,-\h) -- (-\w,\h);
\draw [dashed] (\w,-\h) -- (\w,-0.3\h);
\draw (\w,-0.3\h) -- (\w,0.3\h);
\draw [dashed] (\w,0.3\h) -- (\w,\h);
\node [below left] at (-\w,0) {\scriptsize$\partial_0$};
\node [ right] at (\w,0) {\scriptsize$\partial_0$};
\node [below right] at (\w,-0.3\h) {\scriptsize$\partial_1$};
\node [above right] at (\w,0.3\h) {\scriptsize$\partial_1$};
\node [above right=-2pt] at (\w,\h) {\scriptsize 2};
\node [below right=-2pt] at (\w,-\h) {\scriptsize 2};
\node [above left=-2pt] at (-\w,\h) {\scriptsize 2};
\node [below left=-2pt] at (-\w,-\h) {\scriptsize 2};
\end{tikzpicture}
\caption{}
\end{subfigure}
\hfil
\begin{subfigure}[b]{0.23\textwidth}
\centering
\begin{tikzpicture}
\newcommand{\outrad}{1cm}
\newcommand{\inrad}{0.4cm}
\draw (0,0) circle (\inrad);
\draw (0,0) circle (\outrad);
\draw [thick, dotted] (\inrad,0) -- node [above] {\scriptsize$\lambda$} (\outrad,0);
\node [above ] at (45:\outrad) {\scriptsize$\partial_0$};
\node [below ] at (80:\inrad) {\scriptsize$\partial_0$};
\node [below] at (0,-\outrad) {\phantom{\scriptsize 2}}; 
\end{tikzpicture}
\caption{}
\end{subfigure}
\\[4mm]\begin{subfigure}[b]{0.23\textwidth}
\centering
\begin{tikzpicture}
\newcommand{\outrad}{1cm}
\newcommand{\inrad}{0.4cm}
\draw (0,0) circle (\outrad);
\draw [thick, dotted] (0,-\outrad) --  (0,\outrad);
\node [right] at (0,0.5\outrad) {\scriptsize$\lambda$};
\draw[fill] (-0.5\outrad,0) circle (1pt);
\node [below] at (-0.5\outrad,0) {\scriptsize 2};
\node [below] at (0.5\outrad,0) {\scriptsize 2};
\draw[fill] (0.5\outrad,0) circle (1pt);
\node [above ] at (45:\outrad) {\scriptsize$\partial_0$};
\node [below] at (0,-\outrad) {\phantom{\scriptsize 2}}; 
\end{tikzpicture}
\caption{}
\end{subfigure}
\hfil
\begin{subfigure}[b]{0.23\textwidth}
\centering
\begin{tikzpicture} 
\newcommand{\outrad}{1}
\newcommand{\inrad}{0.4}
\draw (0:\outrad) to [out=45,in=0, out looseness=1.5] (90:\outrad);
\draw (90:\outrad) to [out=180,in=90] (180:\outrad);
\draw (180:\outrad) to [out=-90,in=180] (-90:\outrad);
\draw (-90:\outrad) to [out=0,in=-45, in looseness=1.2] (1.08,-0.08);
\draw  (0.9,0.1) to [out=135,in=90] (180:\inrad);
\draw (180:\inrad) to [out=-90,in=-135] (0:\outrad);
\draw [thick, dotted] (180:\outrad)-- node [above] {\scriptsize$\lambda$} (180:\inrad);
\node [above=2pt ] at (45:\outrad) {\scriptsize$\partial_0$};
\node [below] at (0,-\outrad) {\phantom{\scriptsize 2}}; 
\end{tikzpicture}
\caption{}
\end{subfigure}
\hfil
\begin{subfigure}[b]{0.23\textwidth}
\centering
\begin{tikzpicture}
\newcommand{\w}{1cm}
\newcommand{\h}{1cm}
\draw [very thick] (-\w,\h) -- node [above] {\scriptsize 2} (\w,\h);
\draw [very thick] (-\w,-\h) -- node [below] {\scriptsize 2} (\w,-\h);
\draw [thick, dotted] (-\w, 0) -- node [above] {\scriptsize$\lambda$} (\w,0);
\draw (-\w,-\h) -- (-\w,\h);
\draw (\w,-\h) -- (\w,\h);
\node [below left] at (-\w,0) {\scriptsize$\partial_0$};
\node [above right] at (\w,0) {\scriptsize$\partial_0$};
\node [above right=-2pt] at (\w,\h) {\scriptsize 2};
\node [below right=-2pt] at (\w,-\h) {\scriptsize 2};
\node [above left=-2pt] at (-\w,\h) {\scriptsize 2};
\node [below left=-2pt] at (-\w,-\h) {\scriptsize 2};
\end{tikzpicture}
\caption{}
\end{subfigure}
\hfil
\begin{subfigure}[b]{0.23\textwidth}
\centering
\begin{tikzpicture}
\newcommand{\outrad}{1cm}
\newcommand{\inrad}{0.4cm}
\draw [very thick] (0,0) circle (\inrad);
\draw (0,0) circle (\outrad);
\draw [thick, dotted] (-\outrad,0) -- node [above] {\scriptsize$\lambda$} (-\inrad,0);
\node [above ] at (45:\outrad) {\scriptsize$\partial_0$};
\node [right=-1pt ] at (\inrad,0) {\scriptsize $2$};
\node [below] at (0,-\outrad) {\phantom{\scriptsize 2}}; 
\end{tikzpicture}
\caption{}
\end{subfigure}
\\[4mm]\begin{subfigure}[b]{0.3\textwidth}
\centering
\begin{tikzpicture}
\newcommand{\outrad}{1cm}
\newcommand{\inrad}{0.4cm}
\draw [name path=ellipse] (0,1) arc[start angle=90, end angle=270, x radius=2.0\outrad, y radius=\outrad];
\node at (155:1.9\outrad) {\scriptsize$\partial_0$}; 
\draw [very thick] (0,-\outrad) -- (0,\outrad);
\path[name path=con] (-0.8\outrad,-\outrad) -- (-0.8\outrad,\outrad); 
\draw[thick, dotted,  name intersections={of=ellipse and con}] (intersection-1) -- node [above right] {\scriptsize$\lambda$} (intersection-2); 
\draw[fill] (-1.4\outrad,0) circle (1pt);
\node [below] at (-1.4\outrad,0) {\scriptsize 2};
\node [above right=-2pt] at (0,\outrad) {\scriptsize 2};
\node [below right=-2pt] at (0,-\outrad) {\scriptsize 2};
\node [right=-1pt] at (0,0) {\scriptsize 2};
\end{tikzpicture}
\caption{}
\end{subfigure}
\hfil
\begin{subfigure}[b]{0.3\textwidth}
\centering
\begin{tikzpicture}
\newcommand{\w}{1cm}
\newcommand{\h}{1cm}
\draw [very thick] (-\w,\h) -- node [above] {\scriptsize 2} (\w,\h);
\draw [very thick] (-\w,-\h) -- node [below] {\scriptsize 2} (\w,-\h);
\draw [thick, dotted] (-\w, 0) -- node [above] {\scriptsize$\lambda$} (\w,0);
\draw (-\w,-\h) -- (-\w,\h);
\draw [very thick] (\w,-\h) -- (\w,\h);
\node [below left] at (-\w,0) {\scriptsize$\partial_0$};
\node [above right] at (\w,0) {\scriptsize$2$};
\node [below left=-2pt] at (-\w,-\h) {\scriptsize$2$};
\node [above left=-2pt] at (-\w,\h) {\scriptsize$2$};
\node [above right=-2pt] at (\w,\h) {\scriptsize 2};
\node [below right=-2pt] at (\w,-\h) {\scriptsize 2};
\end{tikzpicture}
\caption{}
\end{subfigure}
\caption{The case in which $\chi(X)=0$}%
\label{fig1}%
\end{figure}

We next consider the cases with $\chi(X)<0$. Recall that $\lambda$ is a
"simple arc" in $(X,\partial_{0}X)$ which is essential in $(X,\partial X)$ and
crosses no essential \textquotedblleft simple closed curve" in $X$. Corollary
5.10\ of \cite{GL} tells us that $X$ admits no essential "simple closed
curves" at all. Thus $X$ lies on the list of ten orbifolds given in
Proposition 5.12 of \cite{GL}. However, two of these ten have no boundary. In
Figure \ref{fignew2}, we show the remaining eight orbifolds. For each of these
eight orbifolds, we use Lemmas \ref{essentiality} and
\ref{d1XintersectCisempty} to determine the possible decompositions of
$\partial X$ into $\partial_{0}X$ and $\partial_{1}X$ which admit an isolated
essential arc, and we show all these cases in Figures \ref{fignew3} and
\ref{fignew4}. Figure \ref{fignew3} shows the cases where $\partial_{1}X$ is
empty, and Figure \ref{fignew4} shows the other cases.

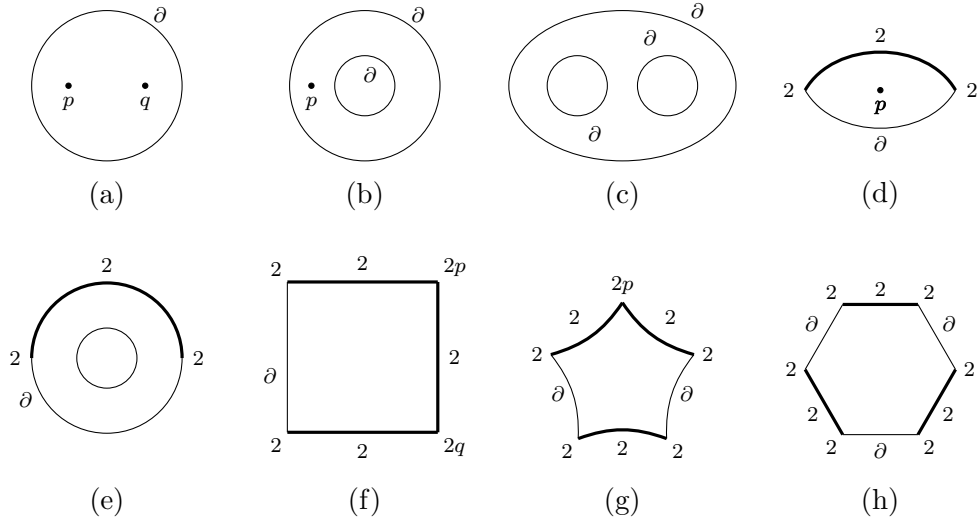
\begin{figure}[ph]
\centering
\begin{subfigure}[b]{0.23\textwidth}
\centering
\begin{tikzpicture}
\newcommand{\outrad}{1cm}
\newcommand{\inrad}{0.4cm}
\draw (0,0) circle (\outrad);
\draw[fill] (-0.5\outrad,0) circle (1pt);
\node [below] at (-0.5\outrad,0) {\scriptsize $p$};
\node [below] at (0.5\outrad,0) {\scriptsize $q$};
\draw[fill] (0.5\outrad,0) circle (1pt);
\node [above ] at (45:\outrad) {\scriptsize$\partial$};
\end{tikzpicture}
\caption{}
\end{subfigure}
\hfil
\begin{subfigure}[b]{0.23\textwidth}
\centering
\begin{tikzpicture}
\newcommand{\outrad}{1cm}
\newcommand{\inrad}{0.4cm}
\draw (0,0) circle (\inrad);
\draw (0,0) circle (\outrad);
\node [above ] at (45:\outrad) {\scriptsize$\partial$};
\node [below ] at (80:\inrad) {\scriptsize$\partial$};
\draw[fill] (-0.7\outrad,0) circle (1pt);
\node [below] at (-0.7\outrad,0) {\scriptsize $p$};
\end{tikzpicture}
\caption{}
\end{subfigure}
\hfil
\begin{subfigure}[b]{0.23\textwidth}
\centering
\begin{tikzpicture}
\newcommand{\outrad}{1cm}
\newcommand{\inrad}{0.4cm}
\draw (0,0) circle [x radius=1.5\outrad, y radius=\outrad];
\node at (45:1.4\outrad) {\scriptsize $\partial$};
\draw (-0.6,0) circle (\inrad); 
\node [below right] at (-0.6,-\inrad) {\scriptsize $\partial$};
\draw (0.6,0) circle (\inrad); 
\node [above left] at (0.6,\inrad) {\scriptsize $\partial$};
\end{tikzpicture}
\caption{}
\end{subfigure}
\hfil
\begin{subfigure}[b]{0.23\textwidth} 
\centering
\begin{tikzpicture}
\newcommand{\outrad}{1cm}
\newcommand{\inrad}{0.4cm}
\draw [very thick] (-\outrad,0) to [out=60, in=120] (\outrad,0);
\draw  (-\outrad,0) to [out=-60, in=-120] node [below=-1pt] {\scriptsize$\partial$}(\outrad,0);
\draw[fill] (0,0) circle (1pt);
\node [below] at (0,0) {\scriptsize $p$};
\node [below] at (0,0) {\scriptsize $p$};
\node [below] at (0,0) {\scriptsize $p$};
\node [above] at (0,0.5) {\scriptsize $2$};
\node [left] at (-\outrad,0) {\scriptsize $2$};
\node [right] at (\outrad,0) {\scriptsize $2$};
\end{tikzpicture}
\caption{}
\end{subfigure}
\\[4mm]\begin{subfigure}[b]{0.23\textwidth} 
\centering
\begin{tikzpicture}
\newcommand{\outrad}{1cm}
\newcommand{\inrad}{0.4cm}
\draw [very thick] (\outrad,0) arc [start angle=0, end angle=180, radius=\outrad];
\draw  (\outrad,0) arc [start angle=0, end angle=-180, radius=\outrad];
\draw (0,0) circle (\inrad);
\node [above] at (0,\outrad) {\scriptsize $2$};
\node [left] at (-\outrad,0) {\scriptsize $2$};
\node [right] at (\outrad,0) {\scriptsize $2$};
\node [below left] at (-160:0.9) {\scriptsize $\partial$};
\node [below] at (0,-\outrad) {\phantom{\scriptsize $2$}}; 
\end{tikzpicture}
\caption{}
\end{subfigure}
\hfil
\begin{subfigure}[b]{0.23\textwidth}
\centering
\begin{tikzpicture}
\newcommand{\w}{1cm}
\newcommand{\h}{1cm}
\draw [very thick] (-\w,\h) -- node [above] {\scriptsize 2} (\w,\h);
\draw [very thick] (-\w,-\h) -- node [below] {\scriptsize 2} (\w,-\h);
\draw (-\w,-\h) -- (-\w,\h);
\draw [very thick] (\w,-\h) -- (\w,\h);
\node [ right] at (\w,0) {\scriptsize$2$};
\node [below right=-2pt] at (\w,-\h) {\scriptsize$2q$};
\node [above right=-2pt] at (\w,\h) {\scriptsize$2p$};
\node [below left] at (-\w,0) {\scriptsize$\partial$};
\node [below left=-2pt] at (-\w,-\h) {\scriptsize$2$};
\node [above left=-2pt] at (-\w,\h) {\scriptsize$2$};
\end{tikzpicture}
\caption{}
\end{subfigure}
\hfil
\begin{subfigure}[b]{0.23\textwidth}
\centering
\begin{tikzpicture}
\newcommand{\outrad}{1cm}
\newcommand{\inrad}{0.4cm}
\draw [very thick, rotate=0] (-126:\outrad) to [out=20, in=160] node [below] {\scriptsize 2} (-54:\outrad);
\draw [rotate=72] (-126:\outrad) to [out=20, in=160] node [right=-1pt]{\scriptsize$ \partial$}(-54:\outrad);
\draw [very thick, rotate=144] (-126:\outrad) to [out=20, in=160] node [above right ] {\scriptsize 2} (-54:\outrad);
\draw [very thick, rotate=-144] (-126:\outrad) to [out=20, in=160] node [above left ] {\scriptsize 2} (-54:\outrad);
\draw [rotate=-72] (-126:\outrad) to [out=20, in=160]  node [left=-1pt]{\scriptsize$ \partial$} (-54:\outrad);
\node [above=-1pt] at (0,\outrad) {\scriptsize $2p$};
\node [right=-1pt] at (18:\outrad) {\scriptsize $2$};
\node [left=-1pt] at (162:\outrad) {\scriptsize $2$};
\node [below right=-2pt] at (-54:\outrad) {\scriptsize $2$};
\node [below left=-2pt] at (-126:\outrad) {\scriptsize $2$};
\end{tikzpicture}
\caption{}
\end{subfigure}
\hfil
\begin{subfigure}[b]{0.23\textwidth}
\centering
\begin{tikzpicture}
\newcommand{\outrad}{1cm}
\draw [rotate=0] (-120:\outrad) to  (-60:\outrad); 
\draw [rotate=60, very thick] (-120:\outrad) to  (-60:\outrad);
\draw [rotate=120] (-120:\outrad) to  (-60:\outrad);
\draw [rotate=180, very thick] (-120:\outrad) to  (-60:\outrad);
\draw [rotate=-60, very thick] (-120:\outrad) to  (-60:\outrad);
\draw [rotate=-120] (-120:\outrad) to  (-60:\outrad);
\node [below right=-2pt] at (-60:\outrad) {\scriptsize$2$};
\node [below left=-2pt] at (-120:\outrad) {\scriptsize$2$};
\node [right=-1pt] at (0:\outrad) {\scriptsize$2$};
\node [above right=-2pt] at (60:\outrad) {\scriptsize$2$};
\node [above left=-2pt] at (120:\outrad) {\scriptsize$2$};
\node [left=-1pt] at (180:\outrad) {\scriptsize$2$};
\node [below right=-2pt] at (-30:0.866\outrad) {\scriptsize$2$};
\node [above=-1pt] at (90:0.866\outrad) {\scriptsize$2$};
\node [below left=-2pt] at (-150:0.866\outrad) {\scriptsize$2$};
\node [above right=-2pt] at (30:0.866\outrad) {\scriptsize$\partial$};
\node [above left=-2pt] at (150:0.866\outrad) {\scriptsize$\partial$};
\node [below=-1pt] at (-90:0.866\outrad) {\scriptsize$\partial$};
\end{tikzpicture}
\caption{}
\end{subfigure}
\caption{The eight orbifolds with $\chi<0$, non-empty boundary, and no
essential closed curves}%
\label{fignew2}%
\end{figure}

Here is a verbal description of the eight orbifolds in Figure \ref{fignew2},
and of the possible isolated essential simple arcs.

\begin{enumerate}
\item $X=D^{2}(p,q)$, the $2$--disc with two interior cone points of orders
$p,q\geq2$ with at least one strictly larger than $2$. There is an isolated
essential simple arc in $(X,\partial_{0}X)$ iff $\partial_{1}X$ is empty.

\item $X=S^{1}\times I(p)$, $p\geq2$, the annulus with one interior cone
point. There is an isolated essential simple arc in $(X,\partial_{0}X)$ iff
$\partial_{1}X$ is empty or is a component of $\partial X$.

\item $X$ is a pair of pants, with no singular points. There is an isolated
essential simple arc in $(X,\partial_{0}X)$ iff $\partial_{1}X$ is a component
of $\partial X$, or is the union of two components of $\partial X$.

\item The underlying surface of $X$ is a disc $D$. The boundary of $D$
contains one mirror interval, so that $\partial X$ is the closure of the
complement of this mirror interval in $\partial D$, and $X$ has one interior
cone point labeled $p$. There is an isolated essential arc in $(X,\partial
_{0}X)$ iff $\partial_{1}X$ is empty.

\item The underlying surface of $X$ is an annulus $A$. The boundary of $A$
contains one mirror interval, so that $\partial X$ is the closure of the
complement of this mirror interval in $\partial A$. There is an isolated
essential simple arc in $(X,\partial_{0}X)$ iff $\partial_{1}X$ is empty or is
a component of $\partial X$.

\item The underlying surface of $X$ is a disc $D$, and the boundary $\partial
X$ of $X$ consists of a single interval in $\partial D$ with reflector ends,
and the rest of $\partial D$ is divided into three mirrors separated by two
boundary cone points, labeled $2p$ and $2q$. There is an isolated essential
simple arc in $(X,\partial_{0}X)$ iff $\partial_{1}X$ is empty.

\item The underlying surface of $X$ is a disc $D$. The boundary $\partial X$
of $X$ consists of two disjoint intervals in $\partial D$ each with reflector
ends, and the rest of $\partial D$ is divided into a single mirror and two
mirrors separated by a boundary cone point, labeled $2p$. There is an isolated
essential simple arc in $(X,\partial_{0}X)$ iff $\partial_{1}X$ is empty or is
a component of $\partial X$.

\item The underlying surface of $X$ is a disc $D$. The boundary $\partial X$
of $X$ consists of three disjoint intervals in $\partial D$ each with
reflector ends, and the rest of $\partial D$ consists of three mirrors. There
is an isolated essential simple arc in $(X,\partial_{0}X)$ iff $\partial_{1}X$
is a component of $\partial X$, or is the union of two components of $\partial
X$.
\end{enumerate}

\begin{figure}[ph]
\centering
\begin{subfigure}[b]{0.23\textwidth}
\centering
\begin{tikzpicture}
\newcommand{\outrad}{1cm}
\newcommand{\inrad}{0.4cm}
\draw (0,0) circle (\outrad);
\draw [thick, dotted] (0,-\outrad) --  (0,\outrad);
\node [right] at (0,0.5\outrad) {\scriptsize$\lambda$};
\draw[fill] (-0.5\outrad,0) circle (1pt);
\node [below] at (-0.5\outrad,0) {\scriptsize $p$};
\node [below] at (0.5\outrad,0) {\scriptsize $q$};
\draw[fill] (0.5\outrad,0) circle (1pt);
\node [above=-3pt ] at (60:1.1\outrad) {\scriptsize$\partial_0$};
\end{tikzpicture}
\caption{}
\end{subfigure}
\hfil
\begin{subfigure}[b]{0.23\textwidth}
\centering
\begin{tikzpicture}
\newcommand{\outrad}{1cm}
\newcommand{\inrad}{0.4cm}
\draw (0,0) circle (\inrad);
\draw (0,0) circle (\outrad);
\node [above ] at (45:\outrad) {\scriptsize$\partial_0$};
\node [below ] at (80:\inrad) {\scriptsize$\partial_0$};
\draw[fill] (-0.7\outrad,0) circle (1pt);
\node [below] at (-0.7\outrad,0) {\scriptsize $p$};
\draw [thick, dotted] (90:\inrad) -- node [right=-1pt] {\scriptsize$\lambda$}  (90:\outrad);
\end{tikzpicture}
\caption{}
\end{subfigure}
\hfil
\begin{subfigure}[b]{0.23\textwidth} 
\centering
\begin{tikzpicture}
\newcommand{\outrad}{1cm}
\newcommand{\inrad}{0.4cm}
\draw [very thick] (-\outrad,0) to [out=60, in=120] (\outrad,0);
\draw [name path=lower] (-\outrad,0) to [out=-60, in=-120] node [below=-1pt] {\scriptsize$\partial_0$} (\outrad,0);
\draw[fill] (0,0) circle (1pt);
\node [below] at (0,0) {\scriptsize $p$};
\node [below] at (0,0) {\scriptsize $p$};
\node [below] at (0,0) {\scriptsize $p$};
\node [above] at (0,0.5) {\scriptsize $2$};
\node [left] at (-\outrad,0) {\scriptsize $2$};
\node [right] at (\outrad,0) {\scriptsize $2$};
\path [name path=con] (-1,-0.2) -- (1,-0.2); 
\draw[thick, dotted,  name intersections={of=lower and con}] (intersection-1) to [out=60, in=120] node [pos=0.3,below=-1pt] {\scriptsize$\lambda$} (intersection-2); 
\end{tikzpicture}
\caption{}
\end{subfigure}
\hfil
\begin{subfigure}[b]{0.23\textwidth} 
\centering
\begin{tikzpicture}
\newcommand{\outrad}{1cm}
\newcommand{\inrad}{0.4cm}
\draw [very thick] (\outrad,0) arc [start angle=0, end angle=180, radius=\outrad];
\draw  (\outrad,0) arc [start angle=0, end angle=-180, radius=\outrad];
\draw (0,0) circle (\inrad);
\node [above] at (0,\outrad) {\scriptsize $2$};
\node [left] at (-\outrad,0) {\scriptsize $2$};
\node [right] at (\outrad,0) {\scriptsize $2$};
\draw [thick, dotted] (-90:\inrad) -- node [right]  {\scriptsize$\lambda$}  (-90:\outrad);
\node [below left] at (180:0.2) {\scriptsize $\partial_0$};
\node [below left] at (-160:0.8) {\scriptsize $\partial_0$};
\end{tikzpicture}
\caption{}
\end{subfigure}
\\[4mm]\begin{subfigure}[b]{0.23\textwidth}
\centering
\begin{tikzpicture}
\newcommand{\w}{1cm}
\newcommand{\h}{1cm}
\draw [very thick] (-\w,\h) -- node [above] {\scriptsize 2} (\w,\h);
\draw [very thick] (-\w,-\h) -- node [below] {\scriptsize 2} (\w,-\h);
\draw (-\w,-\h) -- (-\w,\h);
\draw [very thick] (\w,-\h) -- (\w,\h);
\node [ right] at (\w,0) {\scriptsize$2$};
\node [below right=-2pt] at (\w,-\h) {\scriptsize$2q$};
\node [above right=-2pt] at (\w,\h) {\scriptsize$2p$};
\draw [thick, dotted] (-\w, 0) -- node [above] {\scriptsize$\lambda$} (\w,0);
\node [below left] at (-\w,0) {\scriptsize$\partial_0$};
\node [below left=-2pt] at (-\w,-\h) {\scriptsize$2$};
\node [above left=-2pt] at (-\w,\h) {\scriptsize$2$};
\end{tikzpicture}
\caption{}
\end{subfigure}
\hfil
\begin{subfigure}[b]{0.23\textwidth}
\centering
\begin{tikzpicture}
\newcommand{\outrad}{1cm}
\newcommand{\inrad}{0.4cm}
\draw [very thick, rotate=0] (-126:\outrad) to [out=20, in=160] node [below] {\scriptsize 2} (-54:\outrad);
\draw [name path=right, rotate=72] (-126:\outrad) to [out=20, in=160] node [right=-1pt]{\scriptsize$ \partial_0$} (-54:\outrad);
\draw [very thick, rotate=144] (-126:\outrad) to [out=20, in=160] node [above right ] {\scriptsize 2} (-54:\outrad);
\draw [very thick, rotate=-144] (-126:\outrad) to [out=20, in=160] node [above left ] {\scriptsize 2} (-54:\outrad);
\draw [name path= left,rotate=-72] (-126:\outrad) to [out=20, in=160] node [left=-1pt]{\scriptsize$ \partial_0$} (-54:\outrad);
\node [above] at (0,\outrad) {\scriptsize $2p$};
\path [name path=con] (-1,-0.3) -- (1,-0.3); 
\path [red, name intersections={of=left and con, by={A}}] (A) circle (2pt);
\path [red, name intersections={of=right and con, by={B}}] (B) circle (2pt);
\draw [thick, dotted] (A) to [out=20, in=160] node [above]  {\scriptsize$\lambda$} (B);
\node [right=-1pt] at (18:\outrad) {\scriptsize $2$};
\node [left=-1pt] at (162:\outrad) {\scriptsize $2$};
\node [below right=-2pt] at (-54:\outrad) {\scriptsize $2$};
\node [below left=-2pt] at (-126:\outrad) {\scriptsize $2$};
\end{tikzpicture}
\caption{}
\end{subfigure}
\caption{Cases with $\partial_{1}X=\varnothing$}%
\label{fignew3}%
\end{figure}

\begin{figure}[ph]
\centering
\begin{subfigure}[b]{0.23\textwidth}
\centering
\begin{tikzpicture}
\newcommand{\outrad}{1}
\newcommand{\inrad}{0.4}
\draw [dashed] (0,0) circle (\inrad);
\draw (0,0) circle (\outrad);
\node [above ] at (45:\outrad) {\scriptsize$\partial_0$};
\node [below ] at (80:\inrad) {\scriptsize$\partial_1$};
\draw[fill] (-0.7\outrad,0) circle (1pt);
\node [below] at (-0.7\outrad,0) {\scriptsize $p$};
\draw [thick, dotted] (140:\outrad) to [out=-50, in=50] node [pos=0.1,right=-1pt] {\scriptsize$\lambda$} (-140:\outrad);
\end{tikzpicture}
\caption{}
\end{subfigure}
\hfil
\begin{subfigure}[b]{0.23\textwidth}
\centering
\begin{tikzpicture}
\newcommand{\outrad}{1cm}
\newcommand{\inrad}{0.4cm}
\draw [name path=ellipse] (0,0) circle [x radius=1.5\outrad, y radius=\outrad];
\node at (45:1.4\outrad) {\scriptsize $\partial_0$};
\draw (-0.6,0) circle (\inrad); 
\node [below right] at (-0.6,-\inrad) {\scriptsize $\partial_0$};
\draw [dashed] (0.6,0) circle (\inrad); 
\node [above left] at (0.6,\inrad) {\scriptsize $\partial_1$};
\path [name path=con] (-0.6,0) -- (-0.6,\outrad); 
\draw[thick, dotted,  name intersections={of=ellipse and con}] (intersection-1) -- node [right] {\scriptsize$\lambda$} (-0.6,\inrad); 
\end{tikzpicture}
\caption{}
\end{subfigure}
\hfil
\begin{subfigure}[b]{0.23\textwidth}
\centering
\begin{tikzpicture}
\newcommand{\outrad}{1cm}
\newcommand{\inrad}{0.4cm}
\draw [name path=ellipse] (0,0) circle [x radius=1.5\outrad, y radius=\outrad];
\node at (45:1.4\outrad) {\scriptsize $\partial_0$};
\draw [dashed] (-0.6,0) circle (\inrad); 
\node [below right] at (-0.6,-\inrad) {\scriptsize $\partial_1$};
\draw [dashed] (0.6,0) circle (\inrad); 
\node [above left] at (0.6,\inrad) {\scriptsize $\partial_1$};
\draw [thick, dotted] (0,-\outrad) --  (0,\outrad);
\node [below right] at (0,-\inrad) {\scriptsize $\lambda$};
\end{tikzpicture}
\caption{}
\end{subfigure}
\hfil
\begin{subfigure}[b]{0.23\textwidth}
\centering
\begin{tikzpicture}
\newcommand{\outrad}{1cm}
\newcommand{\inrad}{0.4cm}
\draw [very thick] (\outrad,0) arc [start angle=0, end angle=180, radius=\outrad];
\draw [name path=lower] (\outrad,0) arc [start angle=0, end angle=-180, radius=\outrad]; 
\draw [dashed] (0,0) circle (\inrad); 
\node [above] at (0,\outrad) {\scriptsize $2$};
\node [left] at (-\outrad,0) {\scriptsize $2$};
\node [right] at (\outrad,0) {\scriptsize $2$};
\node [below] at (-90:0.3) {\scriptsize $\partial_1$};
\node [below left] at (-100:0.8) {\scriptsize $\partial_0$};
\path [name path=con] (-1,-0.3) -- (1,-0.3); 
\draw[thick, dotted,  name intersections={of=lower and con}] (intersection-1) to [out=110, in=0]  (0,0.6) to [out=180, in=70] (intersection-2); 
\node at (-0.7,0) {\scriptsize$\lambda$};
\end{tikzpicture}
\caption{}
\end{subfigure}
%
\\[4mm]\begin{subfigure}[b]{0.23\textwidth}
\centering
\begin{tikzpicture}
\newcommand{\outrad}{1cm}
\newcommand{\inrad}{0.4cm}
\draw [very thick] (\outrad,0) arc [start angle=0, end angle=180, radius=\outrad];
\draw [dashed]  (\outrad,0) arc [start angle=0, end angle=-180, radius=\outrad]; 
\draw (0,0) circle (\inrad);
\node [above] at (0,\outrad) {\scriptsize $2$};
\node [left] at (-\outrad,0) {\scriptsize $2$};
\node [right] at (\outrad,0) {\scriptsize $2$};
\draw [thick, dotted] (90:\inrad) -- node [right]  {\scriptsize$\lambda$}  (90:\outrad);
\node [below left] at (180:0.2) {\scriptsize $\partial_0$};
\node [below left] at (-160:0.8) {\scriptsize $\partial_1$};
\end{tikzpicture}
\caption{}
\end{subfigure}
\hfil
\begin{subfigure}[b]{0.23\textwidth}
\centering
\begin{tikzpicture}
\newcommand{\outrad}{1cm}
\newcommand{\inrad}{0.4cm}
\draw [name path=bottom,very thick, rotate=0] (-126:\outrad) to [out=20, in=160] node [below] {\scriptsize 2} (-54:\outrad);
\draw [rotate=72,name path=bright] (-126:\outrad) to [out=20, in=160] node [right=-1pt]{\scriptsize$ \partial_0$} (-54:\outrad); 
\draw [very thick, rotate=144] (-126:\outrad) to [out=20, in=160] node [above right ] {\scriptsize 2} (-54:\outrad);
\draw [name path= left,very thick, rotate=-144] (-126:\outrad) to [out=20, in=160] node [above left ] {\scriptsize 2} (-54:\outrad);
\draw [rotate=-72, dashed] (-126:\outrad) to [out=20, in=160] node [left=-1pt]{\scriptsize$ \partial_1$} (-54:\outrad);
\node [above] at (0,\outrad) {\scriptsize $2p$};
\path [name path=con] (126:0.8) -- (-18:0.8); 
\path [red, name intersections={of=left and con, by={A}}] (A) circle (2pt);
\path [red, name intersections={of=bright and con, by={B}}] (B) circle (2pt);
\draw [thick, dotted] (A) -- node [below left=-2pt]  {\scriptsize$\lambda$} (B);
\node [right=-1pt] at (18:\outrad) {\scriptsize $2$};
\node [left=-1pt] at (162:\outrad) {\scriptsize $2$};
\node [below right=-2pt] at (-54:\outrad) {\scriptsize $2$};
\node [below left=-2pt] at (-126:\outrad) {\scriptsize $2$};
\end{tikzpicture}
\caption{}
\end{subfigure}
\hfil
\begin{subfigure}[b]{0.23\textwidth}
\centering
\begin{tikzpicture}
\newcommand{\outrad}{1cm}
\draw [rotate=0, dashed] (-120:\outrad) to  (-60:\outrad); 
\draw [rotate=60, very thick] (-120:\outrad) to  (-60:\outrad);
\draw [rotate=120, name path=right] (-120:\outrad) to  (-60:\outrad);
\draw [rotate=180, very thick] (-120:\outrad) to  (-60:\outrad);
\draw [rotate=-60, very thick] (-120:\outrad) to  (-60:\outrad);
\draw [rotate=-120, name path=left] (-120:\outrad) to  (-60:\outrad);
\draw [thick, dotted] (150:0.866) -- node [below]  {\scriptsize$\lambda$} (30:0.866);
\node [below right=-2pt] at (-60:\outrad) {\scriptsize$2$};
\node [below left=-2pt] at (-120:\outrad) {\scriptsize$2$};
\node [right=-1pt] at (0:\outrad) {\scriptsize$2$};
\node [above right=-2pt] at (60:\outrad) {\scriptsize$2$};
\node [above left=-2pt] at (120:\outrad) {\scriptsize$2$};
\node [left=-1pt] at (180:\outrad) {\scriptsize$2$};
\node [below right=-2pt] at (-30:0.866\outrad) {\scriptsize$2$};
\node [above=-1pt] at (90:0.866\outrad) {\scriptsize$2$};
\node [below left=-2pt] at (-150:0.866\outrad) {\scriptsize$2$};
\node [above right=-2pt] at (30:0.866\outrad) {\scriptsize$\partial_0$};
\node [above left=-2pt] at (150:0.866\outrad) {\scriptsize$\partial_0$};
\node [below=-1pt] at (-90:0.866\outrad) {\scriptsize$\partial_1$};
\end{tikzpicture}
\caption{}
\end{subfigure}
\hfil
\begin{subfigure}[b]{0.23\textwidth}
\centering
\begin{tikzpicture}
\newcommand{\outrad}{1cm}
\draw [rotate=0, dashed] (-120:\outrad) to  (-60:\outrad); 
\draw [rotate=60, very thick] (-120:\outrad) to  (-60:\outrad);
\draw [rotate=120, name path=right, dashed] (-120:\outrad) to  (-60:\outrad);
\draw [rotate=180, very thick] (-120:\outrad) to  (-60:\outrad);
\draw [rotate=-60, very thick] (-120:\outrad) to  (-60:\outrad);
\draw [rotate=-120, name path=left] (-120:\outrad) to  (-60:\outrad);
\draw [thick, dotted] (150:0.866) -- node [below]  {\scriptsize$\lambda$} (-30:0.866);
\node [below right=-2pt] at (-60:\outrad) {\scriptsize$2$};
\node [below left=-2pt] at (-120:\outrad) {\scriptsize$2$};
\node [right=-1pt] at (0:\outrad) {\scriptsize$2$};
\node [above right=-2pt] at (60:\outrad) {\scriptsize$2$};
\node [above left=-2pt] at (120:\outrad) {\scriptsize$2$};
\node [left=-1pt] at (180:\outrad) {\scriptsize$2$};
\node [below right=-2pt] at (-30:0.866\outrad) {\scriptsize$2$};
\node [above=-1pt] at (90:0.866\outrad) {\scriptsize$2$};
\node [below left=-2pt] at (-150:0.866\outrad) {\scriptsize$2$};
\node [above right=-2pt] at (30:0.866\outrad) {\scriptsize$\partial_1$};
\node [above left=-2pt] at (150:0.866\outrad) {\scriptsize$\partial_0$};
\node [below=-1pt] at (-90:0.866\outrad) {\scriptsize$\partial_1$};
\end{tikzpicture}
\caption{}
\end{subfigure}
\caption{Cases with $\partial_{1}X\neq\varnothing$}%
\label{fignew4}%
\end{figure}

Thus when $\chi(X)<0$, we have fourteen orbifolds with an isolated essential
simple arc, of which the six shown in Figure \ref{fignew3} have $\partial
_{1}X$ empty. In these six cases, the group $G$ (in Theorem \ref{thm5.8}) is
$VPCn$--by--$\pi_{1}^{orb}(X)$.

Finally we can show that the family $\mathcal{S}_{n,n+1}(G)$ has a regular
neighbourhood which is a refinement $\Sigma_{n,n+1}(G)$ of $\Gamma_{n,n+1}%
(G)$. Every element of $\mathcal{S}_{n,n+1}(G)$ determines a simple closed
curve or simple arc in the base $2$--orbifold of one of the $V_{0}$--vertices
of $\Gamma_{n,n+1}(G)$. Now a connected compact $2$--orbifold is filled by
simple closed curves and simple arcs, unless it is one of the exceptional
cases listed above. Thus cutting the base $2$--orbifold along the exceptional
arc yields $2$--orbifolds which contain no isolated essential simple arc.
However, in several cases the $2$--orbifolds obtained by cutting along an
isolated arc contain non-isolated essential simple arcs. Now splitting an
exceptional $V_{0}$--vertex $v$ along the exceptional annulus yields a vertex
or vertices with base orbifold obtained by cutting $X_{v}$ along the isolated
essential arc. These new vertices enclose elements of $\mathcal{S}_{n,n+1}(G)$
which correspond to simple arcs in the new base orbifolds. Thus these new
vertices enclose elements of $\mathcal{S}_{n,n+1}(G)$ other than edge
splittings of $\Gamma_{n,n+1}(G)$ if and only if the new base orbifold
contains essential simple arcs.

Using the above notation, we can now describe the regular neighbourhood
$\Sigma_{n,n+1}(G)$ of $\mathcal{S}_{n,n+1}(G)$. It is a refinement of
$\Gamma_{n,n+1}(G)$ which can be obtained essentially by splitting each
exceptional $V_{0}$--vertex of $\Gamma_{n,n+1}(G)$ along the exceptional
annulus it contains. Each non-exceptional $V_{0}$--vertex of $\Gamma
_{n,n+1}(G)$ yields unchanged a $V_{0}$--vertex of $\Sigma_{n,n+1}(G)$, and
each $V_{1}$--vertex of $\Gamma_{n,n+1}(G)$ yields unchanged a $V_{1}$--vertex
of $\Sigma_{n,n+1}(G)$. If $v$ is an exceptional $V_{0}$--vertex of
$\Gamma_{n,n+1}(G)$, which contains a separating exceptional annulus, then $v$
is split into two new vertices. If $v$ is an exceptional $V_{0}$--vertex of
$\Gamma_{n,n+1}(G)$, which contains a non-separating exceptional annulus, then
$v$ is split into a single new vertex. If a new vertex encloses elements of
$\mathcal{S}_{n,n+1}(G)$ other than edge splittings of $\Gamma_{n,n+1}(G)$ we
label it as a $V_{0}$--vertex. Otherwise, we label it as a $V_{1}$--vertex.
This yields a refinement of $\Gamma_{n,n+1}(G)$, but it may not be bipartite.
By adding an isolated $V_{0}$--vertex between adjacent $V_{1}$--vertices, and
an isolated $V_{1}$--vertex between adjacent $V_{0}$--vertices, and then
reducing if needed, we can create a bipartite graph of groups which will be
the regular neighbourhood $\Sigma_{n,n+1}(G)$ of $\mathcal{S}_{n,n+1}(G)$. We
have shown the following result.

\begin{theorem}
\label{thm5.8} Let $(G,\partial G)$ be an orientable $PD(n+2)$ pair such that
$G$ is not $VPC$, and let $\mathcal{S}_{n,n+1}(G)$ denote the family of all
a.i. subsets of $G$ which are dual to splittings of $G$ over annuli or tori in
$(G,\partial G)$. Then the regular neighbourhood $\Sigma_{n,n+1}(G)$ of
$\mathcal{S}_{n,n+1}(G)$ in $G$ exists and is obtained from $\Gamma
_{n,n+1}(G)$ by splitting each exceptional $V_{0}$--vertex along the
exceptional annulus it contains, as described above.
\end{theorem}

\begin{remark}
It follows from this theorem that if $\partial G$ is empty, so that $G$ is an
orientable $PD(n+2)$ group, then the regular neighbourhood $\Sigma_{n,n+1}(G)$
of $\mathcal{S}_{n,n+1}(G)$ in $G$ exists and is equal to $\Gamma
_{n,n+1}(G)=\Gamma_{n+1}(G)$.
\end{remark}

In Figures \ref{fig1}, \ref{fignew3} and \ref{fignew4}, we drew only the
essential isolated arcs. In Figure \ref{fignew5} we draw some examples with
other possible essential arcs to illustrate that they do not lead to isolated
arcs. We point out the corresponding figures from the text, but omit the labels.

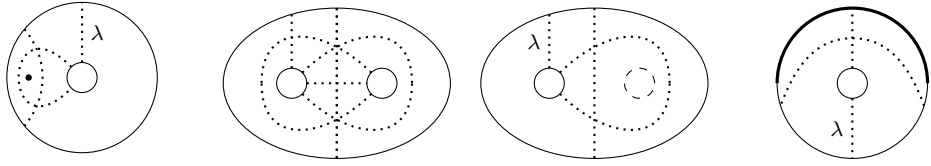
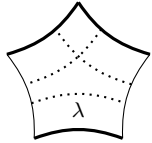
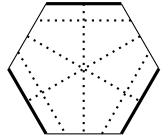
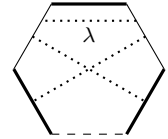
\begin{figure}[ptb]
\centering
\begin{subfigure}[b]{0.23\textwidth}
\centering
\begin{tikzpicture}
\newcommand{\outrad}{1}
\newcommand{\inrad}{0.2}
\draw (0,0) circle (\inrad);
\draw (0,0) circle (\outrad);
\draw[fill] (-0.7\outrad,0) circle (1pt);
\draw [thick, dotted] (90:\inrad) -- node [right=-1pt] {\scriptsize$\lambda$} (90:\outrad);
\draw [thick, dotted] (140:\outrad) to [out=-50, in=50] (-140:\outrad);
\draw [thick, dotted] (135:\inrad) to [out=135, in=-135, looseness=12](-135:\inrad);
\end{tikzpicture}
\caption{For Figure \ref{fignew3}b)}
\end{subfigure}
\hfil
\begin{subfigure}[b]{0.23\textwidth}
\centering
\begin{tikzpicture}
\newcommand{\outrad}{1}
\newcommand{\inrad}{0.2}
\draw [name path=ellipse] (0,0) circle [x radius=1.5\outrad, y radius=\outrad];
\draw (-0.6,0) circle (\inrad); 
\draw (0.6,0) circle (\inrad); 
\draw [thick, dotted] (90:\outrad) -- (-90:\outrad);
\draw [thick, dotted] (90:\outrad) -- (-90:\outrad);
\path [name path=con] (-0.6,0) -- (-0.6,\outrad); 
\draw [thick, dotted,  name intersections={of=ellipse and con}] (-0.6,\inrad) --  (intersection-1); 
\draw  [thick, dotted] (-0.6,0) + (45:\inrad) to [out=45,in=90, looseness=1.5] (1,0); 
\draw  [thick, dotted] (-0.6,0) + (-45:\inrad) to [out=-45,in=-90, looseness=1.5] (1,0); 
\draw  [thick, dotted, rotate=180] (-0.6,0) + (45:\inrad) to [out=45,in=90, looseness=1.5] (1,0); %
\draw  [thick, dotted, rotate=180] (-0.6,0) + (-45:\inrad) to [out=-45,in=-90, looseness=1.5] (1,0); %
\draw  [thick, dotted] (-0.6,0) + (0:\inrad) -- ($(0.6,0) + (180:\inrad)$);
\end{tikzpicture}
\caption{For Figure \ref{fignew2}c)}
\end{subfigure}
\hfil
\begin{subfigure}[b]{0.23\textwidth}
\centering
\begin{tikzpicture}
\newcommand{\outrad}{1}
\newcommand{\inrad}{0.2}
\draw [name path=ellipse] (0,0) circle [x radius=1.5\outrad, y radius=\outrad];
\draw (-0.6,0) circle (\inrad); 
\draw [dashed] (0.6,0) circle (\inrad); 
\draw [thick, dotted] (90:\outrad) -- (-90:\outrad);
\path [name path=con] (-0.6,0) -- (-0.6,\outrad); 
\draw [thick, dotted,  name intersections={of=ellipse and con}] (-0.6,\inrad) --  node [left=-1pt] {\scriptsize$\lambda$} (intersection-1); 
\draw  [thick, dotted] (-0.6,0) + (45:\inrad) to [out=45,in=90, looseness=1.5] (1,0); 
\draw  [thick, dotted] (-0.6,0) + (-45:\inrad) to [out=-45,in=-90, looseness=1.5] (1,0); 
\end{tikzpicture}
\caption{For Figure \ref{fignew4}b)}
\end{subfigure}
\hfil
\begin{subfigure}[b]{0.23\textwidth}
\centering
\begin{tikzpicture}
\newcommand{\outrad}{1}
\newcommand{\inrad}{0.2}
\draw [very thick] (\outrad,0) arc [start angle=0, end angle=180, radius=\outrad];
\draw [name path=lower] (\outrad,0) arc [start angle=0, end angle=-180, radius=\outrad]; 
\draw [] (0,0) circle (\inrad); 
\path [name path=con] (-1,-0.3) -- (1,-0.3); 
\draw[thick, dotted,  name intersections={of=lower and con}] (intersection-1) to [out=110, in=0]  (0,0.6) to [out=180, in=70] (intersection-2); 
\draw[thick, dotted] (0,\inrad) -- (0,\outrad);
\draw[thick, dotted] (0,-\inrad) -- node [left=-1pt] {\scriptsize$\lambda$} (0,-\outrad);
\end{tikzpicture}
\caption{For Figure \ref{fignew3}d)}
\end{subfigure}
\\[4mm]
\hfil
\begin{subfigure}[b]{0.23\textwidth}
\centering
\begin{tikzpicture}[rotate=-72] 
\newcommand{\outrad}{1}
\newcommand{\inrad}{0.4}
\begin{scope}[rotate=0]  
\draw [name path=bottom, rotate=0] (-126:\outrad) to [out=20, in=160]  (-54:\outrad);
\end{scope}
\begin{scope}[rotate=72]  
\draw [name path=bottom, rotate=0, very thick] (-126:\outrad) to [out=20, in=160]  (-54:\outrad);
\path [rotate=-72, name path=left] (-126:\outrad) to [out=20, in=160] (-54:\outrad);
\path [rotate=72, name path=right] (-126:\outrad) to [out=20, in=160] (-54:\outrad);
\path [name path=con, red] (-160:\outrad)  -- (-20:\outrad); 
\path [red, name intersections={of=left and con, by={A}}] (A) circle (2pt);
\path [red, name intersections={of=right and con, by={B}}] (B) circle (2pt);
\draw [thick, dotted] (A) to [out=20, in=160] node [below=-1pt] {\scriptsize$\lambda$}  (B); 
\end{scope}
\begin{scope}[rotate=-72]  
\draw [name path=bottom, rotate=0, very thick] (-126:\outrad) to [out=20, in=160]  (-54:\outrad);
\path [rotate=-72, name path=left] (-126:\outrad) to [out=20, in=160] (-54:\outrad);
\path [rotate=72, name path=right] (-126:\outrad) to [out=20, in=160] (-54:\outrad);
\path [name path=con, red] (-160:\outrad)  -- (-20:\outrad); 
\path [red, name intersections={of=left and con, by={A}}] (A) circle (2pt);
\path [red, name intersections={of=right and con, by={B}}] (B) circle (2pt);
\draw [thick, dotted] (A) to [out=20, in=160]  (B); 
\end{scope}
\begin{scope}[rotate=-144]  
\draw [name path=bottom, rotate=0, very thick] (-126:\outrad) to [out=20, in=160]  (-54:\outrad);
\path [rotate=-72, name path=left] (-126:\outrad) to [out=20, in=160] (-54:\outrad);
\path [rotate=72, name path=right] (-126:\outrad) to [out=20, in=160] (-54:\outrad);
\path [name path=con, red] (-160:\outrad)  -- (-20:\outrad); 
\path [red, name intersections={of=left and con, by={A}}] (A) circle (2pt);
\path [red, name intersections={of=right and con, by={B}}] (B) circle (2pt);
\draw [thick, dotted] (A) to [out=20, in=160]  (B); 
\end{scope}
\begin{scope}[rotate=144]  
\draw [name path=bottom, rotate=0] (-126:\outrad) to [out=20, in=160]  (-54:\outrad);
\end{scope}
\end{tikzpicture}
\caption{For Figure \ref{fignew2}g)}
\end{subfigure}
\hfil
\begin{subfigure}[b]{0.23\textwidth}
\centering
\begin{tikzpicture} 
\newcommand{\outrad}{1}
\newcommand{\inrad}{0.4}
\begin{scope}[rotate=60]  
\draw [name path=bottom, rotate=0, very thick] (-120:\outrad) to   (-60:\outrad);
\path [rotate=-60, name path=left] (-120:\outrad) to  (-60:\outrad);
\path [rotate=60, name path=right] (-120:\outrad) to  (-60:\outrad);
\path [name path=con, red] (-140:\outrad)  -- (-40:\outrad); 
\path [red, name intersections={of=left and con, by={A}}] (A) circle (2pt);
\path [red, name intersections={of=right and con, by={B}}] (B) circle (2pt);
\draw [thick, dotted] (A) to   (B); 
\end{scope}
\begin{scope}[rotate=-60]  
\draw [name path=bottom, rotate=0, very thick] (-120:\outrad) to   (-60:\outrad);
\path [rotate=-60, name path=left] (-120:\outrad) to  (-60:\outrad);
\path [rotate=60, name path=right] (-120:\outrad) to  (-60:\outrad);
\path [name path=con, red] (-140:\outrad)  -- (-40:\outrad); 
\path [red, name intersections={of=left and con, by={A}}] (A) circle (2pt);
\path [red, name intersections={of=right and con, by={B}}] (B) circle (2pt);
\draw [thick, dotted] (A) to   (B); 
\end{scope}
\begin{scope}[rotate=180]  
\draw [name path=bottom, rotate=0, very thick] (-120:\outrad) to   (-60:\outrad);
\path [rotate=-60, name path=left] (-120:\outrad) to  (-60:\outrad);
\path [rotate=60, name path=right] (-120:\outrad) to  (-60:\outrad);
\path [name path=con, red] (-140:\outrad)  -- (-40:\outrad); 
\path [red, name intersections={of=left and con, by={A}}] (A) circle (2pt);
\path [red, name intersections={of=right and con, by={B}}] (B) circle (2pt);
\draw [thick, dotted] (A) to   (B); 
\end{scope}
\begin{scope}[rotate=0]  
\draw [name path=bottom, rotate=0] (-120:\outrad) to   (-60:\outrad);
\end{scope}
\begin{scope}[rotate=120]  
\draw [name path=bottom, rotate=0] (-120:\outrad) to   (-60:\outrad);
\end{scope}
\begin{scope}[rotate=-120]  
\draw [name path=bottom, rotate=0] (-120:\outrad) to   (-60:\outrad);
\end{scope}
\draw [thick, dotted] (90:0.866\outrad) -- (-90:0.866\outrad);
\draw [thick, dotted, rotate=60] (90:0.866\outrad) -- (-90:0.866\outrad);
\draw [thick, dotted, rotate=-60] (90:0.866\outrad) -- (-90:0.866\outrad);
\end{tikzpicture}
\caption{For Figure \ref{fignew2}h)}
\end{subfigure}
\hfil
\begin{subfigure}[b]{0.23\textwidth}
\centering
\begin{tikzpicture} 
\newcommand{\outrad}{1}
\newcommand{\inrad}{0.4}
\begin{scope}[rotate=60]  
\draw [name path=bottom, rotate=0, very thick] (-120:\outrad) to   (-60:\outrad);
\end{scope}
\begin{scope}[rotate=-60]  
\draw [name path=bottom, rotate=0, very thick] (-120:\outrad) to   (-60:\outrad);
\end{scope}
\begin{scope}[rotate=180]  
\draw [name path=bottom, rotate=0, very thick] (-120:\outrad) to   (-60:\outrad);
\path [rotate=-60, name path=left] (-120:\outrad) to  (-60:\outrad);
\path [rotate=60, name path=right] (-120:\outrad) to  (-60:\outrad);
\path [name path=con, red] (-140:\outrad)  -- (-40:\outrad); 
\path [red, name intersections={of=left and con, by={A}}] (A) circle (2pt);
\path [red, name intersections={of=right and con, by={B}}] (B) circle (2pt);
\draw [thick, dotted] (A) to node [below=-2pt] {\scriptsize $\lambda $}  (B); 
\end{scope}
\begin{scope}[rotate=0]  
\draw [name path=bottom, rotate=0,dashed] (-120:\outrad) to   (-60:\outrad);
\end{scope}
\begin{scope}[rotate=120]  
\draw [name path=bottom, rotate=0] (-120:\outrad) to   (-60:\outrad);
\end{scope}
\begin{scope}[rotate=-120]  
\draw [name path=bottom, rotate=0] (-120:\outrad) to   (-60:\outrad);
\end{scope}
%
\draw [thick, dotted, rotate=60] (90:0.866\outrad) -- (-90:0.866\outrad);
\draw [thick, dotted, rotate=-60] (90:0.866\outrad) -- (-90:0.866\outrad);
\end{tikzpicture}
\caption{For Figure \ref{fignew4}g)}
\end{subfigure}
\caption{Note that $\partial_{1}=\emptyset$ in all cases except (c) and (g)}%
\label{fignew5}%
\end{figure}

\section{Results in dimension $3\label{section:dimension3}$}

In this section, we consider the special case of $PD3$ pairs and compare our
results in this case with the results of Neumann and Swarup in \cite{NS}. In
the previous section, a key role was played by the classification of compact
$2$--orbifolds with certain properties. In the case when $n=1$, so that we are
considering $PD3$ pairs, the following result greatly reduces the number of
cases which need considering.

\begin{lemma}
\label{nomirrors}Let $(G,\partial G)$ be an orientable $PD3$ pair, let $v$ be
a $V_{0}$--vertex of $\Gamma_{1,2}(G)$ which is of Seifert type or of
commensuriser type, and let $X$ denote the base $2$--orbifold of $v$. Then $X$
has no mirrors.
\end{lemma}

\begin{remark}
Note that we are allowing $G$ to be $VPC$ in the above statement. This result
means that when $n=1$, only Figures \ref{fig1}a), \ref{fig1}d)-f),
\ref{fignew3}a),b) and \ref{fignew4}a)-c) are relevant to the results in this
section. It also means that each $V_{0}$--vertex of $\Gamma_{1,2}(G)$ can be
regarded as a Seifert fibre space or an $I$--bundle.
\end{remark}

\begin{proof}
Recall that $G(v)$ is $VPC1$--by--$\pi_{1}^{orb}(X)$, and that $G$ is torsion
free. Now a $VPC1$ group is a finite extension of $\mathbb{Z}$, so a torsion
free $VPC1$ group is $PD1$ and must also be isomorphic to $\mathbb{Z}$. And a
$VPC2$ group is a finite extension of $\mathbb{Z}\times\mathbb{Z}$, so a
torsion free $VPC2$ group is $PD2$ and must be isomorphic to $\mathbb{Z}%
\times\mathbb{Z}$ or to $\pi_{1}(K)$, where $K$ denotes the Klein bottle. In
particular, a torus in a $PD3$ pair must be isomorphic to $\mathbb{Z}%
\times\mathbb{Z}$.

If $X$ contains a mirror, there are three possibilities. The mirror must be a
circle, or meet $\partial X$, or meet another mirror (or itself) in a corner
reflector point. We will show that each of these cases is impossible, which
implies that $X$ has no mirrors, as required.

First we consider a component $C$ of $\partial X$, which must be a circle or
the quotient $Q$ of a circle by a reflection. We know that $\pi_{1}^{orb}(C)$
is the image of a torus in $\partial M_{v}$. As $\pi_{1}^{orb}(Q)\cong%
\mathbb{Z}_{2}\ast\mathbb{Z}_{2}$, which is not abelian, $\pi_{1}(Q)$ cannot
be a quotient of the abelian group $\mathbb{Z}\times\mathbb{Z}$. It follows
that all components of $\partial X$ are circles. Hence no mirror of $X$ can
meet $\partial X$, as required.

Next suppose that $X$ has a corner reflector. This yields a finite dihedral
subgroup $D$ of $\pi_{1}^{orb}(X)$, where the term dihedral group includes the
group $\mathbb{Z}_{2}\times\mathbb{Z}_{2}$. In particular $D$ is not cyclic.
But the pre-image of $D$ in $G(v)$ is a torsion free $VPC1$ group and so is
isomorphic to $\mathbb{Z}$, which implies that $D$ must be cyclic. This
contradiction show that $X$ cannot have corner reflectors, as required.

Finally suppose that $X$ has a mirror $m$ which is a circle. Then $m$ has a
neighbourhood orbifold $Y$ in $X$ with underlying space an annulus, such that
$\partial Y$ is equal to one boundary component $C$ of the annulus and the
other boundary component is $m$. We have $\pi_{1}^{orb}(C)=\pi_{1}%
(C)\cong\mathbb{Z}$, and $\pi_{1}^{orb}(m)=\pi_{1}(C)\times\mathbb{Z}_{2}$,
and we let $R$ and $S$ denote the pre-images in $G(v)$ of $\pi_{1}(C)$ and
$\pi_{1}^{orb}(m)$ respectively. Thus $R$ is a subgroup of $S$ of index $2$.
Each of $R$ and $S$ is a torsion free $VPC2$ group. As $C$ determines a
splitting of $\pi_{1}^{orb}(X)$ over $\pi_{1}(C)$ which is adapted to
$\partial X$, this yields a splitting of $G(v)$ over $R$ which is adapted to
$\partial_{1}v$, and hence determines a splitting of $G$ over $R$ which is
adapted to $\partial G$. As the pair $(G,\partial G)$ is orientable, it
follows that $R$ is orientable, and so is a torus in $G$. Also the splitting
of $G\ $over $R$ yields one or two $PD3$ pairs with $R$ as a boundary group.
Now Lemma 2.2 of \cite{KR2} implies that $R$ is maximal among torus subgroups
of $G$, so that $S$ cannot be a torus. Thus $S$ is isomorphic to $\pi_{1}(K)$.
Now consider the presentation $<a,b:bab^{-1}=a^{-1}>$ of $\pi_{1}(K)$. The
kernel of the map $S\cong\pi_{1}(K)\rightarrow\pi_{1}^{orb}(m)\cong%
\mathbb{Z}\times\mathbb{Z}_{2}$ must be the subgroup $A$ generated by $a^{2}$,
and this must also be the kernel of the map $R\rightarrow\pi_{1}%
(C)\cong\mathbb{Z}$. As $R$ has index $2$ in $\pi_{1}(K)$, it must be the
orientation subgroup generated by $a$ and $b^{2}$. But then the quotient of
$R$ by $A$ is isomorphic to $\mathbb{Z}\times\mathbb{Z}_{2}$, which is a contradiction.

This completes the proof that $X$ has no mirrors.
\end{proof}

Now we can compare our results from section \ref{section:comparisons}, with
those of Neumann and Swarup in \cite{NS}. Recall that in section
\ref{section:comparisons}, we considered an orientable $PD(n+2)$ pair
$(G,\partial G)$ such that $G$ is not $VPC$ and described the possible
exceptional annuli. These are splittings of $G$ dual to an annulus which are
not edge splittings of $\Gamma_{n,n+1}^{c}(G)$. Each exceptional annulus is
enclosed by some $V_{0}$--vertex $v$ of $\Gamma_{n,n+1}^{c}(G)$ of
commensuriser type, and corresponds to an isolated arc $\lambda$ in the base
$2$--orbifold $X_{v}$ of $v$. In Figures \ref{fig1}a)-c), \ref{fignew3} and
\ref{fignew4} we showed all possible such arcs and orbifolds. We will be
interested in the special case when $n=1$, so that $(G,\partial G)$ is an
orientable $PD3$ pair. Lemma \ref{nomirrors} tells us that those figures in
which the orbifold $X$ has a mirror, are not relevant in this case. This
substantially reduces the number of possibilities. We only need to consider
the six isolated arcs shown in Figures \ref{fig1}a), \ref{fignew3}a),b) and
\ref{fignew4}a)-c).

Recall from the beginning of section \ref{section:prelim} that an annulus in a
$PD(n+2)$ pair is a certain type of orientable $PD(n+1)$ pair, whose
fundamental group is $VPCn$. When $n=1$, a torsion free $VPC1$ group must be
isomorphic to $\mathbb{Z}$, and this is an orientable $PD1$ group. Thus
twisted annuli do not appear when considering $PD3$ pairs, and an untwisted
annulus in our generalized sense is exactly the same as the ordinary annulus
$S^{1}\times I$. Further if our $PD3$ pair $(G,\partial G)$ comes from a
compact orientable $3$--manifold $M$, then there is a precise correspondence
between $\Gamma_{1,2}^{c}(G)$ and the JSJ decomposition of $M$. Also any
exceptional annuli in $(G,\partial G)$ correspond to embedded annuli in $M$
which cross no other embedded essential annulus in $M$ and are not splitting
annuli of the JSJ decomposition of $M$. In \cite{NS}, such annuli are called
matched annuli, and the possibilities are listed in Lemma 3.4 of \cite{NS}. We
would expect this list to be the same as our list of six possible isolated
arcs in Figures \ref{fig1}a), \ref{fignew3}a),b) and \ref{fignew4}a)-c), but
there are some differences. The four isolated arcs shown in Figures
\ref{fig1}a), \ref{fignew3}a), \ref{fignew4}a) and \ref{fignew4}c) yield the
examples of matched annuli shown in Figure 5 of \cite{NS}, but the two
isolated arcs shown in \ref{fignew3}b) and \ref{fignew4}b) do not correspond
to matched annuli shown in Figure 5 of \cite{NS}. In \ref{fignew3}b),
$\partial_{1}X$ is empty which implies that $G=G(v)$, and that $M$ is a
Seifert fibre space, so this case is not of much interest. But in Figure
\ref{fignew4}b), $\partial_{1}X$ is non-empty, so the isolated arc $\lambda$
in this figure corresponds to an interesting matched annulus in $M$. This
seems to be an omission in \cite{NS}. Figure \ref{fignew5}c) shows that the
orbifold $X$ in Figure \ref{fignew4}b) contains two essential arcs other than
$\lambda$, but they cross, so neither is isolated. Cutting along the vertical
one of the two crossing arcs in Figure \ref{fignew5}c) expresses $X$ as the
union of two orbifolds glued along a boundary arc. These are the first and
second orbifolds shown in Figure 1 of \cite{NS}. This is the unique case where
gluing two of the orbifolds shown in Figure 1 of \cite{NS} yields an orbifold
with no isolated essential arc. This possibility was omitted in the discussion
in the second paragraph on page 35 of \cite{NS}. Specifically the last
sentence of that paragraph is incorrect.

The result of Lemma \ref{nomirrors} fails in higher dimensions. Mirrors of all
three types discussed in the proof of Lemma \ref{nomirrors} can exist in all
dimensions greater than $3$. We discuss some examples in dimension $4$. Again
higher dimensional examples can be obtained by taking the product with circles.

Our starting point is that the orientable $3$--manifold $W$ which is a twisted
$I$--bundle over the Klein bottle $K$ is an example of a twisted
$3$--dimensional annulus, and the double $DW$ of $W$ is a $3$--dimensional
torus. Thus the orientable $4$--manifold $DW\times I$ is the underlying space
of a $PD4$ pair $(G,\partial G)$, and $\Gamma_{2,3}(G)$ consists of a single
$V_{0}$--vertex $v$, so $G=G(v)$, and $v$ is of $VPC2$--by--Fuchsian type with
base orbifold $Q\times I$, where $Q$ is the quotient of the circle by a
reflection. This orbifold has two mirrors each meeting the orbifold boundary
in reflector points. If instead one considers the manifold $DW\times S^{1}$,
one will have two mirrors each homeomorphic to a circle.

Finally, one can also give examples with corner reflectors as follows. A
useful way to think about $W$ is as the $I$--bundle over $K$ associated to the
$\partial I=S^{0}$--bundle given by the double covering map $T\rightarrow K$.
Note that this map is determined by a surjective homomorphism $\pi
_{1}(K)\rightarrow\mathbb{Z}_{2}$. One way to construct $W$ is to start with
the product $T\times I$ of the $2$--torus with the unit interval, and consider
the involution $(\tau,\sigma)$ on $T\times I$, where $\tau$ is the free
involution of $T\ $associated to the double covering map $T\rightarrow K$, and
$\sigma$ is the reflection of $I$. As $\tau$ is free, so is $(\tau,\sigma)$,
and the quotient of $T\times I$ by $(\tau,\sigma)$ is clearly $W$.

We will perform a similar construction starting with the product $T\times
I\times I$, and using the natural homomorphism $\varphi:\pi_{1}(K)\rightarrow
H_{1}(K;\mathbb{Z}_{2})\cong\mathbb{Z}_{2}\times\mathbb{Z}_{2}$. Only one of
the three surjections $\pi_{1}(K)\rightarrow\mathbb{Z}_{2}$ yields an
orientable double cover, and we will choose the basis of $H_{1}(K;\mathbb{Z}%
_{2})$ so that projection onto each factor yields non-orientable double covers
$K^{\prime}$ and $K^{\prime\prime}$. Let $T$ denote the torus which is the
$4$--fold cover of $K$ corresponding to the kernel of $\varphi$. Thus
$\mathbb{Z}_{2}\times\mathbb{Z}_{2}$ acts freely on $T$ with quotient $K$. It
also acts on $I\times I$ as the group generated by reflections in each factor,
and we let $X$ denote the quotient $2$--orbifold of this action. The
underlying space of $X$ is a disc $D$, whose boundary is divided into two
mirrors and an arc of $\partial X$. The product action on $T\times I\times I$
is free and orientation preserving, so that the quotient of $T\times I\times
I$ by this action is an aspherical orientable $4$--manifold $Z$, and $Z$ has a
natural projection to $X$. The pre-image in $Z$ of each interior point of $X$
is $T$. The pre-image of the corner reflector of $X$ is $K$, and the pre-image
of all other points of one mirror is $K^{\prime}$ and of the other mirror is
$K^{\prime\prime}$. Finally the pre-image of all other points of $\partial X$
is $T$. Further the pre-image of $\partial X$, which is equal to $\partial Z$,
consists of the union of the twisted $I$--bundle over $K^{\prime}$ with
boundary $T$ and the twisted $I$--bundle over $K^{\prime\prime}$ with boundary
$T$, glued along $T$. The pre-image of one mirror is a twisted $I$--bundle
over $K$ with boundary $K^{\prime}$, and the pre-image of the other mirror is
a twisted $I$--bundle over $K$ with boundary $K^{\prime\prime}$. Thus
$(Z,\partial Z)$ is the underlying space of a $PD4$ pair $(G,\partial G)$, and
$\Gamma_{2,3}(G)$ consists of a single $V_{0}$--vertex $v$, so $G=G(v)$, and
$v$ is of $VPC2$--by--Fuchsian type with base orbifold $X$.

\section{Some related questions\label{section:relatedquestions}}

An unsatisfactory part of our work is that there is no algebraic treatment of
the triple $(G(v),\partial_{0}v,\partial_{1}v)$ we discussed.

\begin{problem}
Construct a theory of Poincar\'e triads for groups.
\end{problem}

There is a discussion by Wall in the case of complexes \cite{Wall1}.

In Johannson's Deformation Theorem, he considers a homotopy equivalence
$F:M\rightarrow M^{\prime}$ between two Haken $3$--manifolds with
incompressible boundary. He shows that there is a bijection between the pieces
of the JSJ\ decomposition of $M$ and those of $M^{\prime}$, and that $F$ can
be homotoped to send the pieces of $M$ to the pieces of $M^{\prime}$. In
particular, the splitting annuli and tori of $M$ are sent to splitting annuli
and tori of $M^{\prime}$. For the non-characteristic pieces of $M$, he shows
one can further homotop $F$ to arrange that the intersection with the boundary
of $M$ is mapped to the boundary of the corresponding piece of $M^{\prime}$.
Finally one can arrange that the restriction of $F$ to each non-characteristic
piece is a homeomorphism to the corresponding piece of $M^{\prime}$. It is
natural to ask whether there is an algebraic analogue of this. The natural
analogue would be when one considers two $PD(n+2)$ pairs $(G,\partial G)$ and
$(G^{\prime},\partial G^{\prime})$ with an isomorphism between $G\ $and
$G^{\prime}$. There is a bijection between the underlying graphs of
$\Gamma_{n,n+1}^{c}(G)$ and $\Gamma_{n,n+1}^{c}(G^{\prime})$, and one would
like to know that for a $V_{1}$--vertex $v$ of $\Gamma_{n,n+1}^{c}(G)$, the
part $\partial_{0}v$ of $\partial v$ coming from $\partial G$ can be deformed
into $\partial_{0}v^{\prime}$ of the corresponding $V_{1}$--vertex of
$\Gamma_{n,n+1}^{c}(G^{\prime})$. It seems reasonable this should hold when
$n=1$, but this seems far from clear when $n>1$. The reason is that the proof
of Johannson's Deformation Theorem depends on the non-existence of certain
types of essential annulus in the non-characteristic pieces of $M$. In higher
dimensions the analogous fact would be the non-existence of essential higher
dimensional annuli, but that may not exclude the existence of essential maps
of the $2$--dimensional annulus.

\begin{problem}
In the case of orientable $PD3$ pairs $(G,\partial G)$ and $(G^{\prime
},\partial^{\prime}G)$, with $G$ and $G^{\prime}$ isomorphic, for any $V_{1}%
$--vertex $v$ of $\Gamma_{n,n+1}^{c}(G)$, and the corresponding $V_{1}%
$--vertex $v^{\prime}$ of $\Gamma_{n,n+1}^{c}(G^{\prime})$, show that
$\partial_{0}v$ can be deformed into $\partial_{0}v^{\prime}$.
\end{problem}

In Theorem 8.1 of \cite{BE01}, Bieri and Eckmann proved a result which we have
used several times. Namely that if a $PDn$ pair is split along a $PD(n-1)$
subgroup relative to the boundary, then we again get $PDn$ pairs.

\begin{problem}
Is there an analogue of the Bieri-Eckmann Theorem when a $PDn$ pair is split
along a $PD(n-1)$ pair?
\end{problem}

Examples in dimension $3$ show that if one splits a $3$--manifold with
incompressible boundary along a surface with non-empty boundary, the resulting
manifold may have compressible boundary. Thus if one splits a $PD3$ pair along
a $PD2$ pair, the resulting object need not be a $PD3$ pair. This again seems
to need a theory of $PD$ triples for groups. However, Gitik \cite{Gitik} has
proven an analogue of the Bieri-Eckmann Theorem in the special case when
splitting a $PDn$ pair along a $PD(n-1)$ pair does yield a $PDn$ pair.

A related natural question is:

\begin{problem}
Is there a theory of $PD$ pairs when the maps from the boundary groups are not injective?
\end{problem}

\end{document}